\documentclass[10pt]{article}
\usepackage[utf8]{inputenc}

\usepackage[left=1.4in,top=1.3in,right=1.3in,bottom=1.3in,footskip = 0.333in]{geometry}

\usepackage{authblk, amsmath, amsthm, amsfonts, amssymb, amsbsy, bbm, bigstrut, graphicx, enumerate,  upref, longtable, comment, comment,xcolor,siunitx,mathrsfs}
\usepackage{float}
\usepackage{tabularx, booktabs, makecell, caption}
\usepackage{siunitx}
\usepackage{algorithm}

\usepackage{subcaption}
\usepackage[breaklinks]{hyperref}

\usepackage[numbers, sort&compress]{natbib} 
\usepackage{comment}
\usepackage[export]{adjustbox}

\newtheorem{theorem}{Theorem}
\newtheorem{lemma}{Lemma}
\newtheorem{corollary}{Corollary}
\newtheorem{conjecture}{Conjecture}

\newtheorem{defn}{Definition}

\theoremstyle{definition}
\newtheorem{remark}{Remark}

\newtheorem{ass}{Assumption}
\newcommand{\TT}{\mathcal{T}}

\numberwithin{theorem}{section}
\numberwithin{remark}{section}
\numberwithin{ex}{section}
\numberwithin{ass}{section}
\numberwithin{defn}{section}
\numberwithin{lemma}{section}
\numberwithin{corollary}{section}
\numberwithin{prop}{section}
\numberwithin{conjecture}{section}

\newcommand{\argmax}{\operatorname{argmax}}
\newcommand{\argmin}{\operatorname{argmin}}

\newcommand{\bx}{\mathbf{x}}
\newcommand{\bA}{\mathbf{A}}
\newcommand{\by}{\mathbf{y}}

{\tiny {\tiny }}

\newcommand{\ee}{\mathbb{E}}

\newcommand{\rr}{\mathbb{R}}

\allowdisplaybreaks

\newtheorem{definition}{Definition}

\newcommand{\bzero}{ {\bf 0} }
\newcommand{\bG}{ { G} }
\newcommand{\bh}{ \mathbf{h} }

\newcommand{\bm}{ \mathbf{m} }

\newcommand{\bI}{ { I} }
\newcommand{\bw}{\mathbf{w}}

\newcommand{\bB}{ \mathbf{B} }

\def \bz{\mathbf{z}}
\def \bX{\mathbf{X}}
\def \bY{\mathbf{Y}}
\def \bu{\mathbf{u}}
\def \bv{\mathbf{v}}

\def \bt{\mathbf{t}}

\def \bE{\mathbb{E}}
\def \bT{\mathbf{T}}

\def \btheta {\boldsymbol{\theta}}
\newcommand{\bP}{\mathbb{P}}

\newcommand{\independent}{\perp\!\!\!\!\perp}

\makeatletter
\def\thickhline{%
	\noalign{\ifnum0=`}\fi\hrule \@height \thickarrayrulewidth \futurelet
	\reserved@a\@xthickhline}
\def\@xthickhline{\ifx\reserved@a\thickhline
	\vskip\doublerulesep
	\vskip-\thickarrayrulewidth
	\fi
	\ifnum0=`{\fi}}
\makeatother

\newlength{\thickarrayrulewidth}
\usepackage{algorithm}
\numberwithin{equation}{section}

\DeclareMathOperator{\sech}{sech}

\begin{document}
\title{Causal inference under interference: computational barriers and algorithmic solutions}
\date{}

\author[1]{Sohom Bhattacharya\thanks{bhattacharya.s@ufl.edu}}
\author[2]{Subhabrata Sen \thanks{subhabratasen@fas.harvard.edu}}
\affil[1]{Department of Statistics, University of Florida}
\affil[2]{Department of Statistics, Harvard University}
\makeatletter
\renewcommand\AB@affilsepx{\\ \protect\Affilfont}
\renewcommand\Authsep{, }
\renewcommand\Authand{, }
\renewcommand\Authands{, }
\makeatother

\maketitle

\begin{abstract}
We study causal effect estimation under interference from network data. We work under the chain-graph formulation pioneered in \cite{tchetgen2021auto}. Our first result shows that polynomial time evaluation of treatment effects is computationally hard in this framework without additional assumptions on the underlying chain graph. Subsequently, we assume that the interactions among the study units are governed either by (i) a dense graph or (ii) an i.i.d. Gaussian matrix. In each case, we show that the treatment effects have well-defined limits as the population size diverges to infinity. Additionally, we develop polynomial time algorithms to consistently evaluate the treatment effects in each case. Finally, we estimate the unknown parameters from the observed data using maximum pseudo-likelihood estimates, and establish the stability of our causal effect estimators under this perturbation. Our algorithms provably approximate the causal effects in polynomial time even in low-temperature regimes where the canonical MCMC samplers are slow mixing. For dense graphs, our results use the notion of regularity partitions; for Gaussian interactions, our approach uses ideas from spin glass theory and Approximate Message Passing.

\end{abstract}

\section{Introduction}
The learning of causal effects from observational data is critical in modern data science. Traditional methods for causal inference are developed under the \emph{no interference} assumption. This assumption is often violated in diverse modern applications e.g. social networks~\cite{ogburn2014causal}, epidemiology~\cite{reich2021review}, public policy~\cite{matthay2022causal} etc.

The inference of causal relations under interference has received significant attention in the recent literature. One prominent approach in this context, pioneered by~\cite{liu2025auto,tchetgen2021auto} is based on the chain graph formalism of~\cite{lauritzen2002chain}. Although this formalism provides an elegant framework to study causal inference under interference, the evaluation of causal effects within this framework presents several algorithmic challenges, which are currently unresolved. In this article, we focus on the following questions: (i) When is \emph{computationally efficient} evaluation of causal effects possible under the chain-graph framework? (ii) What are the appropriate algorithms to estimate causal effects within this setup?

Formally,  we work under the Neyman-Rubin potential outcomes framework with binary treatments. Let $n$ denote the number of study units. Denote any treatment assignment as $\mathbf{t} \in \{\pm 1\}^n$. We denote the potential outcomes as $\{ \mathbf{Y}_i ( \mathbf{t}) : \mathbf{t} \in \{\pm 1\}^n\}$. Next, we introduce the causal estimands of interest. Specifically, we study the \emph{direct} and the \emph{indirect/spillover} effect of the assigned treatments on the outcomes. 
To this end, we first introduce the average direct causal effect for unit $i$ upon
changing the unit’s treatment status from $t_i=-1$ to $t_i=1$:
\begin{equation}\label{eq:define_de_i}
    \text{DE}_{i}(\mathbf{t}_{-i}) := \mathbb{E}[\mathbf{Y}_i(1, \bt_{-i})] - \mathbb{E}[\mathbf{Y}_i (-1, \bt_{-i})]
\end{equation}
where $(c, \mathbf{t}_{-i})$, $c \in \{\pm 1\}$, denotes the binary vector where the $i^{th}$ entry is $c$ and the remaining entries are specified by $\mathbf{t}_{-i}$. 
Note that the direct effect $\mathrm{DE}_i(\cdot)$ is dependent on the treatment assignments of the other units $\mathbf{t}_i \in \{\pm 1\}^{n-1}$. To define an averaged direct effect, following \cite{hudgens2008toward,tchetgen2012causal,tchetgen2021auto}, we average these effects over a hypothetical allocation probability measure $\pi$ on $\{\pm 1\}^{n-1}$:
\begin{equation}\label{eq:define_de}
    \text{DE}(\pi)=\frac{1}{n}\sum_{i=1}^{n}\sum_{\mathbf{t}_{-i}\in \{\pm 1\}^{n-1}} \pi(\mathbf{t}_{-i}) \text{DE}_i(\mathbf{t}_{-i}).
\end{equation}
Note that under the no interference setting i.e. if $\bY_i(\bt)= \bY_i(t_i)$, the direct effect $\mathrm{DE}(\pi)$ reduces to the traditional average treatment effect. 
Next, we define the average indirect or spillover causal effect experienced by unit $i$ if the unit’s treatment is set to be inactive, while changing the treatment of other units from inactive to $\mathbf{t}_{-i}$ :
\begin{equation}\label{eq:define_ie_i}
    \text{IE}_{i}(\mathbf{t}_{-i}) := \mathbb{E}[\bY_i(-1, \bt_{-i})] - \mathbb{E}[\bY_i ( -\mathbf{1})]. 
\end{equation}
Similar to direct effect, we average over the allocation $\pi$ to obtain  
\begin{equation}\label{eq:define_ie}
    \text{IE}(\pi)=\frac{1}{n}\sum_{i=1}^{n}\sum_{\mathbf{t}_{-i}\in \{\pm 1\}^{n-1}} \pi(\mathbf{t}_{-i}) \text{IE}_i(\mathbf{t}_{-i}).
\end{equation}
Observe that under no interference, i.e. if $\bY_i(\bt)= \bY_i(t_i)$, the indirect effect $\mathrm{IE}(\pi)=0$. 
In our subsequent discussion, we assume that the allocation measure $\pi$ appearing in \eqref{eq:define_de} and \eqref{eq:define_ie} is, in fact, the uniform distribution on $\{\pm 1\}^{n-1}$ i.e.,  $\pi(\mathbf{t}_{-i})=2^{-(n-1)}$ for all $\mathbf{t}_{-i} \in \{\pm 1\}^{n-1}$. Our arguments extend in a straightforward manner to any i.i.d. measure on $\{\pm 1\}^{n-1}$ (See \cite[Remark 1.1]{bhattacharya2024causal}).   For notational simplicity, we suppress the dependence on $\pi$, and  write $\text{DE}$ and $\text{IE}$ in our subsequent discussion.

We observe data $\{(Y_i, T_i, \mathbf{X}_i): 1\leq i \leq n\}$, where $Y_i \in \mathbb{R}$ denotes the observed response, $T_i \in \{\pm 1\} $ represents the assigned treatment and $\mathbf{X}_i \in [-1,1]^d$ represents the observed covariates for the $i^{th}$ unit. Set $\mathbf{Y} = (Y_1, \cdots, Y_n) \in \mathbb{R}^n$, $\mathbf{T} = (T_1,\cdots,T_n) \in \{\pm 1\}^n$ and $\mathbf{X}^{\top} = (\mathbf{X}_1, \cdots, \mathbf{X}_n) \in \mathbb{R}^{d \times n}$. To estimate the causal effects, we need to relate the observed data with the potential outcomes. To this end, we will work under the following standard assumptions: 
\begin{itemize}
    \item[(i)] Consistency--- We assume that $\mathbf{Y}(\mathbf{T}) = \mathbf{Y}$ -- this is the network version of the traditional consistency condition. 
    \item[(ii)] No unmeasured confounding---For identifiability of the causal effect, we assume 
\begin{align*}
    \mathbf{T} \independent \mathbf{Y}(\mathbf{t}) | \mathbf{X} \,\, \mathrm{for\, all } \,\, \mathbf{t} \in \{\pm 1\}^n. 
\end{align*}
This reduces to the traditional no unmeasured confounding assumption in the absence of interference. 
\item[(iii)] Positivity---Finally, we assume $\mathbb{P}[\mathbf{T} = \mathbf{t} | \mathbf{X}]\geq \sigma_n >0$ for some $\sigma_n>0$. This is the appropriate analogue of the traditional positivity assumption in our setting. 
\end{itemize} 
Under these assumptions a network version of Robins's g-formula implies 
\begin{align}
\label{eq:g_computation1}
    \text{DE}_{i}(\mathbf{t}_{-i}) := \mathbb{E}_{\mathbf{X}} \big[\mathbb{E}[Y_i| \mathbf{T}= (1,\mathbf{t}_{-i}),\mathbf{X}]-\mathbb{E}[Y_i|\mathbf{T} = (-1,\mathbf{t}_{-i}),\mathbf{X}]\big],  
\end{align}
and thus $\text{DE}_{i}(\mathbf{t}_{-i})$ can be expressed as a function of the observed data law. Similarly, we have, 
\begin{align}
\label{eq:g_computation2}
    \text{IE}_{i}(\mathbf{t}_{-i}) := \mathbb{E}_\mathbf{X} \big[\mathbb{E}[\mathbf{Y}_i|\mathbf{T} = (-1, \mathbf{t}_{-i}), \mathbf{X}]-\mathbb{E}[\mathbf{Y}_i| \mathbf{T}= - \mathbf{1},\mathbf{X}]\big].  
\end{align}
Thus the indirect effect $\text{IE}_{i}(\mathbf{t}_{-i})$ can also be expressed as a functional of the observed data law. 
 By linearity, we obtain that the causal estimands $\mathrm{DE}$ and $\mathrm{IE}$ are functions of the observed data law. However, to ensure identifiability of these causal estimands, one needs additional structure on the observed data law (we refer the interested reader to 
\cite[Section 2.2]{tchetgen2021auto} for an in-depth discussion of this point).

Here we follow the Markov Random Field  (MRF) based framework introduced in \cite{tchetgen2021auto} and subsequently explored by \cite{bhattacharya2020causal,sherman2018identification}. Throughout, we assume that $\mathbf{Y} \in \{-1,1\}^n$---this reduces the notational overhead, and simplifies some key technical arguments in our analysis. Our techniques extend naturally to bounded $\mathbf{Y}$; we refer to the discussion in Section~\ref{sec:discussions} for additional details.  Given covariates $\mathbf{x}^{\top} = (\mathbf{x}_1, \cdots, \mathbf{x}_n)$, $\mathbf{x}_i \in [-1,1]^d$ and a treatment assignment $\mathbf{t} \in \{\pm 1\}^n$, the observed outcome $\bY \in \{\pm 1\}^n$ is given by the joint density 
 \begin{equation}\label{eq:define_gibbs}
    f(\mathbf{y}|\mathbf{t},\mathbf{x}) = \frac{1}{Z_n(\mathbf{t},\mathbf{x})} \exp\Big(\frac{1}{2}\,\mathbf{y}^\top \mathbf{A}_n \mathbf{y}+ \mathbf{y}^\top (\tau_0 \mathbf{t}+ \mathbf{x} \btheta_0) \Big),
\end{equation}
where 
\begin{equation}\label{eq:define_Z}
    Z_n(\mathbf{t},\mathbf{x}) = \sum\limits_{\by \in \{\pm 1\}^n}\exp\Big(\frac{1}{2}\,\mathbf{y}^\top \mathbf{A}_n \mathbf{y}+ \mathbf{y}^\top (\tau_0 \mathbf{t}+  \mathbf{x} \btheta_0) \Big)
\end{equation}
is the normalizing constant. The matrix 
$\mathbf{A}_n = \bA_n^{\top} \in \mathbb{R}^{n \times n}$ captures the interaction among units which is assumed known throughout and $\tau_0\in \mathbb{R}$ , $\btheta_0 \in \rr^d$ represent unknown parameters.  In social network applications, the matrix $\mathbf{A}_n$ is usually a scaled version of the adjacency matrix of the observed network. Throughout, we make the following assumptions on the parameter space. 

\begin{ass}[Parameter space]\label{assn:parameter_space}
    $(\tau_0,\btheta_0)\in [-B_0,B_0]\times [-M_0,M_0]^d$ for some $B_0,M_0>0$.
\end{ass}
Given covariates $\mathbf{x} = (\mathbf{x}_1, \cdots, \mathbf{x}_n)$, 
 the treatment assignments $\mathbf{T} = (\mathbf{T}_1, \cdots, \mathbf{T}_n) \in \{ \pm 1\}^n$ follow a propensity score model 
\begin{equation}\label{eq:prop_score}
    \bP(\mathbf{T} = \mathbf{t} |\mathbf{x})= \frac{1}{Z_n'(\mathbf{x})} \exp \Big(\frac{1}{2}\mathbf{t}^\top \mathbf{M}_n \mathbf{t} + \sum_{i=1}^n t_i \mathbf{x}^\top_i \boldsymbol{\gamma}_0 \Big),
\end{equation}
with $\boldsymbol{\gamma_0} \in [-M_0,M_0]^d$ for $M_0>0$. $Z_n'(\mathbf{x})$ refers to the normalization constant in the above model. We assume that the interaction matrix $\mathbf{M}_n = \mathbf{M}_n^{\top}$ is known throughout, and the propensity score model is known up to the parameter $\boldsymbol{\gamma}_0$. Note that we do not necessarily assume that $\mathbf{A}_n = \mathbf{M}_n$. 

Finally, we assume that the observed covariates $\bX_i \sim \mathbb{P}_X$ are i.i.d., where $\mathbb{P}_X$ is a probability distribution supported on $[-1,1]^d$. Assume that $\text{Var}(\bX_i)=\boldsymbol{\Sigma}$, where $\boldsymbol{\Sigma}$ is a $d \times d$ matrix with 
\begin{align}\label{eq:sigma_min_eval}
    \lambda_{\min}(\boldsymbol{\Sigma}) \ge c>0
\end{align}
for some $c>0$. 

Given the outcome regression model \eqref{eq:define_gibbs} and the g-computation formulae \eqref{eq:g_computation1}, \eqref{eq:g_computation2}, the natural algorithm to evaluate the causal effects $\mathrm{DE}$ and $\mathrm{IE}$ would involve sampling from the MRF \eqref{eq:define_gibbs}. In the seminal work~\cite{tchetgen2021auto} which introduced this formulation, the authors implement this sampling based strategy, and use an appropriate Gibbs sampler for \eqref{eq:define_gibbs}. This algorithm is \emph{universal} in that one can use the same algorithm irrespective of the precise details of the interaction matrix $\bA_n$ \eqref{eq:define_gibbs}. Unfortunately, it is well-known that MCMC algorithms might often be slow mixing in MRFs of the form \eqref{eq:define_gibbs}~\cite{levin2017markov}. In the most extreme case, the mixing time for common MCMC algorithms (initialized from an arbitrary starting state) scales as $\exp(\Theta(n))$. Consequently, this sampling based strategy for causal effect evaluation is ineffective as $n \to \infty$. In prior work~\cite{bhattacharya2024causal}, the authors developed fast iterative algorithms for causal effect estimation using mean-field algorithms. However, these algorithms assume that the outcome model \eqref{eq:define_gibbs} is at \emph{high temperature}; formally, we reparametrize $\bA_n = \beta \bG_n$ for some $\beta>0$ and a sequence of `standardized' interaction matrices $\bG_n$ (for examples see Section~\ref{sec:examples}). The parameter $\beta>0$ is referred to as the inverse temperature in statistical physics, and the high-temperature regime corresponds to $\beta>0$ being small. At high-temperature, the correlations in the MRF \eqref{eq:define_gibbs} are relatively weak, and one expects MCMC algorithms to also be fast-mixing. Thus although the prior mean-field algorithms provide practical speedup over sampling-based algorithms, both strategies rely crucially on weak-dependence in the outcome regression model \eqref{eq:define_gibbs}. This prompts the natural question:
\begin{center}
    \emph{Is efficient estimation of causal effects possible beyond high-temperature?}
\end{center}

\noindent 
\textbf{Our contributions:} In this work, we investigate causal effect estimation in the setup described above, focusing specifically on the low-temperature regime. 
\begin{itemize}
    \item[(i)] \textbf{Computational hardness:} Our first result (Theorem~\ref{thm:estimation_hard}) provides formal evidence against the existence of \emph{universal} algorithms for evaluation of  causal effects. Specifically, we show if there exists a polynomial time (in $n$) algorithm $\mathcal{A}$ which computes the direct effect for any interaction $\bA_n$, then there exists a sequence of polynomial time hypothesis tests for detecting a negatively spiked Wishart distribution~\cite{bandeira2020computational}. This problem is believed to exhibit average case computational hardness and thus provides rigorous evidence to the non-existence of universal algorithms for causal effect estimation. To the best of our knowledge, this is the first result exhibiting  computational hardness for causal inference under interference. Additionally, we note that there is substantial recent evidence for the existence of computational barriers in statistical models with high-dimensional parameters e.g. regression models~\cite{celentano2022fundamental}, community detection~\cite{hopkins2018statistical}, low-rank matrix estimation~\cite{berthet2013complexity} etc. In sharp contrast, we discover a computational bottleneck in the evaluation of low-dimensional treatment effect functionals; in our setting, the hardness arises due to the dependency in the model \eqref{eq:define_gibbs}, and not due to the high-dimensionality of the parameter space. 

    \item[(ii)] \textbf{Dense graphs:} The main takeaway from our first result is that to evaluate the treatment effects, particularly at low-temperature, one must utilize additional features in the interaction matrix $\bA_n$. In our second result, we assume that the interaction matrix $\bA_n$ corresponds to the (scaled) adjacency matrix of a sequence of dense graphs. Under this assumption, we utilize the algorithmic regularity lemma~\cite{fox2019fast} from combinatorics to develop polynomial time algorithms for the direct and indirect causal effects. We also show that as $n \to \infty$, the causal effects converge to well-defined limits, which are determined by the graphon limit of the underlying graph sequence $\bA_n$. This limit provides a well-defined notion of a population causal effect under interference. To the best of our knowledge, this interaction between causal inference and graph limit theory appears for the first time in our work. 

    \item[(iii)] \textbf{Gaussian interactions:} Finally, we study the case when $\bA_n$ is a symmetric matrix with i.i.d. Gaussian entries above the diagonal. In this case, we use the Parisi formula for spin glasses to derive a notion of limiting causal effects. We also develop a new algorithm based on Approximate Message Passing (AMP) to estimate the treatment effects. In prior work~\cite{bhattacharya2024causal}, the authors developed an AMP based algorithm for causal effect estimation which worked at high-temperature. The new algorithm leverages the structure of the optimal Parisi measure for spin glasses, and works at any temperature. This extends the scope of AMP methods significantly beyond the prior art.   

    \item[(iv)] \textbf{Parameter estimation:} The algorithms introduced above require knowledge of the model parameters $\tau_0$ and $\btheta_0$ in \eqref{eq:define_gibbs}. These parameters need to estimated from data. We utilize maximum pseudo-likelihood to estimate the model parameters; these estimates are consistent for the true model parameters. We subsequently show that our proposed algorithms are stable under perturbations to the model parameters. This facilitates fully data-driven causal effect estimation for the models described above.  
\end{itemize}
We emphasize that in both the examples described above i.e., $\bA_n$ arising from a dense graph or an i.i.d. Gaussian matrix, the natural Gibbs sampler for \eqref{eq:define_gibbs} mixes in exponential time at low temperature. Our result shows that despite the slow mixing for natural MCMC algorithms, causal effects can still be estimated efficiently, given some a priori structural assumptions on the interaction matrix $\bA_n$. We consider this to be one of the major conceptual contributions of this work.

\noindent 
\textbf{Prior work:}
There is growing interest in settings where treatments spill over from one unit to another~\cite{aronow2017estimating,athey2018exact,eckles2016design,hong2006evaluating,rosenbaum2007interference,sobel2006randomized,vanderweele2010direct}. Interference makes causal estimation intrinsically high-dimensional and thus most approaches impose structural constraints on the interference pattern. Early works relied on specific structural models~\cite{bramoulle2009identification,graham2008identifying,lee2007identification,manski1993identification} and are often criticized for their restrictive nature~\cite{angrist2014perils,goldsmith2013social}. The partial interference assumption, i.e., interference confined to known disjoint groups, offered a milder alternative~\cite{basse2019randomization,ferracci2014evidence,hayes2017cluster,hong2006evaluating,hudgens2008toward,kang2016peer,liu2014large,lundin2014estimation,park2023assumption,tchetgen2012causal}.  More recently, interference has been modeled through general networks using exposure mappings~\cite{aronow2017estimating,forastiere2021identification,jagadeesan2020designs,li2022random,manski2013identification,toulis2013estimation,ugander2013graph}, though typically under sparsity assumptions (e.g., bounded degree). In contrast, we study dense interference settings and study computational barriers in estimating the treatment effects. There is also an emerging line of work that goes beyond network-based interference either by imposing algebraic constraints on the interference structure (e.g. low-degree interference) \cite{cortez2023exploiting,eichhorn2024low} or by studying general interference \cite{yu2022estimating,viviano2025policy,choi2024new}. In the latter case, In these settings, prior work typically allows multiple interventions or focuses on alternative estimands which remain estimable under weaker assumptions.. In contrast, we focus on the computational barriers in estimating causal effects under interference.


We investigate treatment effect estimation from observational network data represented by a class of graphical models known as chain graphs~\cite{lauritzen2002chain,tchetgen2021auto, bhattacharya2020causal,sherman2018identification,shpitser2017modeling}. Existing approaches either use general purpose MCMC samplers or exploit weak interaction~\cite{bhattacharya2024causal} to estimate causal effects. 
Assuming $\mathbf{y} \in \{ \pm 1\}^n$, the outcome regression model \eqref{eq:define_gibbs} is closely connected to the Ising model from statistical physics. 
Sampling from the Ising model is a well-studied problem~\cite{sompolinsky1981dynamic}. It is well-established that the traditional Glauber dynamics or MCMC methods mix rapidly for sufficiently high temperature~\cite{adhikari2024spectral,anari2022entropic,anari2024trickle,eldan2022spectral}, with efficient approximate sampling also possible by some diffusion-based methods~\cite{el2022sampling,el2025sampling,huang2024sampling}. However, we focus on the low-temperature regime, where sampling from the Gibbs measure is provably hard~\cite{blanca2025tractability,galanis2024sampling,gheissari2022low,koehler2022sampling, sellke2025exponentially}.  Our first main result (Theorem~\ref{thm:estimation_hard}) provides evidence that evaluating causal effects by general purpose methods is also computationally hard in this regime.

We then identify two important classes of interaction matrices for which estimation remains tractable even at low temperature. First, we
consider dense graphs (Assumption~\ref{assn:interaction}), which includes regular graphs and block models~\cite{berthet2019exact}. Classical graph regularity results, notably Szemer\'edi's regularity lemma~\cite{szemeredi1975regular} and Frieze-Kannan regularity lemma~\cite{frieze1996regularity}, provides structural decomposition of dense graphs and and have become foundational tools across extremal combinatorics, additive number theory, and graph limits~\cite{komlos1995szemeredi,Lovasz2012}. Significant progress on algorithmic variants of regularity lemma~\cite{alon1994algorithmic, fox2019fast,frieze1999quick} now enables polynomial-time constructions of regular partitions and weak regularity approximations. These tools were recently used in~\cite{jain2018mean} to obtain $O(1)$ approximations of the log-partition function; here, we use regularity-based decompositions to design approximation algorithms for causal effects.

Beyond the mean-field regime, we study Gaussian interaction matrices and develop an AMP-based estimator for treatment effects. AMP methods have recently been applied in causal inference~\cite{bayati2024higher,bhattacharya2024causal,jiang2022new,shirani2023causal}, but our work is the first to formulate message-passing algorithms for causal effect estimation in parameter regimes where sampling is computationally infeasible. Our algorithm builds on recent advances in optimization for the Sherrington-Kirkpatrick models~\cite{el2021optimization,montanari2025optimization, sellke2024optimizing}. We exhibit how these ideas and tools are useful  in the context of estimation of causal effects under dense interference.

\noindent 
\textbf{Notation:} Given any $n \times n$, symmetric matrix $\mathbf{B}_n$, denote its operator norm by $\|\mathbf{B}_n\|$ and trace by $\text{Tr}(\mathbf{B}_n)$. Define its largest and smallest eigenvalues by $\lambda_{\max}(\mathbf{B}_n)$ and $\lambda_{\min}(\mathbf{B}_n)$ respectively. Denote by $\mathbf{I}_n$ the $n \times n$ identity matrix. Denote by $\mathbf{1}$ the $n$-length vector of all $\mathbf{1}$s. For $n \in \mathbb{N}$, define $[n]= \{1,2\ldots,n\}$. For two sequences of real numbers $a_n$ and $b_n$, $a_n= O(b_n)$ will denote that $\limsup_{n \rightarrow \infty} a_n/b_n= C$ for some $C \in [0,\infty)$, $a_n=o(b_n)$ will denote $\limsup_{n \rightarrow \infty} a_n/b_n=0$, and $a_n= \Theta(b_n)$ will denote $a_n= O(b_n)$ and $b_n=O(a_n)$ simultaneously. The $\ell^2$ and $\ell^{\infty}$ norms of $\mathbf{a}$ are denoted by $\|\mathbf{a}\|$ and $\|\mathbf{a}\|_\infty$, respectively. We use $\lesssim$ to denote an inequality up to a constant independent of $n$.

\noindent
\textbf{Structure:} The rest of the paper is structured as follows. We describe our main results in Section~\ref{sec:non-asymptotic}. We discuss some consequences of our results and some directions for future research in Section~\ref{sec:discussions}. Finally, we prove our results in Section~\ref{sec:proofs}. We defer some of our technical arguments to the Appendix.  

\noindent
\textbf{Acknowledgements:} SS thanks Mark Sellke for discussions on the performance of AMP at low temperature. SS
thankfully acknowledges support from NSF (DMS CAREER 2239234), ONR (N00014-23-1-2489)
and AFOSR (FA9950-23-1-0429).

\section{Our results}\label{sec:non-asymptotic}

Our starting point is the following expression for the causal effects derived in \cite[Lemma 1.1]{bhattacharya2024causal}.

\begin{lemma}\label{lem:de_define}
Set $\pi(\mathbf{t}_{-i})=2^{-(n-1)}$ for all $\mathbf{t}_i \in \{\pm 1\}^{n-1}$. Under the outcome model \eqref{eq:define_gibbs}, we have,  
\begin{align}\label{eq:simplify_de}
    &\mathrm{DE}=\frac{2}{n}\sum_{i=1}^{n} \bE_{\bar \bT, \bar \bX} \bE( \bar T_i \bY_i)=: \frac{2}{n} \bE_{\bar \bT, \bar\bX} \Big[\sum_{i=1}^{n}\bar T_i \langle \bY_i \rangle \Big], \nonumber \\
    &\mathrm{IE}:= \frac{1}{n} \bE_{\bar \bT, \bar\bX} \Big[\sum_{i=1}^{n} \langle \bY_i \rangle \Big]- \frac{1}{n} \bE_{-\mathbf{1}, \bar\bX} \Big[\sum_{i=1}^{n} \langle \bY_i \rangle \Big]- \frac{1}{2}\mathrm{DE},
\end{align}
where $\langle \bY_i \rangle := \langle \mathbf{Y}_i \rangle_{\mathbf{t},\bx}= \bE(\mathbf{Y}_i|\bt,\bx)$ and the expectation is taken with respect to the density \eqref{eq:define_gibbs}. Note that in \eqref{eq:simplify_de} above, $(\bar \bT,\bar \bX)$ are independent, $\bar \bT \sim \mathrm{Unif}(\{\pm 1 \}^n)$ and $\bar \bX= (\bar \bX_1, \cdots, \bar \bX_n)$, $\bar \bX_i \sim \mathbb{P}_{X}$ are i.i.d. 

\end{lemma}

Thus if the model parameters $\tau_0$, $\boldsymbol{\theta}_0$ in \eqref{eq:define_gibbs} are known, one can evaluate the causal effects $\mathrm{DE}$ and $\mathrm{IE}$ by computing the low-dimensional expectations $\langle \mathbf{Y} \rangle$. In \cite{bhattacharya2024causal}, the authors develop efficient algorithms to approximate these expectations for specific classes of interaction matrices $\bA_n$, under additional high-temperature assumptions on the outcome model \eqref{eq:define_gibbs}. In general, one would employ MCMC based techniques to approximate the low-dimensional marginals $\langle \mathbf{Y} \rangle$.   

It is well-known that sampling/approximating low-dimensional marginals of Markov Random Field models of the form \eqref{eq:define_gibbs} is hard at low-temperature (cf.~\cite{galanis2024sampling} and the references therein). This suggests that computing causal effects might also be challenging in low temperature regimes. Our first result shows that this is indeed true. 

\subsection{Computational hardness of treatment-effect estimation}


In this section, we investigate the inherent computational hardness in evaluating the treatment effects $\mathrm{DE}$ and $\mathrm{IE}$ at low-temperature. We refer to the direct effect as $\mathrm{DE}(\tau)$ below to highlight the dependence of the direct effect on $\tau$. Let $\mathcal{A}_n= \mathcal{A}_n(\bA_n, \tau)$ be a possibly randomized algorithm that runs in polynomial time in $n$, which computes $\widehat{\mathrm{DE}}(\tau)= \mathcal{A}_n(\bA_n, \tau)$. We introduce the definitions below for the direct effect $\mathrm{DE}$ for simplicity. These definitions have direct extensions to the indirect effect $\mathrm{IE}$. 

\begin{definition}[Uniform estimator]
\label{def:uniform_estimation}
    Fix a sequence $\mathbf{A}_n$ such that $\sup_n \| \bA_n\| < \infty$ and a sequence of polynomial time algorithms $\mathcal{A}_n$. For $\tau, \eta>0$, we say that $\mathcal{A}_n$ is a uniform estimator of $\mathrm{DE}$ on $[0,\tau]$ with tolerance $\eta$ if 
    \begin{align}\label{eq:de_approximable}
        \mathbb{P}\left( \sup_{\tau' \in [0, \tau]} \Big| \widehat{\mathrm{DE}}(\tau') - \mathrm{DE}(\tau') \Big| < \eta \right) = 1 - o(1),  
    \end{align}
    where $\mathbb{P}(\cdot)$ refers to the randomness of the algorithm $\mathcal{A}_n$. 
\end{definition}

\begin{remark}
    The notion of uniform estimators is closely linked to classical minimax estimation. Concretely, fix the sequence of interaction matrices $\{\bA_n: n \geq 1\}$ in \eqref{eq:define_gibbs} and consider the algorithmic task of computing the direct effect $\mathrm{DE}$. An adversary picks any $\tau' \in [0, \tau]$. An algorithm $\mathcal{A}_n$ is a universal estimator with tolerance $\eta$ if it can estimate $\mathrm{DE}(\tau')$ with error at most $\eta$ for any choice of 
    $\tau' \in [0, \tau]$ by the adversary. 
\end{remark}

\begin{remark}
     Note that the data generating distribution \eqref{eq:define_gibbs} is completely specified in our setting, and the bottleneck in computing $\mathrm{DE}$ is purely computational. Given polynomial time computational resources, one is only able to compute an approximation to the parameter of interest. The notion of uniform estimators introduced above captures this computational barrier to parameter evaluation. In theoretical computer science, such algorithms would be referred to as Polynomial Time Approximation Schemes (PTAS)~\cite{vazirani2001approximation}. Additionally, we note that the notion of uniform estimators is distinct from the traditional notion of statistical estimators; in statistical estimation, one seeks to learn the parameters of the unknown data distribution. On the contrary, uniform estimators compute a noisy approximation to a well-defined parameter under computational constraints. We still use the estimation terminology as it is more natural to a statistical audience. 
\end{remark}

\begin{remark}[Running time of $\mathcal{A}_n$] 
For any $\tau>0$ and $\tau' \in [0, \tau]$ we assume that the running time of $\mathcal{A}_n(\bA_n, \tau')$ is $O(n^{C(\tau, \eta)})$ for some constant $C(\tau, \eta)>0$. Equivalently, for a fixed tolerance $\eta>0$, the running time of $\mathcal{A}_n$ is uniformly bounded for all $\tau' \in [0, \tau]$. Additionally, the exponent of the polynomial is allowed to depend on the tolerance $\eta$. Consequently, the computational complexity is allowed to grow as we let $\eta \to 0$.      
\end{remark}

The notion of uniform estimation is intrinsically related to a tolerance $\eta>0$. In an ideal setting, one would have a polynomial time algorithm for any desired tolerance $\eta$. This corresponds to a notion of consistent estimation, and is formalized in the following definition. 

\begin{definition}[Consistent uniform estimation]
\label{def:consistent_uniform}
Fix a sequence $\bA_n$ with $\sup_n \|\bA_n\| <\infty$ and $\tau>0$. If for every $\eta>0$, there exists a uniform estimator $\mathcal{A}= \mathcal{A}_{n,\eta}$ of $\mathrm{DE}$ on $[0,\tau]$ with tolerance $\eta$, we say that $\mathrm{DE}$ admits consistent uniform estimation on $[0,\tau]$. 
\end{definition}

So far, we allow the algorithm $\mathcal{A}_n$ to depend on the interaction matrix $\bA_n$. One could hope for a universal algorithm, which would suffice for a broad class of interaction matrices. We formalize this notion in our following definition.

\begin{definition}[Universal uniform estimator]
For $\varepsilon>0$, let 
\begin{align}\label{eq:set_t_eps}
    \mathscr{T}(\varepsilon) = \{ \{\bA_n: n \geq 1\}: \sup_n \|\bA_n\| < \infty, \lambda_{\max}(\bA_n) - \lambda_{\min}(\bA_n) < 1+ \varepsilon\}. 
\end{align}
Fix $\tau, \eta>0$. We say that $\mathcal{A}_n$ is a universal uniform estimator of $\mathrm{DE}$ with parameters $(\varepsilon,\tau,\eta)$ if for any sequence $\{\bA_n : n \geq 1\} \in \mathscr{T}(\varepsilon)$, $\mathcal{A}_n(\bA_n, \cdot)$ is a uniform estimator of $\mathrm{DE}$ on $[0,\tau]$ with tolerance $\eta$. 
    
\end{definition}

\begin{remark}
    The notion of universal uniform estimation is stronger than uniform estimation introduced in Definition~\ref{def:uniform_estimation}. In this case, given $\varepsilon>0$ and $\tau>0$, the adversary can choose $\tau' \in [0, \tau]$ and a sequence of interaction matrices $\bA_n$ in $\mathscr{T}(\varepsilon)$. The statistician has to produce one algorithm $\mathcal{A}_n$ which is simultaneously $\eta$ close to $\mathrm{DE}(\tau')$ for any choice of $\tau' \in [0,\tau]$ and $\bA_n$ by the adversary. 
\end{remark}

Armed with these notions, we provide rigorous evidence that universal uniform estimation of the direct effect $\mathrm{DE}$ is impossible. To this end, we first introduce a problem which is expected to exhibit average-case hardness.

\begin{definition}
\label{def:wishart}
    Suppose $\beta \ge -1$, $\gamma>0$, and for $n \in \mathbb{N}$, define $N= N(n)= \lceil n/\gamma\rceil$. Define two probability measures on $\mathbb{R}^{n\times N}$ as follows:
    \begin{enumerate}
        \item[(i)] Under $\mu_0$, draw $\bz_1,\ldots , \bz_N \sim N(0,\bI_n)$ independently.
        \item[(ii)] Under $\mu_1$, we first draw $\mathbf{u} \sim \mathrm{U}(\{ \pm 1\}^n)$. Given $\mathbf{u}$, draw $\bz_1,\ldots \bz_N \sim \mathsf{N}(0, \bI_n+ \frac{\beta}{n} \bu \bu^\top)$.
    \end{enumerate}
    Denote the two measures $\mu_0$ and $\mu_1$ collectively by $\mathrm{Wishart}(\beta, \gamma)$.
\end{definition}

\begin{definition}[Hypothesis test]
    A polynomial time hypothesis test is an arbitrary two-valued function $\phi : \mathbb{R}^{n \times N}$ which can be evaluated in polynomial time in $n$. A polynomial time hypothesis test is asymptotically consistent if 
    \begin{align*}
    \lim\limits_{n} \mu_0(\phi(\bz_1,\ldots, \bz_N)=m_0)=  \lim\limits_{n} \mu_1(\phi(\bz_1,\ldots, \bz_N)=m_1)=1.
\end{align*}
\end{definition}

The following conjecture~\cite{bandeira2020computational} deals with the existence of asymptotically consistent polynomial time hypothesis tests.

\begin{conjecture}\label{conj:bkw}
    If $\beta>-1$ and $\beta^2<\gamma$, there does not exist an asymptotically consistent sequence of polynomial time hypothesis tests between the distributions of $\mathrm{Wishart}(\beta, \gamma)$. 
\end{conjecture}

In \cite{bandeira2020computational}, the authors provide rigorous evidence for this conjecture based on the low-degree likelihood framework~\cite{hopkins2018statistical}. Our next result is a reduction from universal uniform estimation to hypothesis testing in this spiked Wishart problem. 

\begin{theorem}\label{thm:estimation_hard}
  Assume $\boldsymbol{\theta}_0=0$ in \eqref{eq:define_gibbs}.  For any $\varepsilon>0$, there exists $\bar{\tau}= \bar{\tau}(\varepsilon)>0$ and a function $\eta: (  \bar{\tau}, \infty) \to \mathbb{R}^+$ such that the following holds: If there exists a universal uniform estimator $\mathcal{A}_n$ of $\mathrm{DE}$ with parameters $(\varepsilon, \tau, \eta(\tau))$ for some $\tau >  \bar{\tau}$ then Conjecture~\ref{conj:bkw} is  false.    
\end{theorem}

Theorem~\ref{thm:estimation_hard} implies if Conjecture~\ref{conj:bkw} holds, there does not exist a universal uniform estimator with arbitrarily small tolerance $\eta$. Equivalently, no universal uniform estimator over $[0,\tau]$ can achieve tolerance below $\eta(\tau)$. This result establishes that it is impossible to estimate the direct effect $\mathrm{DE}$ consistently by a common algorithm. The tolerance parameter $\eta(\tau)$ is analogous to a minimax lower bound on the estimation error; however, in our context, it captures a fundamental lower bound on the tolerance that can be achieved by a universal algorithm.



\begin{remark}
    Theorem~\ref{thm:estimation_hard} performs an average case reduction from the existence of universal uniform estimators of the direct effect $\mathrm{DE}$ to a polynomial time hypothesis test in the  spiked Wishart model (Definition~\ref{def:wishart}). We refer the interested reader to~\cite{brennan2018reducibility,brennan2019optimal,boix2025average} for recent progress on average case reductions in high-dimensional statistics. We note that these results focus on statistical-computational gaps in the inference of high-dimensional parameters. In contrast, we establish computational hardness in computing a low-dimensional treatment effect $\mathrm{DE}$. 
\end{remark}

\begin{remark}[Connections to NP-hardness]
Theorem~\ref{thm:estimation_hard} relies on average case hardness in the spiked Wishart problem (Definition \ref{def:wishart}). In the proof of  Theorem~\ref{thm:estimation_hard}, we show that a universal uniform estimator for $\mathrm{DE}$ can be used to design a PTAS for $\log Z_n$  \eqref{eq:define_Z} with $\tau_0=0$, $\btheta_0=0$. In a recent result, Kunisky \cite{kunisky2024optimality} shows that a PTAS for $\log Z_n$ implies the existence of consistent polynomial time hypothesis tests in the spiked Wishart problem, which contradicts Conjecture~\ref{conj:bkw}.  Theorem~\ref{thm:estimation_hard} thus follows upon combining our PTAS with the conclusions of \cite{kunisky2024optimality}.  Following the work of Kunisky \cite{kunisky2024optimality}, Galanis et. al. \cite{galanis2024sampling} show that approximating $\log Z_n$ at $\tau_0=0$, $\btheta_0=0$ with small tolerance is NP hard. On the other hand, with $\varepsilon$, $ \bar{\tau}(\varepsilon)$ and $\eta$ as in Theorem~\ref{thm:estimation_hard}, if there exists a universal uniform estimator of $\mathrm{DE}$ with parameters $(\varepsilon, \tau, \eta(\tau))$ for some $\tau >  \bar{\tau}$, then we show that one can estimate $\log Z_n$ (at $\tau_0 =0$, $\btheta_0=0$) with tolerance $\eta(\tau)$. This contradicts the NP hardness established in \cite{galanis2024sampling}. With this simple modification, we can reduce the evaluation of treatment effects under interference to NP hard problems from complexity theory.

\end{remark}

The main takeaway from Theorem~\ref{thm:estimation_hard} is that universal uniform estimation of treatment effects is impossible. However, under additional assumptions on the interaction matrix $\bA_n$, one can potentially develop tailored algorithms which facilitate consistent uniform estimation of the treatment effects $\mathrm{DE}$ and $\mathrm{IE}$. In the next two subsections, we consider $\bA_n$ arising from dense graphs and i.i.d. Gaussian matrices respectively, and develop consistent uniform estimates in these special cases. 


\subsection{Causal effect estimation for dense graphs}

In this section, we assume that the interaction matrix $\bA_n$ arises from an underlying sequence of dense graphs. In Section~\ref{sec:asymptotic}, we show that the direct and indirect causal effects converge to an asymptotic limit as $n \to \infty$. In Section~\ref{sec:regularity_lemma}, we turn to the estimation problem, and develop new algorithms for causal effect estimation based on the algorithmic regularity lemma. We emphasize that the asymptotic limit and the algorithm are valid, even at low temperature.

Throughout, we make the following assumption on the interaction matrices $\bA_n$. 

\begin{ass}[Interaction matrix]\label{assn:interaction}
    $\max_{i,j} |n \bA_n(i,j)| \leq 1$.  
\end{ass}
We provide natural examples of matrices $\bA_n$ satisfying these assumptions in Section~\ref{sec:examples}. 

\subsubsection{Asymptotic characterization using graph limits}\label{sec:asymptotic}
In this section, we study the limiting behavior of the causal estimands of interest. Assuming that the sequence of (scaled) interaction matrices $\bA_n$ converge in cut metric to a limiting graphon, we derive variational characterizations of the causal effects in terms of the limiting graphon. 
%
%
%
%
Cut distance/cut metric has been introduced in the combinatorics literature to study limits of graphs and matrices (see~\cite{frieze1999quick}), and has received significant attention in the theory of graph limits  (\cite{bc_lpi,borgs2018p,borgsdense1,borgsdense2}). For more details on the cut metric and its manifold applications, we refer the interested reader to \cite{Lovasz2012}. Below we formally introduce the notion of strong and weak cut distances used in our work.

\begin{definition}\label{def:defirst}
	Suppose $\mathcal{W}$ is the space of all symmetric real-valued functions on $[0,1]^2$ taking values in $[0,1]$. Given two functions $W_1,W_2\in \mathcal{W}$, define the strong cut distance between $W_1, W_2$ by setting
	$$d_\square(W_1,W_2):=\sup_{S,T}\Big|\int_{S\times T} \Big[W_1(x,y)-W_2(x,y)\Big]dx dy\Big|.$$
Here, the supremum is taken over all measurable $S,T \subseteq [0,1]$.  Define the weak cut distance $$\delta_\square(W_1,W_2):= \inf_{\sigma} d_\square(W^\sigma_1,W_2)= \inf_{\sigma} d_\square(W_1,W^\sigma_2)$$ where $\sigma$ ranges from all measure preserving bijections $[0,1]\rightarrow [0,1]$ and $W^\sigma(x,y)=W(\sigma(x),\sigma(y))$. Given a symmetric matrix $\mathbf{A}_n$, define the empirical graphon $W_{\mathbf{A}_n}\in \mathcal{W}$:
	\begin{align*}
		W_{\mathbf{A}_n}(x,y)=& \mathbf{A}_{n}(i,j)\text{ if }\lceil nx\rceil =i, \lceil ny\rceil =j.
	\end{align*}
    \end{definition}

 \begin{ass}
 \label{assn:graphon_conv}
	We will assume in this section that 
	the sequence of matrices $\{\mathbf{A}_n\}_{n\ge 1}$ defined in \eqref{eq:define_gibbs} converges in weak cut distance,~i.e. for some $W\in \mathcal{W}$,
	\begin{align}\label{eq:cut_con}
		\delta_{\square}(W_{n \mathbf{A}_n},W) \rightarrow 0.
	\end{align} 
\end{ass}

\noindent
We also require the following definition  to state our result. 

\begin{definition}\label{def:exp_tilt}
    For any probability measure $\mu$ on $[-1,1]$ and $\lambda \in \mathbb{R}$, define its $\lambda$-exponential tilt as 
    $$\frac{d\mu_\lambda}{d \mu}(x) := \exp(\lambda x - \alpha (\lambda)), \qquad \text{where } \alpha (\lambda):= \log \int e^{\lambda x} d \mu(x).$$
Then the function $\alpha(\cdot)$ is infinitely differentiable, with $$\alpha'(\lambda)=\mathbb{E}_{\mu_\lambda}(X),\quad \alpha''(\lambda)=\mathrm{Var}_{\mu_\lambda}(X)>0.$$
Assume now that $\mathrm{Supp}(\mu) = [-1,1]$. Consequently, for $m\in (-1,1)$, there exists $\lambda=\lambda(m)\in \rr$ such that $\bE_{\mu_\lambda}(y)=m$. Define $I(m)=D(\mu_\lambda|\mu)$, where $D(\cdot|\cdot)$ denotes the Kullback-Leibler divergence. Finally, define $I(1) = D(\delta_1 | \mu)$ and $I(-1) = D(\delta_{-1}|\mu)$, where $\delta_1$ and $\delta_{-1}$ refer to the point masses at $1$ and $-1$ respectively. 
\end{definition}
\begin{definition}
    Let $\mathcal{F}$ denote the set of all measurable functions on $[0,1]\times\rr \times [-1,1]$ to $[-1,1]$. For $F \in \mathcal{F}$ and $i \in \{1,2\}$,  define $F_i=F(U_i,\bX_i,\bT_i)$, where $U_i \sim \mathrm{Unif}(0,1)$, $\bX_i\sim \bP_X$, $\bT_i \sim \mathrm{Unif}(\{\pm 1 \})$ are independent. Let $W \in \mathcal{W}$ and recall $I$ introduced in Definition \ref{def:exp_tilt} with $\mu= \frac{1}{2}(\delta_{+1}+ \delta_{-1})$. Define 
    \begin{align}\label{eq:opt_g_f}
        &G_{W,\tau,\boldsymbol{\theta}, \gamma}(F)= \bE(W(U_1,U_2)F_1F_2)+ \bE(F_1 (\boldsymbol{\theta}^\top \bX_1+ \tau \bT_1+\gamma)) - \bE(I(F_1)),
        \nonumber \\
         &\tilde G_{W,\tau,\theta, \gamma}(F)= \bE(W(U_1,U_2)F_1F_2)+ \bE(F_1 (\boldsymbol{\theta}^\top \bX_1 - \tau+\gamma)) - \bE(I(F_1)).
    \end{align}
\end{definition}
For $\bt \in \{ \pm 1\}^n$ and $\bx = (\bx_1, \cdots, \bx_n) \in ([-1,1]^d)^{\otimes n}$, define, for $\gamma \in [-B,B]$,
\begin{align}
    &\tilde Z_n(\bt,\bx)= \sum\limits_{\by \in \{-1,1\}^n}\exp\Big(\frac{1}{2}\,\mathbf{y}^\top \mathbf{A}_n \mathbf{y}+ \sum_{i=1}^{n}  y_i (\tau_0 t_i + \btheta^\top_0 \bx_i +\gamma) \Big). \label{eq:extra_gamma}
\end{align}

\begin{theorem}\label{thm:zn_limit}
    Suppose the interaction matrix $\bA_n$ satisfies 
    \eqref{eq:cut_con}. Let $\bar{\mathbf{T}} \sim \mathrm{Unif}(\{ \pm 1\}^n)$, $\bar{\mathbf{X}} \sim \bP^{\otimes n}_X$ and $\mathbf{y}|\mathbf{t},\mathbf{x}$ satisfy \eqref{eq:define_gibbs}  with $\tau=\tau_0$, $\boldsymbol{\theta}=\boldsymbol{\theta}_0$. Then,
    \begin{equation}\label{eq:zn_limit}
       \frac{1}{n}\mathbb{E}_{\bar{\bT},\bar{\bX}} [\log \tilde Z_n (\bar{\bT},\bar{\bX})] \rightarrow \sup_{F\in \mathcal{F}} G_{W,\tau_0,\btheta_0,\gamma}(F).
    \end{equation}
If $\sup_{F\in \mathcal{F}} G_{W,\tau,\btheta_0, 0}(F)$ is differentiable w.r.t. $\tau$ at $\tau_0$, and we have
    \begin{equation}\label{eq:de_limit}
 \mathrm{DE} \to \mathrm{DE}_{\infty} := 2 \frac{\partial}{\partial \tau} \sup_{F\in \mathcal{F}} G_{W,\tau, \btheta_0,0}(F)\Big\vert_{\tau=\tau_0}.  
    \end{equation}
Further, 
\begin{align}
 \frac{1}{n}\mathbb{E}_{\bar{\bX}}[\log \tilde Z_n (-\mathbf{1},\bar{\bX})] \rightarrow \sup_{F\in \mathcal{F}} \tilde G_{W,\tau_0, \btheta_0, \gamma}(F).
 \end{align}
If both $\sup_{F\in \mathcal{F}} G_{W,\tau_0, \btheta_0,\gamma}$ and $\sup_{F\in \mathcal{F}} \tilde G_{W,\tau_0, \btheta_0,\gamma}$ are differentiable w.r.t. $\gamma$ at $0$, then we have

\begin{equation}\label{eq:ie_limit}
\mathrm{IE} \to \mathrm{IE}_{\infty}:=  \frac{\partial}{\partial \gamma} \sup_{F\in \mathcal{F}} G_{W,\tau_0, \btheta_0,\gamma}(F)\Big\vert_{\gamma=0}+ \frac{\partial}{\partial \gamma} \sup_{F\in \mathcal{F}} \tilde G_{W,\tau_0, \btheta_0,\gamma}(F)\Big\vert_{\gamma=0} - \frac{1}{2}\mathrm{DE}_{\infty}.
    \end{equation}
\end{theorem}

\begin{remark}{(Differentiability of the limit)}
\label{rem:differentiability}
 By direct differentiation, it follows that $\log \tilde{Z}_n(\bt, \bx)$ is a convex function in $\tau$. Thus the pointwise limit of $\mathbb{E}_{\bar{\bT},\bar{\bX}}[\log \tilde{Z}_n(\bar{\bT},\bar{\bX})]$ is also a convex function in $\tau$. Consequently, the limit is differentiable in $\tau$ at all but countably many points. Thus \eqref{eq:de_limit} specifies $\mathrm{DE}_{\infty}$ at all but countably many values of $\tau$.  The differentiability at $\gamma=0$ is more nuanced, and needs to be verified on a case-by-case basis. If the limit of $\mathbb{E}_{\bar{\bT},\bar{\bX}}[\log \tilde{Z}_n(\bar{\bT}, \bar{\bX})]$ is not differentiable at $\gamma=0$, we can derive bounds on the limiting indirect effect using one-sided derivatives.    
\end{remark}

\begin{remark}
    We present the limit characterization for sequences of dense interaction matrices $\bA_n$. However, the same techniques should extend to matrices $\bA_n$ converging to a limiting graphon in $L^p$ sense~\citep{bc_lpi}. 
    %
\end{remark}

 Theorem~\ref{thm:zn_limit} provides an exact expression for the limiting causal effects for sequences of dense interaction matrices $\bA_n$. If the limiting graphon $W$ and the parameters $\tau_0$, $\boldsymbol{\theta}_0$ are known, one could use this characterization to evaluate the limiting causal effects. The parameters $\tau_0$, $\boldsymbol{\theta}_0$ can be estimated from the data using maximum pseudo-likelihood, as discussed in Section~\ref{sec:parameter_estimation}. The limiting graphon $W$ can be estimated consistently by the empirical graphon $W_{n\bA_n}$. This yields an estimator which can be obtained by plugging in the estimated quantities into the variational representation. However, this variational problem could be challenging to solve in practice. In the next section, we present an algorithmic approach to estimate the causal effects, based on the algorithmic regularity lemma \cite{fox2019fast}. 

\subsubsection{Graphons and regularity Lemma}
\label{sec:regularity_lemma}
Here we introduce a more algorithmic approach based on the algorithmic regularity lemma \cite{fox2019fast}. 
%


We first describe our methodology for the case where $\bX_i$ are finitely supported. Formally, there exists 
\begin{equation}\label{eq:define_support}
    \mathcal{H}:= \{h_1,\ldots, h_m\}
\end{equation}
with $h_a \in [-1,1]^d$ such that $\mathbb{P}_{X}$ is supported on $\mathcal{H}$. The extension to general compactly supported covariates is discussed in Lemma~\ref{lem:general_h} in the Appendix. 


For general matrices $\bA_n$ satisfying Assumption~\ref{assn:interaction}, we will approximate the matrix using block-regular matrices, constructed using the algorithmic regularity lemma~\citep{fox2019fast}. To this end, define a block matrix as follows: 

\begin{definition}
    Fix $\varepsilon>0$ and $r \in \mathbb{N}$. For any $n \times n$ real symmetric matrix $\bA_n$, 
    we say that $\tilde{\mathbf{A}}_n$ is an $(r,\varepsilon)-$\textit{block approximation} of $\bA_n$ if there exists disjoint subsets $\{U_1, \cdots, U_{2^r}\} \subseteq [n]$ and $\{c_{kl} \in \mathbb{R}: 1\leq k,l \leq 2^r\}$ such that $\tilde{\mathbf{A}}_n = \sum_{k,l=1}^{2^r} c_{kl} \mathbf{1}_{{U}_k} \mathbf{1}_{U_l}^{\top}  $ and $\|\bA_n - \tilde{\mathbf{A}}_n \| \leq \varepsilon$. We set $U_0 = [n] \backslash \cup_{k=1}^{2^r} U_k$.  
\end{definition}

\begin{lemma}\label{lem:block_approx}
    For any matrix $\bA_n$ satisfying Assumption~\ref{assn:interaction} and $\varepsilon >0$, there exists $r:=r(\varepsilon)$ so that $\bA_n$ has an $(r,\varepsilon)$-block approximation $\tilde{\mathbf{A}}_n$. Further, this block approximation can be derived in $O(\varepsilon^{-O(1)} n^2 + n r)$ time. 
\end{lemma}

\begin{remark}
Following~\cite[Theorem 2.1]{fox2019fast} can choose $r= O(\varepsilon^{-16})$. In our subsequent discussion, we will suppress the dependence of $r$ on $\varepsilon$ for notational convenience. 
\end{remark}

\begin{remark}
    In the last section, we characterized the treatment effects using an infinite-dimensional graphon formulation. The algorithmic regularity lemma effectively implements a finite dimensional approximation to this infinite dimensional characterization via the block approximation $\tilde{\bA}_n$.  
\end{remark}

 Given an interaction matrix $\bA_n$, fix an $(r,\varepsilon)$-block approximation $\tilde{\mathbf{A}}_n$ and the corresponding partition $[n]= \cup
_{k=0}^{2^r} U_k$. Define the sets $S_a:=\{i \in [n]: \bX_i= h_a\}$, $a \in [m]$, where $h_i$'s are defined in \eqref{eq:define_support}. Finally, set $S_+:= \{i \in [n]: \bar T_i= 1\}$ and $S_-:= [n]\setminus S_{+}$.  Armed with these sets, define 
\begin{equation}
\label{eq:define_part_sum}
\mathcal{A}_{a,k,+}= S_a \cap U_k \cap S_{+}, \,\,\, \mathcal{A}_{a,k,-}= S_a \cap U_k \cap S_{-} 
\end{equation}
where $a \in [m]$, $k \in 0 \cup [2^r]$. In addition, we sample $\bar{\bT} \sim \mathrm{Unif}(\{\pm 1\}^n)$ and $\bar{\bX} = (\bar{\bX}_1, \cdots, \bar{\bX}_n)$ i.i.d. samples from $\mathbb{P}_X$. 

\begin{lemma}\label{lem:gibbs_simplify}
\label{lem:induced_dist}
    Let $\mathbf{y}\sim f_{(r,\varepsilon)}(\cdot|\bar{\bT}, \bar{\bX})$, where $f_{(r,\varepsilon)}$ denotes the distribution in \eqref{eq:define_gibbs} with interaction matrix $\tilde{\bA}_n$. 
    Define \begin{equation}\label{eq:define_part_sum_1}
     V_{a,k,+}= \sum_{\ell \in  \mathcal{A}_{a,k,+}} y_{\ell}, \,\,\,\, V_{a,k,-}= \sum_{\ell \in  \mathcal{A}_{a,k,-}} y_{\ell},
\end{equation} 
Then we have, for $a\in [m]$, $ 0\leq k \leq 2^r$,  
\begin{align}\label{eq:simplified_gibbs}
     &f_{(r,\varepsilon)}(V_{a,k,+} = v_{a,k,+}, V_{a,k,-} = v_{a,k,-}, a\in[m], 0\leq k \leq 2^r| \overline{\bT}, \overline{\bX}) \nonumber \\ 
     &\propto \prod_{a=1}^{m} \prod_{k=1}^{2^r} \binom{| \mathcal{A}_{a,k,+}|}{\frac{| \mathcal{A}_{a,k,+}|+ v_{a,k,+}}{2}} \binom{| \mathcal{A}_{a,k,-}|}{\frac{| \mathcal{A}_{a,k,-}|+ v_{a,k,-}}{2}} \times \nonumber\\
     &\exp \Bigg( \sum_{k,l=1}^{2^r} c_{kl} \Big(\sum_{a=1}^m\Big(v_{a,k,+}+v_{a,k,-}\Big)\Big)\Big(\sum_{a=1}^m\Big(v_{a,l,+}+v_{a,l,-}\Big)\Big) \nonumber\\
    &+ \sum_{a=1}^m \sum_{k=0}^{2^r}  \left( v_{a,k,+} (\tau_0+ h^\top_a \btheta_0)+ v_{a,k,-} ( - \tau_0+ h^\top_a \btheta_0)\right) \Bigg). 
\end{align}
In particular, the induced distribution of $\{V_{a,k,+}, V_{a,k,-}: a \in [m], 0 \leq k \leq 2^r\}$ is supported on $O(n^{2m(2^r+1)})$ points; thus, the normalization constant of the induced distribution may be explicitly evaluated in $O(n^{2m(2^r+1)})$ time. 
\end{lemma}

Using Lemma~\ref{lem:induced_dist}, we  can evaluate $\mathbb{E}_{f_{(r,\varepsilon)}}(V_{a,k,+}| \overline{\bT}, \overline{\bX})$ and $\mathbb{E}_{f_{(r,\varepsilon)}}(V_{a,k,-}| \overline{\bT}, \overline{\bX})$ in $O(n^{2m(2^r+1)})$ time. We will denote these two conditional expectations by $\langle V_{a,k,+} \rangle_{(r,\varepsilon)}$ and $\langle V_{a,k,-} \rangle_{(r,\varepsilon)}$ respectively. 
%
%
Now we turn to computationally efficient estimators for the treatment effects. 
Using \eqref{eq:simplify_de}, a natural estimator of direct effect is given by
\begin{equation}\label{eq:de_hat_define}
    \widehat{\mathrm{DE}}_{(r,\varepsilon)} =\frac{2}{n} \sum_{a=1}^{m} \sum_{k=0}^{2^r} (\langle V_{a,k,+} \rangle_{(r,\varepsilon)}- \langle V_{a,k,-}\rangle_{(r,\varepsilon)} ).
\end{equation}
Further, given treatment $(-1, \ldots, -1)$, and covariate $\bar \bX$, we can compute $\langle \tilde V_{a,k,+} \rangle_{(r,\varepsilon)}$ and $ \langle \tilde V_{a,k,-} \rangle_{(r,\varepsilon)}$ as above. Again by \eqref{eq:simplify_de}, a natural estimator of indirect effect is
\begin{equation}\label{eq:ie_hat_define}
    \widehat{\mathrm{IE}}_{(r,\varepsilon)} =\frac{1}{n} \sum_{a=1}^{m} \sum_{k=0}^{2^r} (\langle V_{a,k,+} \rangle_{(r,\varepsilon)}+\langle V_{a,k,-} \rangle_{(r,\varepsilon)} )- \frac{1}{n} \sum_{a=1}^{m} \sum_{k=0}^{2^r} ( \langle \tilde V_{a,k,+} \rangle_{(r,\varepsilon)} +  \langle \tilde V_{a,k,-} \rangle_{(r,\varepsilon)}) -\frac{1}{2} \widehat{\mathrm{DE}}_{(r,\varepsilon)}.
\end{equation}

\noindent 
Our proposed method is summarized in Algorithm \ref{algo:de_method}.

\begin{algorithm}[ht] 
  \begin{flushleft}
  \textbf{Input:} 
  The interaction matrix $\bA_n$, $\delta>0$.

    \textbf{Output:} Estimates $\widehat{\mathrm{DE}}_{(r,\delta)}$ (Direct Effect) and $\widehat{\mathrm{IE}}_{(r,\delta)}$ (Indirect Effect).

\textbf{Steps:} 
\begin{enumerate}

\item Generate $\bar \bT= (\bar T_1,\ldots, \bar T_n)$ from the uniform probability distribution on $\{\pm 1\}^n$. Generate $\bar \bX  = (\bar \bX_1,\cdots, \bar \bX_n) $ i.i.d. $\mathbb{P}_X$ independent of $\bar \bT$.
\item Compute the \textit{$(r,\delta)$-block approximation} of $\bA_n$, denoted by $\widetilde{\bA}_n$.
\item Define the sets $S_a$'s and $U_k$'s as~\eqref{eq:define_part_sum}. Compute $\langle V_{a,k,+} \rangle_{(r,\delta)},  \langle V_{a,k,-} \rangle_{(r,\delta)} $ exactly from \eqref{eq:simplified_gibbs}. 
\item Plugging in $ \langle V_{a,k,+} \rangle_{(r,\delta)},  \langle V_{a,k,-} \rangle_{(r,\delta)}$ in \eqref{eq:de_hat_define}, compute $\widehat{\mathrm{DE}}$.
\item Set treatment $= (-1,\ldots, -1)$ and sample $ \langle \tilde V_{a,k,+} \rangle_{(r,\delta)},  \langle \tilde V_{a,k,-}\rangle_{(r,\delta)}$ from \eqref{eq:simplified_gibbs}. Compute $\widehat{\mathrm{IE}}_{(r,\delta)}$ using \eqref{eq:ie_hat_define}.
\end{enumerate}
\end{flushleft}
\caption{Alg($\bA_n,\delta)$}
\label{algo:de_method}
\end{algorithm} 

The following result establishes rigorous guarantees for Algorithm~\ref{algo:de_method} on a broad class of interaction matrices. 

\begin{theorem}\label{thm:de_consistency}
    Assume $\bA_n$ satisfies Assumption~\ref{assn:interaction}. 
    For any $\varepsilon >0$, there exists $\delta :=\delta(\varepsilon) >0$ such that the estimates $\widehat{\mathrm{DE}}_{(r,\delta)}$ and $\widehat{\mathrm{IE}}_{(r,\delta)}$ satisfy
    \begin{align}
        \Big| \ee_{\bar \bT, \bar \bX} [\widehat{\mathrm{DE}}_{(r,\delta)}] - \mathrm{DE} \Big| <\varepsilon , \qquad 
        \Big| \ee_{\bar \bT, \bar \bX} [\widehat{\mathrm{IE}}_{(r,\delta)}] - \mathrm{IE} \Big| <\varepsilon.\,\,\,  \nonumber 
    \end{align}
\end{theorem}

\begin{remark}
\label{rem:averages}
    In practice, we sample i.i.d. copies $(\bar{\mathbf{T}}_1, \bar{\bX}_1), \cdots, (\bar{\mathbf{T}}_k, \bar{\bX}_k)$ of $(\bar{\mathbf{T}}, \bar{\bX})$ and compute independent estimates $(\widehat{\mathrm{DE}}^{(1)}_{(r,\delta)}, \widehat{\mathrm{IE}}^{(1)}_{(r,\delta)}), \cdots, (\widehat{\mathrm{DE}}^{(k)}_{(r,\delta)}, \widehat{\mathrm{IE}}^{(k)}_{(r,\delta)})$. Finally, we compute the averaged estimate 
    \begin{align}
        \widehat{\mathrm{DE}}_{\mathrm{avg}} = \frac{1}{k} \sum_{j=1}^{k} \widehat{\mathrm{DE}}^{(j)}_{(r,\delta)},\,\,\,\, \widehat{\mathrm{IE}}_{\mathrm{avg}} = \frac{1}{k} \sum_{j=1}^{k} \widehat{\mathrm{IE}}^{(j)}_{(r,\delta)}. \nonumber 
    \end{align}
    The averages can be computed in $O(k. n^{2m (2^r+1)})$ time, and have accuracy $\varepsilon + O(1/k)$. Recalling Definition~\ref{def:consistent_uniform}, we note that this algorithm facilitates consistent uniform estimation of the treatment effects. 
\end{remark}
We state Theorem~\ref{thm:de_consistency} and Remark~\ref{rem:averages} for deterministic interaction matrices. For random interaction matrices, the above result continues to hold if $\bA_n$ satisfies the conditions of Theorem~\ref{thm:de_consistency} almost surely. 
We note that Theorem~\ref{thm:de_consistency} guarantees $\varepsilon$-consistency, and one cannot generally let $\varepsilon \to 0$ as the population size $n \to \infty$. For general interaction matrices $\bA_n$, this is an artifact of the regularity lemma~\cite[Theorem 2.1]{fox2019fast} we invoke to approximate $\bA_n$ by a block matrix. If the sequence $\bA_n$ can be approximated by a block constant matrix with error $\varepsilon \to 0$ as $n \to \infty$, the corresponding treatment effect estimates would, in turn, be consistent (i.e. $\varepsilon\to 0$ as $n \to \infty$). We illustrate this via a concrete example in  Section~\ref{sec:examples}. Finally, we finish this section by connecting our algorithmic estimate to the asymptotic limiting causal effects characterized in Theorem~\ref{thm:zn_limit}. 

\begin{corollary}
\label{cor:algo_limiting}
    Assume that $\bA_n$ satisfies \eqref{eq:cut_con} and that $\mathrm{DE} \to \mathrm{DE}_{\infty}$ and $\mathrm{IE} \to \mathrm{IE}_{\infty}$ as $n \to \infty$. Then we have , 
    \begin{align}
        \lim_{\varepsilon \to 0} \lim_{n \to \infty} \Big| \mathbb{E}_{\bar{\bT},\bar{\bX}}[\widehat{\mathrm{DE}}_{(r,\delta)}] - \mathrm{DE}_{\infty} \Big| = 0,\,\,\,\,\,\lim_{\varepsilon \to 0} \lim_{n \to \infty} \Big| \mathbb{E}_{\bar{\bT},\bar{\bX}}[\widehat{\mathrm{IE}}_{(r,\delta)}] - \mathrm{IE}_{\infty} \Big| = 0. \nonumber    
    \end{align}
\end{corollary}

The proof of this corollary follows immediately from Theorem~\ref{thm:zn_limit} and Theorem~\ref{thm:de_consistency}, and is thus omitted.

 \subsubsection{Applications}\label{sec:examples}

 In this section, we present some concrete examples of interaction matrices $\bA_n$ satisfying Assumptions \ref{assn:interaction} and \ref{assn:graphon_conv}. In addition, we specialize Algorithm \ref{algo:de_method} to the complete graph---in this case, the interaction matrix $\bA_n$ is already constant, and thus the causal estimators are consistent as $n \to \infty$.

We first present some canonical examples of matrices $\bA_n$ satisfying Assumptions \ref{assn:interaction} and \ref{assn:graphon_conv}. 
\begin{enumerate}
    \item[(i)] \textbf{Ising blockmodel:} Here, one considers 
    $\by$ distributed according to an Ising model with a block structure analogous to the one arising in
the stochastic blockmodel~\citep{berthet2019exact}. Assume $n \ge 2$ is an even integer and let $S \subset [n]$ with $|S|= n/2$ be a subset of vertices. Define
\begin{equation*}
    \bA_n(i,j): = \begin{cases}
\frac{\alpha}{n} \quad \text{if  } (i,j) \in (S \times S) \cup (S^c \times S^c), \\
\frac{\beta}{n} \quad \text{otherwise}.
\end{cases}
\end{equation*}
for some $\alpha, \beta>0$. If $\alpha=\beta$, then \eqref{eq:define_gibbs} is equivalent to Curie-Weiss models. We see that $\sup_{i,j} n |\bA_n(i,j)| < 1$ if $\max\{\alpha, \beta\} <1$. To check the spectral norm condition, note that $\bA_n(i,j) \geq 0$ and thus $\| \bA_n\| \leq \max_{i} \sum_{j} |\bA_n(i,j)| = (\alpha + \beta)/2$. The empirical graphon $W_{n\bA_n}$ converges to the block constant graphon with $\alpha$ on two diagonal blocks and $\beta$ on the off-diagonal blocks.

\item[(ii)] \textbf{Erd\H{o}s-R\'{e}nyi graphs:} Let $\mathscr{G}(n,p)$ be the Erd\H{o}s-R\'{e}nyi random graph on $n$ vertices with edge probability $p \in [0,1]$. 
For $\beta>0$, set $\bA_n(i,j) = \frac{\beta}{n} \mathbf{1}(i \sim j)$, where $i \sim j$ is $i$ and $j$ are connected in $\mathcal{G}_n$.  Further, $\max_{i,j} n |\bA_n(i,j)| \le \beta$ and the spectral norm condition can be checked analogous to the previous example. The empirical graphon converges to the constant graphon $W \equiv \beta$ in this case. 

\item[(iii)] \textbf{Regular graphs:} Let $\mathscr{G}_n$ be a sequence of $d_n$-regular graphs on $n$ vertices with $d_n = \Theta(n)$. Let $\bA_n(i,j) = \frac{\beta}{d_n} \mathbf{1}(i \sim j)$. We can verify Assumptions \ref{assn:interaction} and \ref{assn:graphon_conv} analogous to the prior examples. The limiting graphon is again the constant function $W \equiv \beta$.

 
\end{enumerate}

We note that in each of the above examples, the matrix $\bA_n$ can be approximated with a block constant matrix with error $\varepsilon_n \to 0$ as $n \to \infty$. Thus in these examples, the causal effects can be consistently estimated as $n \to \infty$ using Algorithm~\ref{algo:de_method}. Below, we re-derive Algorithm~\ref{algo:de_method} for the special case of Curie-Weiss interactions and no covariates. In this case, the matrix $\bA_n$ has equal values on all off-diagonal entries. Our derivation will (i) help motivate Algorithm~\ref{algo:de_method}  and (ii) emphasize the statistical consistency of the resulting causal estimates as the population size $n \to \infty$.

\noindent 
\textbf{Motivating example:}\label{sec:curie_weiss}
For $\beta>0$, consider interaction matrix $\bA_n = \frac{\beta}{n} (\mathbf{1}\mathbf{1}^{\top} - \mathbf{I}_n)$, the scaled adjacency matrix of a complete graph on $n$-vertices. For simplicity, we assume no covariates, i.e., we only observe $\by, \bt$. 
The exponent of the Gibbs measure~\eqref{eq:define_gibbs} is equivalent to $\frac{n\beta}{2}\, (\overline{\mathbf{y}})^2 + \tau_0 \mathbf{y}^\top \mathbf{t}$.

To estimate the causal effects, we first generate $\bar \bT$ from the uniform probability distribution on $\{\pm 1\}^n$. Define 
$S_+= \{i \in [n]: \overline{T}_i= 1\}$, and $S_{-}= [n]\setminus S_{+}$. Using this notation, the exponent equals:
\begin{align}
    \frac{n\beta}{2} \overline{\mathbf{y}}^2 + \mathbf{y}^\top (\tau_0 \overline{\bT}) & = \frac{n\beta}{2} \overline{\mathbf{y}}^2 + \tau_0 \Big( \sum_{i \in S_+} y_i- \sum_{i \in S_-} y_i \Big) \nonumber \\
    & = \frac{n\beta}{2} \Big( \sum_{i \in S_+} y_i + \sum_{i \in S_-} y_i \Big) ^2 +\tau_0 \Big( \sum_{i \in S_+} y_i- \sum_{i \in S_-} y_i \Big)  \label{eq:cw_simplify}
\end{align}
Define $y_+=  \sum_{i \in S_+} y_i$, $y_-=  \sum_{i \in S_-} y_i$. This implies, the conditional distribution can be written as a measure on $\mathbb{R}^2$ as:
\begin{align}\label{eq:sum_gibbs}
    f(y_+ = v_{+}, y_- = v_{-}| \overline{\bT}) \propto \binom{|S_{+}|}{\frac{|S_{+}|+v_{+}}{2}} \binom{|S_{-}|}{\frac{|S_{-}|+v_{-}}{2}} \exp \left( \frac{\beta}{2n} (v_+ + v_{-})^2 +\tau_0 (v_+ - v_{-})\right).
\end{align}
Note that $(v_+,v_{-}) \in \mathbb{Z}^2 \cap ([-|S_+|, |S_+|] \times [-|S_{-}|,|S_{-}|])$ and thus $(v_+, v_{-})$ is supported on $O(n^2)$ points. The normalization constant of $f$ can thus be explicitly evaluated in $O(n^2)$ time.  
Let $(V_{+}, V_{-})$ denote a sample from $f$. Further, using that $f$ is supported on $O(n^2)$ points, one can evaluate $(\mathbb{E}_f(V_1), \mathbb{E}_f(V_2)):= (\langle V_1 \rangle, \langle V_2 \rangle)$ in $O(n^2)$ time.  Recalling \eqref{eq:simplify_de}, we estimate the  direct effect by the estimator \[\widehat{\mathrm{DE}} =\frac{2}{n} ( \langle V_{+} \rangle  - \langle V_{-} \rangle ).\]
To estimate the indirect effect, we repeat the algorithm for treatment assignment $\overline{\bT}= (-1,\ldots,-1)$ and denote the resulting sample by $(\tilde V_{+}, \tilde V_{-})$. In $O(n^2)$ time, we can estimate indirect effect as 
\[\widehat{\mathrm{IE}} =\frac{1}{n} (\langle V_{+} \rangle +\langle V_{-} \rangle ) - \frac{1}{n} (\langle \tilde V_{+} \rangle + \langle \tilde V_{-} \rangle ) -\frac{1}{2} \widehat{\mathrm{DE}}.\]
Of course, this is a specific instantiation of Algorithm~\ref{algo:de_method}, but we include this derivation here to motivate the general version presented earlier. 
This scheme assumes oracle knowledge of the underlying model parameters ($\tau_0$ in this special case, $\tau_0$ and $\btheta_0$ in the general case). These parameters will be estimated from the observed data; we refer to Lemma~\ref{lem:mpl_plug_in} for the estimation guarantees.

 \subsection{Causal effect estimation under Gaussian interactions}

 In this section, we assume that the interaction matrix $\bA_n$ is a symmetric Gaussian matrix. In Section~\ref{sec:gaussian_limit}, we derive an asymptotic limit for the direct and indirect causal effects as the population size $n \to \infty$. In Section~\ref{sec:amp}, we introduce an algorithm to estimate the causal effects based on Approximate Message Passing (AMP). We note that our results and algorithm are valid even at low temperature. In prior work \cite{bhattacharya2024causal}, the authors studied AMP based estimation algorithms for Gaussian interaction matrices. However, this prior algorithm is valid only at high temperature. The algorithm introduced here is more general, and works even at low temperature. 

 Throughout, we make the following assumption on the interaction matrix $\bA_n$. 

 \begin{ass}[Interaction matrix]\label{assn:interaction_gaussian}
    $\bA_n = \bA_n^{\top}$, $\bA_n = \beta \mathbf{G}_n$ for $\beta>0$, $\{G_n(i,j) : i <j\} \sim \mathcal{N}(0, \frac{1}{n})$, $G_n(i,i)=0$ for $1\leq i \leq n$. 
\end{ass}


\subsubsection{Asymptotic characterization using spin glasses}
\label{sec:gaussian_limit}
In this section, we derive a limiting characterization for the causal effects of interest. Our results will be phrased in terms of the Parisi formula for spin glasses \cite{talagrand2006parisi}.

Let $\mathcal{P}([0,1])$ be the space of probability measures on the interval $[0,1]$ endowed with the topology of weak convergence. For any measure $\mu \in \mathcal{P}([0,1])$, denote its distribution function via $\mu(t)= \mu([0,t])$. For any $\beta > 0$, consider the following PDE on $(t,x) \in [0,1] \times\mathbb{R}$: 
\begin{align}
    &\partial_t \Phi(t,x) + \frac{1}{2} \beta^2 \partial_{xx} \Phi(t,x)+  \frac{1}{2} \beta^2 \mu(t) (\partial_x \Phi(t,x))^2 =0, \nonumber \\
    & \Phi(1,x)= \log 2\cosh(x), \label{eq:parisi_pde}
\end{align}
where $\Phi= \Phi_\mu$ depends on the measure $\mu$. This Parisi PDE is solved backwards in time with the given final condition at $t=1$. Estimation and uniqueness of the above PDE is well-established~\cite{jagannath2016dynamic}. Given $\Phi_\mu$, the Parisi functional defined as
\begin{equation}\label{eq:parisi_functional}
    P_{\tau_0,\btheta_0,\gamma}(\mu)= \mathbb{E}[\Phi_{\mu}(0,\tau_0 T+H + \gamma)]- \frac{\beta^2}{2} \int_0^1 t \mu(t) dt,
\end{equation}
where $T\sim \mathrm{Unif}(\pm 1)$, $H \stackrel{d}{=}\bx^\top \btheta_0$, $\bx \sim \mathbb{P}_{X}$ are independent and the expectation $\bE[\cdot]$ in~\eqref{eq:parisi_functional} is w.r.t. $T,H$.
The connection between the free energy of Gaussian interaction matrices and Parisi functional was first conjectured by Parisi~\cite{parisi1979infinite}, and rigorously proved by~\cite{panchenko2013parisi,talagrand2006parisi}. In our setting, recalling $\tilde{Z}_n(\mathbf{t},\mathbf{x})$ from \eqref{eq:extra_gamma}, we have, 
\begin{equation}\label{eq:parisi_answer}
  \lim\limits_{n \rightarrow \infty} \frac{1}{n} \mathbb{E}_{\bar{\bT},\bar{\bX}}[\log \tilde{Z}_n (\bar{\bT}, \bar{\bX})] = \inf_{\mu \in \mathcal{P}([0,1])} P_{\tau_0,\btheta_0,\gamma}(\mu).
\end{equation}
Further the Parisi functional is strictly convex~\cite{auffinger2015parisi}; consequently, the variational problem in the RHS of \eqref{eq:parisi_answer} attains the minimum and  has a unique minimizer $\mu^\star$. In our subsequent computation, it will be helpful to track the dependence of the Parisi variational problem and the optimizer on $\tau_0$ and $\gamma$. Consequently, we set 
\begin{align}\label{eq:define_upsilon}
    \upsilon(\tau_0,\gamma)= \min_{\mu \in \mathcal{P}([0,1])} P_{\tau_0,\btheta_0,\gamma}(\mu),  \qquad \mu^\star_{\tau_0, \gamma}= \argmin_{\mu \in \mathcal{P}([0,1])} P_{\tau_0,\btheta_0,\gamma}(\mu).
\end{align}


To characterize the limiting causal effects, we will need an additional functional, which we introduce next. For $\mu \in \mathcal{P}([0,1])$, we define 
\begin{align}\label{eq:gamma_parisi}
     \widehat{P}_{\tau_0,\btheta_0,\gamma}(\mu)= \mathbb{E}[\Phi_{\mu}(0, - \tau_0 + H+\gamma)]- \frac{\beta^2}{2} \int_0^1 t \mu(t) dt. 
\end{align}
Analogous to \eqref{eq:parisi_answer}, we have, 
\begin{equation}\label{eq:parisi_answer_withgamma}
  \lim\limits_{n \rightarrow \infty} \frac{1}{n} \mathbb{E}_{\bar{\bX}} [\log \tilde{Z}_n(-\mathbf{1}, \bar{\bX})] = \inf_{\mu \in \mathcal{P}([0,1])} \widehat{P}_{\tau_0,\btheta_0,\gamma}(\mu).
\end{equation}
Similar to \eqref{eq:parisi_answer}, the functional $\hat{P}$ is strictly convex, and has a unique minimizer. We denote 
\begin{align}\label{eq:define_mustar}
    \widehat{\upsilon}(\tau_0, \gamma) = \min_{\mu \in \mathcal{P}([0,1])} \widehat{P}_{\tau_0, \btheta_0, \gamma}(\mu),\qquad \widehat{\mu}^\star_{\tau_0,\gamma}= \argmin_{\mu \in \mathcal{P}([0,1])} \widehat{P}_{\tau_0, \btheta_0, \gamma}(\mu).
\end{align}
%
%
Armed with these notions, we have the following characterization of the limiting causal effects.
\begin{theorem}
\label{cor:de_answer}
Suppose the interaction matrix $\bA_n$ satisfies Assumption~\ref{assn:interaction_gaussian}. We have,
    \begin{align*}
       \lim\limits_{n \rightarrow \infty} \mathrm{DE} &=  \mathrm{DE}_{\infty} := 2\frac{\partial}{\partial \tau_0} \upsilon(\tau_0, 0)= 2 \bE[T \partial_x\Phi_{\mu^\star_{\tau_0,0}}(0,\tau_0 T+ H)].
\end{align*}
Additionally we have, 
\begin{align*}
     \lim\limits_{n \rightarrow \infty} \mathrm{IE} &= \mathrm{IE}_{\infty}:=  \frac{\partial}{\partial \gamma}  \upsilon(\tau_0, \gamma) \Big\vert_{\gamma=0}- \frac{\partial}{\partial \gamma} \widehat{\upsilon}(\tau_0, \gamma)\Big\vert_{\gamma=0} -\frac{1}{2}\frac{\partial}{\partial \tau}\upsilon(\tau_0,0) \\
       &=  \bE[\partial_x\Phi_{\mu^\star_{\tau_0,0}}(0,\tau_0 T+H)] -  \bE[ \partial_x\Phi_{\widehat{\mu}^\star_{\tau_0,0}}(0, -\tau_0 + H)]-  
       \frac{1}{2}\mathrm{DE}_{\infty}.
    \end{align*}
\end{theorem}


\subsubsection{Algorithms via Approximate Message Passing}
\label{sec:amp}
In this section, we introduce algorithms to estimate the causal effects under Gaussian interaction matrices $\bA_n$. Our algorithms are based on Approximate Message Passing (AMP). 
In Algorithm~\ref{alg:amp}, we present our algorithm to compute the estimate for the direct effect $\mathrm{DE}$.  

\begin{algorithm}[ht] 
  \begin{flushleft}
  \textbf{Input:} 
  The interaction matrix $\bA_n= \beta \mathbf{G}_n$, $M \geq 1$.

    \textbf{Output:} Estimate $\widehat{\mathrm{DE}}_{M}$ (Direct Effect).

\textbf{Steps:} 
\begin{enumerate}

\item Generate $\bar \bT= (\bar T_1,\ldots, \bar T_n)$ from the uniform probability distribution on $\{\pm 1\}^n$. Generate $\bar \bX  = (\bar \bX_1,\cdots, \bar \bX_n) $ i.i.d. $\mathbb{P}_X$ independent of $\bar \bT$.
\item Define the function
\begin{equation}\label{eq:g_define}
    g(x)= \partial_x\Phi_{\mu^\star_{\tau_0, 0}}(q,x),
\end{equation}
where $\mu^\star_{\tau_0, 0}$ is defined as in ~\eqref{eq:define_upsilon} and $q= \inf (\rm{supp}(\mu^\star_{\tau_0, 0})) \in [0,1)$. Define
\begin{equation}\label{eq:g_k_define}
    g_k(\bh_1,\bh_2, \bw^{0}, \ldots \bw^k)= g(\bh_1+\bh_2+\bw^k)
\end{equation}
\item Initialize: Set $\bh_1= (\tau_0 \bar T_1, \ldots, \tau_0 \bar T_n)$, $\bh_2= (\bar \bX^\top_1 \btheta_0, \ldots, \bar \bX^\top_n \btheta_0)$. Define $\bw^k=\bu^k=\bm^k= \bzero$.
\item Iteration: For $ 1 \le k \le M$, define
    \begin{align}
        &\bw^{k+1}= \beta \mathbf{G}_n \bm^k - \beta^2 \bm^{k-1} d_k, \qquad d_k = \frac{1}{n} \sum_{i=1}^{n}\partial_{xx} \Phi_{\mu^\star_{\tau_0, 0}}(q, x^k_i) \nonumber \\
        &\bx^{k+1} = \bw^{k+1}+ \bh_1+\bh_2 \nonumber \\
        & \bm^{k} = g(\bx^k)= g_k( \bw^k). \label{eq:amp_iterate}
        \end{align}
\item Output: The estimator of direct effect is given by
    \begin{align}\label{eq:de_amp}
        \widehat{\mathrm{DE}}_M= \frac{1}{n} \sum_{i=1}^{n} \bar \bT_i m^{M}_i.
    \end{align}
\end{enumerate}
\end{flushleft}
\caption{Alg($\bA_n)$}
\label{alg:amp}
\end{algorithm} 

Algorithm~\ref{alg:amp} may be extended to also estimate the limiting indirect effect $\mathrm{IE}_{\infty}$. To this end, recall the definition of $\widehat{\mu}^\star_{\tau_0,\gamma}$ from~\eqref{eq:define_mustar}. Define the function $\overline{g}(x)= \partial_x\Phi_{\widehat{\mu}^\star_{\tau_0, 0}}(\widehat q,x)$, where $\widehat q= \inf (\rm{supp}(\widehat{\mu}^\star_{\tau_0, 0})) \in [0,1)$. Then, one computes the iterates~\eqref{eq:g_k_define} and~\eqref{eq:amp_iterate} with $\overline{g}$ and $\mathbf{h}_1= - \tau_0 \mathbf{1}$. Denote the resulting output as $\overline{\mathbf{m}}^{[M]}$. Our estimator for the indirect effect is given by
\begin{align}\label{eq:ie_amp}
        \widehat{\mathrm{IE}}_M= \frac{1}{n} \sum_{i=1}^{n} m^{M}_i - \frac{1}{n} \sum_{i=1}^{n} \overline{m}^{M}_i  - \widehat{\mathrm{DE}}_M,
    \end{align}
where $\widehat{\mathrm{DE}}_M$ is defined as~\eqref{eq:de_amp}. 

Our next result establishes formal guarantees for the accuracy of the estimators $\widehat{\mathrm{DE}}_M$ and  $\widehat{\mathrm{IE}}_M$.  Our algorithms will work on typical realizations of the interaction matrix $\bA_n$. To formalize this notion, we introduce the following definition.


\begin{defn}
    Fix $n \geq 1$. Let $\{X_{n,M} : M \geq 1\}$ be a sequence of random variables measurable with respect to $\mathbf{A}_n$. We say that $X_{n,M} \stackrel{\mathscr{P}_{n,M}}{\longrightarrow} 0$ if there exists a deterministic sequence $\{ \varepsilon_{n,M} : M \geq 1\}$ satisfying  $\varepsilon_{n,M} \geq 0$, 
    \begin{align*}
        \lim_{M \to \infty} \lim_{n \to \infty} \varepsilon_{n,M} = 0 
    \end{align*}
    such that 
    \begin{align*}
        \lim_{M \to \infty} \lim_{n \to \infty} \mathbb{P}[|X_{n,M}| > \varepsilon_{n,M}] = 0. 
    \end{align*}
    In the display above, $\mathbb{P}[\cdot]$ refers to the randomness with respect to $\mathbf{A}_n$. 
\end{defn}
Armed with this notion of convergence, we turn to our main result for Gaussian interaction
matrices $\bA_n$.
\begin{theorem}\label{thm:amp}
    Suppose the interaction matrix $\bA_n$ satisfies Assumption~\eqref{assn:interaction_gaussian}. Consider the estimators $\widehat{\mathrm{DE}}_M$ and $\widehat{\mathrm{IE}}_M$ given by~\eqref{eq:de_amp} and~\eqref{eq:ie_amp} respectively.
    Then
    \begin{align*}
       \Big| \ee_{\bar{\bT},\bar{\bX}} [\widehat{\mathrm{DE}}_M] - \mathrm{DE} \Big| \stackrel{\mathscr{P}_{n,M}}{\longrightarrow} 0,\,\,\,\,\,\, \Big| \ee_{\bar{\bT},\bar{\bX}} [\widehat{\mathrm{IE}}_M] - \mathrm{IE} \Big| \stackrel{\mathscr{P}_{n,M}}{\longrightarrow} 0. 
    \end{align*}
\end{theorem}

Recalling Definition~\ref{def:consistent_uniform}, we note that this AMP algorithm facilitates consistent uniform estimation of the treatment effects $\mathrm{DE}$ and $\mathrm{IE}$ in this setting.

\subsection{Parameter estimation}\label{sec:parameter_estimation}
Algorithm ~\ref{algo:de_method} assumes oracle knowledge of the underlying model parameters $\tau_0$ and $\boldsymbol{\theta}_0$. In practice, these parameters should be estimated from the data. Here we use the pseudo-likelihood based estimators introduced in \cite{bhattacharya2024causal}. 
Formally, given $(\bY,\bT,\bX)$, the pseudo-likelihood estimator of the parameters $(\tau_0,\btheta_0)$ is defined as
\begin{align}
    (\hat \tau_{\text{MPL}}, \hat {\boldsymbol{\theta}}_{\text{MPL}}) &=\argmax_{\tau,\boldsymbol{\theta}} \prod_{i=1}^{n} f(\bY_i| \bY_{-i},\bT,\bX). \label{eq:mple_outcome}  
\end{align}
as long as the maximizers in the above display are unique. We assume that the treatment assignments follow the model \eqref{eq:prop_score}. It is known~\cite[Theorem 2.3]{bhattacharya2024causal} that $(\hat \tau_{\text{MPL}}, \hat {\boldsymbol{\theta}}_{\text{MPL}})$ are $\sqrt{n}$-consistent as long as $\|\bA_n\|,\|\mathbf{M}_n\| =O(1)$. In turn, our next result establishes a stability result for the causal effect estimates, and furnishes fully data driven estimators for the causal effects.  

\begin{theorem}\label{lem:mpl_plug_in}
 Assume that $\| \mathbf{M}_n\| = O(1)$ and Assumption~\ref{assn:interaction} holds. Recall the definitions of  $\widehat{\mathrm{DE}}_{(r,\delta)}, \widehat{\mathrm{IE}}_{(r,\delta)}$ from~\eqref{eq:de_hat_define} and~\eqref{eq:ie_hat_define} respectively. Then for any $\varepsilon>0$, we have
    \begin{align*}
       & \lim\limits_{n \rightarrow \infty} \bP \left(\left|\widehat{\mathrm{DE}}_{(r,\delta)}(\tau_0, \btheta_0) -\widehat{\mathrm{DE}}_{(r,\delta)}(\hat \tau_{\mathrm{MPL}}, \hat {\boldsymbol{\theta}}_{\mathrm{MPL}})\right| > \varepsilon \right)=0, \\
        &\lim\limits_{n \rightarrow \infty} \bP \left(\left|\widehat{\mathrm{IE}}_{(r,\delta)}(\tau_0, \btheta_0) -\widehat{\mathrm{IE}}_{(r,\delta)}(\hat \tau_{\mathrm{MPL}}, \hat {\boldsymbol{\theta}}_{\mathrm{MPL}})\right| > \varepsilon \right)=0.
    \end{align*}
The same conclusion holds if we replace $\widehat{\mathrm{DE}}_{(r,\delta)}, \widehat{\mathrm{IE}}_{(r,\delta)}$ by $\widehat{\mathrm{DE}}_M$ and $\widehat{\mathrm{IE}}_M$ obtain via~\eqref{eq:de_amp} and~\eqref{eq:ie_amp} respectively.
\end{theorem}

\section{Discussion and Future directions}
\label{sec:discussions}
We discuss follow up questions arising from our results, and collect primary thoughts regarding
their resolution.
\begin{enumerate}
    \item[(i)] High-dimensional covariates: We assume throughout that covariates $\bX_i \in \mathbb{R}^d$ are i.i.d., compactly supported, with fixed dimension $d$. It would be interesting to investigate extensions where $d= d(n) \rightarrow \infty$. In classical high-dimensional regimes, $\btheta$ is typically assumed sparse and estimated via $\ell_1$-regularized methods~\cite{mukherjee2024logistic}. Understanding how such ideas might adapt to our setting is an open question. We also assume the covariate distribution $\mathbb{P}_X$ is known. If $\mathbb{P}_{X}$ is unknown, one may estimate it from the data and use a plug-in estimate for the treatment effect. We refer to ~\cite{bhattacharya2024causal} for an analysis of this plug-in estimate.
    \item[(ii)] Higher-order interaction:  Our outcome regression model~\eqref{eq:define_gibbs} focuses on quadratic interactions given by $\bA_n$. A natural next step is to incorporate higher-order Markov random fields, such as tensor Ising models~\cite{akiyama2019phase,sasakura2014ising}. Parameter estimation for these models~\cite{liu2024tensor,mukherjee2022estimation} 
    presents several challenges beyond existing works. Extending our results to tensor interactions remains an appealing direction for future research. 
    \item[(iii)] Universality: For gaussian interaction matrices, we characterize the limiting causal effects using the Parisi formula. The corresponding causal effect estimation algorithm is based on Approximate Message Passing (AMP). The limit of the log-partition function and AMP dynamics are both known to exhibit universality to the distribution of the entries of the interaction matrix ~\cite{bayati2015universality,chen2021universality,dudeja2023universality,chatterjee2005simple,carmona2006universality}. Using these universality results, our results for the Gaussian case extend immediately to symmetric i.i.d. interaction matrices with matching means and variances and sufficiently light tails (e.g. sub-Gaussian).

    \item[(iv)] Sparse interaction matrices: In many applications, the interactions among the study units are sparse e.g. the interaction graph might have bounded maximum degree. Extending our results to sparse interaction matrices is an exciting direction for future research. However, we expect that this will require fundamentally new ideas. The Belief Propagation algorithm~\cite{dembo2013factor,murphy2013loopy} could be useful in computing the low-dimensional marginals under sparse interactions. However, translating these ideas into consistent causal effect estimation is non-trivial, and beyond the scope of current techniques.

    \item[(v)] General outcomes: We assume $\by \in \{-1,1\}^n$ throughout, but our methods extend naturally to general bounded outcomes. For dense interaction matrices $\bA_n$ (Assumption~\ref{assn:interaction}), one can discretize $\by_i$'s prior to applying the Regularity Lemma. The resulting analogue~\eqref{eq:gibbs_simplify} will become more involved for general discrete-valued outcomes, but the overall approach remains valid. For Gaussian interaction matrices, the limiting free energy for general bounded outcomes is characterized in~\cite{panchenko2005free}. It should be possible to construct an AMP algorithm similar to Algorithm~\ref{alg:amp} for general  bounded outcomes. We omit this extension to reduce the notational overhead.

    \item[(vi)] Uncertainty quantification:  Earlier work by the authors~\cite{bhattacharya2024causal} proposed a parametric bootstrap method for construction of confidence intervals for treatment effects. Guarantees for the method were derived under the assumption that  $\|\textrm{Cov}(\by)\|=O_{\mathbb{P}}(1)$; this is expected to hold exclusively at high-temperature. 
    Our main focus in this work is to go beyond the high-temperature regime---uncertainty quantification for the causal effects will require substantially new ideas in this regime.

    \item[(vii)] Model misspecification: Our results hinge crucially on the assumption that the outcome model \eqref{eq:define_gibbs} is well-specified. Causal inference under interference with possible model misspecification is an important area of current research. It would be interesting to see if our current ideas could be extended to tackle this challenging problem.   
\end{enumerate}

\section{Proofs}
\label{sec:proofs}
We prove our main results in this section. The proofs of some intermediate results are deferred to the Appendix. We establish Theorems \ref{thm:estimation_hard}, \ref{thm:zn_limit}, \ref{thm:de_consistency},  
\ref{cor:de_answer}, \ref{thm:amp} and \ref{lem:mpl_plug_in} in Sections ~\ref{proof:estimation_hard}, \ref{sec:zn_limit_proof}, \ref{sec:de_consistency_proof}, \ref{proof:sk_limit}, \ref{proof:amp} and \ref{sec:estimation_proof} respectively.       

\subsection{Proof of Theorem~\ref{thm:estimation_hard}}
\label{proof:estimation_hard}

\begin{proof}[Proof of Theorem~\ref{thm:estimation_hard}]
Recall the normalization constant $Z_n$ from \eqref{eq:define_Z}. We set $\boldsymbol{\theta}_0=0$ and denote the normalization constant as $Z_n(\tau)$ to emphasize the dependence on $\tau$. Our proof proceeds by contradiction -- if there exists a polynomial time algorithm $\widehat{\mathrm{DE}}$ for the direct effect, we will show that one can approximate $\frac{1}{n} \log Z_n(0)$ to arbitrary accuracy. However, in \cite[Theorem 1.2]{kunisky2024optimality}, the author establishes a reduction from such an approximation scheme to Conjecture~\ref{conj:bkw}. This will conclude the proof.  


To this end, we split the proof into the following two steps: first, we will show that $\frac{1}{n} \log Z_n(\tau)$ can be approximated well for $\tau \geq 0$ sufficiently large. Second, we will show if $\frac{1}{n} \log Z_n(\tau)$ is approximable and~\eqref{eq:de_approximable} holds then there exists a polynomial time algorithm which approximates $\frac{1}{n} \log Z_n(0)$. 

\noindent \textbf{(i) Approximation of $\frac{1}{n} \log Z_n(\tau)$ for large $\tau$:}
We will show that the quantity
\begin{align}\label{eq:define_phi}
    \phi(\tau):= \frac{1}{n} \log \sum_{\by\in \{\pm 1 \}^n} \exp\Big(\frac{1}{2}\,\mathbf{t}^\top \mathbf{A}_n \mathbf{t}+ \tau \mathbf{y}^\top  \mathbf{t} \Big) = \frac{1}{2n} \mathbf{t}^{\top}\bA_n \mathbf{t} + \frac{1}{n} \sum_{i=1}^{n} \log (2 \cosh{(\tau t_i)})
\end{align}
approximates $\frac{1}{n}\log Z_n(\tau)$ for $\tau$ sufficiently large. 
In addition,
\begin{align*}
    &\frac{1}{n}\log Z_n(\tau) -\phi (\tau) \nonumber\\ 
    &= \frac{1}{n}\log \frac{\sum_{\by\in \{\pm 1 \}^n} \exp\Big(\frac{1}{2}\,\mathbf{y}^\top \mathbf{A}_n \mathbf{y}+ \tau \mathbf{y}^\top  \mathbf{t} \Big)}{\sum_{\by\in \{\pm 1 \}^n} \exp\Big(\tau \mathbf{y}^\top  \mathbf{t} \Big)} - \frac{1}{2n} \mathbf{t}^\top \mathbf{A}_n \mathbf{t} \\
    &= \frac{1}{n}\log \left\langle e^{\frac{1}{2}\mathbf{y}^\top \mathbf{A}_n \mathbf{y} } \right\rangle - \frac{1}{2n} \mathbf{t}^\top \mathbf{A}_n \mathbf{t} \\
    &= \frac{1}{n} \log \left\langle e^{\frac{1}{2}\mathbf{y}^\top \mathbf{A}_n \mathbf{y} } \mathbbm{1}_{\{\frac{1}{n} \by^\top \bt \ge 1-\kappa\}} + e^{\frac{1}{2}\mathbf{y}^\top \mathbf{A}_n \mathbf{y} } \mathbbm{1}_{\{\frac{1}{n} \by^\top \bt < 1-\kappa\}} \right\rangle - \frac{1}{2n} \mathbf{t}^\top \mathbf{A}_n \mathbf{t} \\
    &=  \Big(\frac{1}{n} \log \left\langle e^{\frac{1}{2}\mathbf{y}^\top \mathbf{A}_n \mathbf{y} } \mathbbm{1}_{\{\frac{1}{n} \by^\top \bt \ge 1-\kappa\}}\right\rangle - \frac{1}{2n} \mathbf{t}^\top \mathbf{A}_n \mathbf{t} \Big) + \frac{1}{n} \log \left[1+ \frac{\left\langle e^{\frac{1}{2}\mathbf{y}^\top \mathbf{A}_n \mathbf{y} } \mathbbm{1}_{\{\frac{1}{n} \by^\top \bt < 1-\kappa\}}\right \rangle}{\left\langle e^{\frac{1}{2}\mathbf{y}^\top \mathbf{A}_n \mathbf{y} } \mathbbm{1}_{\{\frac{1}{n} \by^\top \bt \ge 1-\kappa\}}\right \rangle} \right] \nonumber \\
    &:= T_1 + T_2
\end{align*}
for some fixed $\kappa >0$. In the display above, the notation $\langle \cdot \rangle$ denotes the expectation w.r.t a product measure on $\{\pm 1\}^n$ with means $\tanh(\tau t_i)$, $i=1,\ldots,n$. Next we will show that $T_1, T_2 \rightarrow 0$ as $n \rightarrow \infty$, followed by $\tau \rightarrow \infty$ and $\kappa \rightarrow 0$. To show $T_1 \rightarrow 0$, note that if $\frac{1}{n} \by^\top \bt \ge 1-\kappa$, then $\frac{1}{n}\|y-t\|^2_2 \le 2\kappa$ and 
\begin{align}\label{eq:y_upper}
    e^{\frac{1}{2}\mathbf{y}^\top \mathbf{A}_n \mathbf{y}}&= e^{\frac{1}{2} \bt^\top \bA_n \bt + \bt^\top \bA_n (\by-\bt)+ \frac{1}{2}(\by-\bt)^\top \bA_n (\by-\bt)} \nonumber\\
    & \le e^{\frac{1}{2} \bt^\top \bA_n \bt + \|\bA_n\| \sqrt{2 \kappa} n+ n \kappa\|\bA_n\|}.  
\end{align}
Therefore we have
\begin{align*}
    T_1 &\le \frac{1}{n} \log \left \langle e^{\frac{1}{2} \bt^\top \bA_n \bt + \|\bA_n\| \sqrt{2\kappa} n+ n \kappa\|\bA_n\|} \mathbbm{1}_{\{\frac{1}{n} \by^\top \bt \ge 1-\kappa\}} \right\rangle- \frac{1}{2n} \mathbf{t}^\top \mathbf{A}_n \mathbf{t} \\
    & \le \frac{1}{n} \log \left \langle e^{\frac{1}{2} \bt^\top \bA_n \bt + \|\bA_n\| \sqrt{2 \kappa} n+ n \kappa\|\bA_n\|} \right\rangle- \frac{1}{2n} \mathbf{t}^\top \mathbf{A}_n \mathbf{t}\\
    &= \|\bA_n\| \Big( \sqrt{2 \kappa} + \kappa \Big).
\end{align*}
To show a lower bound on $T_1$, similar to~\eqref{eq:y_upper}, we observe
\begin{align}\label{eq:y_lower}
    e^{\frac{1}{2}\mathbf{y}^\top \mathbf{A}_n \mathbf{y}}& \ge e^{\frac{1}{2} \bt^\top \bA_n \bt - \|\bA\| \sqrt{2 \kappa} n - \kappa\|\bA_n\| n }. 
\end{align}
Hence 
\begin{align*}
    T_1 &\ge \frac{1}{n} \log \left \langle e^{\frac{1}{2} \bt^\top \bA_n \bt - \|\bA_n\| \sqrt{2 \kappa} n - n \kappa \|\bA_n\|} \mathbbm{1}_{\{\frac{1}{n} \by^\top \bt \ge 1-\kappa\}} \right\rangle- \frac{1}{2n} \mathbf{t}^\top \mathbf{A}_n \mathbf{t} \\
    &= \frac{1}{n} \Bigg[\frac{1}{2} \bt^\top \bA_n \bt - \|\bA_n\| \sqrt{2 \kappa} n - n \kappa \|\bA_n\|\Bigg] +\frac{1}{n} \log \left \langle \mathbbm{1}_{\{\frac{1}{n} \by^\top \bt \ge 1-\kappa\}} \right\rangle- \frac{1}{2n} \mathbf{t}^\top \mathbf{A}_n \mathbf{t}\\
    & \ge \frac{1}{n} \log \left \langle \mathbbm{1}_{\{\frac{1}{n} \by^\top \bt \ge 1-\kappa\}} \right\rangle  -\|\bA_n\| \sqrt{2 \kappa} - \kappa \|\bA_n\| \\
    &\ge \frac{1}{n}\log(1- e^{-c(\tau,\kappa) n}) -\|\bA_n\| \Big(\sqrt{2 \kappa} + \kappa \Big)\\
    &\ge -\frac{1}{n} \log 2 - \|\bA_n\| \Big(\sqrt{2 \kappa} + \kappa \Big), 
\end{align*}
for any large $\tau$ where the third inequality is due to Lemma~\ref{lem:bennett}. Therefore, we obtain that
\[
\lim\limits_{\kappa \rightarrow 0} \lim\limits_{\tau \rightarrow \infty} \lim\limits_{n \rightarrow \infty} |T_1| =0.
\]
Now we turn to the proof of $T_2 \rightarrow 0$. Note that 
\begin{align*}
    \left \langle e^{\frac{1}{2}\mathbf{y}^\top \mathbf{A}_n \mathbf{y} } \mathbbm{1}_{\{\frac{1}{n} \by^\top \bt < 1-\kappa\}} \right \rangle \le e^{\frac{1}{2}n \|\bA_n\|} \left \langle  \mathbbm{1}_{\{\frac{1}{n} \by^\top \bt < 1-\kappa\}} \right \rangle  \le e^{n(\frac{1}{2} \|\bA_n\|-c(\tau,\kappa))},
\end{align*}
where the last inequality is due to Lemma~\ref{lem:bennett} and $c(\tau,\kappa)\rightarrow \infty$ as $\tau \rightarrow \infty$. Moreover,
\begin{align*}
     \left \langle e^{\frac{1}{2}\mathbf{y}^\top \mathbf{A}_n \mathbf{y} } \mathbbm{1}_{\{\frac{1}{n} \by^\top \bt \ge 1-\kappa\}} \right \rangle &\ge e^{-\frac{n}{2} \|\bA_n\|} \left \langle \mathbbm{1}_{\{\frac{1}{n} \by^\top \bt \ge 1-\kappa\}} \right \rangle \\
     & \ge e^{-\frac{n}{2} \|\bA_n\|} (1-e^{-c(\tau,\kappa)n}) \\
     &\ge \frac{1}{2} e^{-\frac{n}{2} \|\bA_n\|},
\end{align*}
using Lemma~\ref{lem:bennett}. By the above two displays,
\begin{align*}
    T_2 \le \frac{1}{n} \log \left[ 1 + 2 e^{n \big(\|\bA_n\|-c(\tau,\kappa)\big)}\right] \le \frac{1}{n} \log 3 \rightarrow 0,
\end{align*}
since $n(\|\bA_n\|-c(\tau,\kappa))<0$ for large $\tau$. This implies that
\begin{equation}
\lim\limits_{\tau \rightarrow \infty} \lim\limits_{n \rightarrow \infty} \sup_{\mathbf{t} \in \{ \pm 1\}^n}\left| \frac{1}{n}\log Z_n (\tau) -\phi (\tau) \right| =0. \nonumber 
\end{equation}

In turn, this implies 
\begin{align}
     \lim_{\tau \to \infty} \lim_{n \to \infty} \Big| \frac{1}{n} \mathbb{E}_{\bar{\mathbf{T}}}[\log Z_n (\tau)] - \mathbb{E}_{\bar{\mathbf{T}}}[\phi(\tau)] \Big| = 0, \nonumber 
\end{align}
where $\bar{\mathbf{T}} \sim \mathrm{Unif}(\{\pm 1\}^n)$. Consequently, we can approximate $\frac{1}{n} \mathbb{E}_{\bar{\mathbf{T}}}[\log Z_n(\tau)]$ using $\mathbb{E}_{\bar{\mathbf{T}}}[\phi(\tau)]$, which can be computed explicitly. 

\noindent \textbf{(ii) Approximation of $\frac{1}{n} \log Z_n(0)$:}
Using~\eqref{eq:simplify_de}, we obtain 
\begin{equation}
    \mathrm{DE}(\tau)= \frac{2}{n} \frac{\partial}{\partial \tau'} \mathbb{E}_{\bar{\mathbf{T}}}[\log  Z_n(\tau')]\vert_{\tau'=\tau}.
\end{equation}
Therefore, for any $\tau>0$
\begin{align*}
    \frac{1}{n} \mathbb{E}_{\bar{\mathbf{T}}} [\log Z_n(\tau)] = \frac{1}{n} \log Z_n(0)+ \frac{1}{2} \int_{0}^{\tau} \mathrm{DE}(\tau') d\tau'. \nonumber 
\end{align*}
Suppose for some $\tau$ and $\eta>0$, there exists an estimator $\widehat{\mathrm{DE}}$ such that
\begin{align}\label{eq:de_tau}
   \mathbb{P}\Big( \sup_{\tau' \in[0,\tau]} \left|\widehat{\mathrm{DE}}(\tau')- \mathrm{DE}(\tau')\right|<\eta \Big) = 1-o(1).
\end{align}
Since $\mathrm{DE}$ is continuous in $\tau'$, for any $\eta>0$ there exists $M \in  \mathbb{N}$ and a partition $0=\tau_1 <\tau_2 <\ldots < \tau_M= \tau$ such that
\begin{align*}
    \left|\int_{0}^{\tau} \mathrm{DE}(\tau') d\tau' -\frac{1}{M}\sum_{k=1}^{M} \mathrm{DE}(\tau_k) \right| < \eta.
\end{align*}
Define 
\begin{align}\label{eq:define_psi}
    \psi (\tau) = \mathbb{E}_{\bar{\mathbf{T}}}[\phi (\tau) ] - \frac{1}{M}\sum_{k=1}^{M} \widehat{\mathrm{DE}}(\tau_k),
\end{align}
where $\phi (\tau) $ is defined in~\eqref{eq:define_phi}. Then we have
\begin{align*}
    \left| \frac{1}{n} \log Z_n(0)- \psi(\tau) \right| &= \left| \Big(\frac{1}{n} \mathbb{E}_{\bar{\bT}}[\log Z_n(\tau)]- \mathbb{E}_{\bar{\mathbf{T}}}[\phi(\tau)] \Big) + \frac{1}{M}\sum_{k=1}^{M} \Big( \widehat{\mathrm{DE}}(\tau_k)-\mathrm{DE}(\tau_k)\Big) \right| \\
    & \le  \left| \frac{1}{n}\mathbb{E}_{\bar{\bT}}[\log Z_n (\tau)] -\mathbb{E}_{\bar{\mathbf{T}}}[\phi(\tau)] \right|+ \sup_{\tau' \in[0,\tau]} \left|\widehat{\mathrm{DE}}(\tau')- \mathrm{DE}(\tau')\right| \\
    & \le 2\eta,
\end{align*}
for large $n,\tau$ using~\eqref{eq:zn_limit} and~\eqref{eq:de_tau}. Using~\cite[Theorem 1.2]{kunisky2024optimality} with $\delta=2 \eta$, we can design a hypothesis test which violates Conjecture~\ref{conj:bkw}. This completes the proof.

\end{proof}

Finally we prove the auxilliary concentration lemma used in the proof.
\begin{lemma}\label{lem:bennett}
    Fix $0<\kappa<1$ and $\bt \in \{ \pm 1\}^n$. Suppose $y_i$'s are independent $\{\pm 1\}$ valued random variables such that $\mathbb{E}(y_i)= \tanh(t_i \tau_i)$. Define the set $\mathcal{A}= \left\{\frac{1}{n} \by^\top \bt < 1-\kappa \right\}$. Then we have
    \begin{align*}
        \lim\limits_{\tau \rightarrow \infty} \lim\limits_{n \rightarrow \infty} \frac{1}{n} \log \mathbb{P}(\mathcal{A}) = -\infty.
    \end{align*}
\end{lemma}
\begin{proof}
    We invoke Bennett's inequality~\cite[Theorem 2.9]{boucheronbook} for independent random variables. Define \begin{align*}
        V:= \text{Var}(\bt^\top \by)= \sum_{i=1}^n \text{Var} (y_i) =n \sech^2(\tau),
    \end{align*}
    since $\sech$ is symmetric and $t_i \in \{\pm 1\}$. Further,
    \begin{align*}
        \mathbb{P}(\mathcal{A})&= \mathbb{P}\left(\sum_{i=1}^n t_i(y_i- \tanh(\tau t_i)) < n \Big[1-\kappa - \frac{1}{n}\sum_{i=1}^{n} t_i\tanh(\tau t_i)) \Big]\right) \\
        &= \mathbb{P}\left(\sum_{i=1}^n t_i(y_i- \tanh(\tau t_i)) < n \Big[1-\kappa - \tanh (\tau) \Big]\right) \\
        & \le \mathbb{P}\left(\sum_{i=1}^n t_i(y_i- \tanh(\tau t_i)) <  - \frac{n \kappa}{2}\right),
    \end{align*}
    for large $\tau >0$. Define $h(u)= (1+u)\log(1+u)-u$ for $u \ge 0$. Since $|t_i(y_i- \tanh(\tau t_i))| \le 2$, we obtain by Bennett's inequality that
    \begin{align*}
        \mathbb{P}(\mathcal{A}) \le \exp \left(- \frac{V}{4} h\Big(\frac{n \kappa}{V}\Big)\right) = \exp \left(- \frac{ n \sech^2(\tau)}{4} h\Big(\frac{ \kappa}{\sech^2(\tau)}\Big)\right) =: \exp \left( - n c(\tau,\kappa)\right),
    \end{align*}
    where we defined $c(\tau,\kappa)= \frac{\sech^2(\tau)}{4} h\Big(\frac{ \kappa}{\sech^2(\tau)}\Big)$. Further for any $\kappa >0$, we have  
    \[\lim\limits_{x \rightarrow \infty} \frac{h(\kappa x)}{x} = \lim\limits_{x \rightarrow \infty} \frac{(1+\kappa x) \log (1+\kappa x) - \kappa x}{x}= \infty. \]
Hence for any $\kappa>0$, $c(\tau,\kappa) \rightarrow \infty$ as $\tau \rightarrow \infty$. Therefore,
\[
\lim\limits_{n \rightarrow \infty}\mathbb{P}(\mathcal{A})\le -c(\tau,\kappa) \rightarrow -\infty
\]
as $\tau \rightarrow \infty$. This completes the proof of the Lemma.
\end{proof}

\subsection{Proof of Theorem~\ref{thm:zn_limit}}
\label{sec:zn_limit_proof}

\begin{lemma}\label{lem:efron_stein}
Recall the definition of $\tilde Z_n$ from~\eqref{eq:extra_gamma} and $\bar T_i \sim  \mathrm{Unif}(\pm 1)$ i.i.d., $\bar \bX_i \sim \mathbb{P}_{\bX}$ i.i.d. Then we have, 
    \begin{align*}
        \mathrm{Var}_{\bar\bT, \bar\bX}(\log \tilde Z_n (\bar{\bT}, \bar{\bX}))= O(n), \quad  \mathrm{Var}_{\bar\bX}(\log \tilde Z_n (- \mathbf{1}, \bar{\bX}))= O(n).
    \end{align*}
\end{lemma}
\begin{proof}
Using Efron-Stein inequality~\cite{boucheronbook}, 
\begin{align}
    &\mathrm{Var}_{ \bar\bT, \bar\bX}(\log \tilde Z_n (\bar{\bT}, \bar{\bX}) )\nonumber \\ 
    &\leq \frac{1}{2}\mathbb{E}_{ \bar\bT,   \bar\bX} \Big[ \sum_{i=1}^{n} (\log \tilde Z_n (  \bar\bT,   \bar\bX) - \log \tilde Z_n(  \bar\bT^{(i)},  \bar \bX))^2 + \sum_{i=1}^{n} (\log \tilde Z_n (  \bar\bT,  \bar\bX) - \log \tilde Z_n(  \bar\bT,   \bar\bX^{(i)}))^2 \Big], \nonumber \end{align}
where $  \bar\bT^{(i)} = (  \bar T_1, \cdots,   \bar T_{i-1},   \bar T_i',   \bar T_{i+1}, \cdots,   \bar T_n)$, and $  \bar\bT' = ( \bar T_1', \cdots,   \bar T_n')$ are i.i.d. $\mathrm{Unif}(\{\pm 1\} )$ independent of $\bar{\bT}$. Similarly, $ \bar\bX^{(i)}=(  \bar\bX_1, \cdots,   \bar\bX_{i-1},   \bar\bX_i',   \bar\bX_{i+1}, \cdots,   \bar\bX_n)$, $  \bar\bX'= (  \bar\bX_1', \cdots,   \bar\bX_n')$ are i.i.d. $\mathbb{P}_{X}$ independent of $\bar{\bX}$. The proof of the Lemma follows once we establish that each term in the display above is $O(n)$. Without loss of generality, we work with the first term. The bound for the second term is similar, and thus omitted. 

The proof is by interpolation. Fix $1\leq i \leq n$. For $\upsilon \in [0,1]$, set 
\begin{align}
    e^{H(\upsilon)} = \int_{\by \in [-1,1] ^n}\exp\Big( \mathbf{y}^{\top} \mathbf{A}_n \mathbf{y} + \tau_0 \sum_{j\neq i} y_j  \bar T_j + \tau_0 y_i ((1-\upsilon)  \bar {T}_i + \upsilon   \bar T_i') + \mathbf{y}^{\top} (\bar \bX\mathbf{\theta_0} +\gamma \mathbf{1})\Big)\prod_{i=1}^{n}d\mu (y_i). \nonumber   
\end{align}
This implies, 
\begin{align}
    \log \tilde Z_n ( \bar \bT,  \bar\bX) - \log \tilde Z_n(  \bar\bT^{(i)},  \bar \bX) = \int_{0}^{1} \frac{\partial}{\partial \upsilon} H(\upsilon) \mathrm{d}\upsilon. \nonumber   
\end{align}
By direct computation, we obtain that for all $\upsilon \in [0,1]$, $|\frac{\partial}{\partial \upsilon} H(\upsilon)| \leq \tau_0$. In turn, this directly implies 
\begin{align}
    \sum_{i=1}^{n} \mathbb{E}_{ \bar\bT,  \bar \bX}(\log \tilde Z_n (  \bar\bT,   \bar\bX) - \log \tilde Z_n(  \bar\bT^{(i)} ,   \bar\bX))^2 \leq n \tau_0^2. 
\end{align}
Hence, $\text{Var}_{  \bar\bT,   \bar\bX}\log \tilde Z_n( \bar\bT,   \bar\bX) =O(n)$.  
Upon setting $\bar\bT= - \mathbf{1}$, the proof of the second part of the lemma proceeds analogously to the argument above.
\end{proof}

\begin{proof}[Proof of Theorem~\ref{thm:zn_limit}]

Under Assumption~\ref{assn:interaction}, we have $\text{Tr}(\bA^2_n)= o(n)$. Additionally, the parameter space is bounded. Consequently, we obtain~\cite{basak2017universality},~\cite[Theorem 1.6]{chatterjee2016nonlinear}
\begin{equation}\label{eq:cha_dembo}
    \sup_{\tau, \btheta, \bt, \bx, \gamma} \frac{1}{n} \Big\vert \log \tilde Z_n (\bt, \bx)- \sup_{\bv \in [-1,1]^n} \TT(\bv)\Big\vert \rightarrow 0,
\end{equation}
where $\TT$ is defined as:
\begin{equation}
    \TT(\bv):= \frac{1}{2} \bv^\top \bA_n \bv + \sum_{i=1}^n v_i(\tau t_i +\btheta^\top \bx_i + \gamma) - \sum_{i=1}^{n}I(v_i). 
\end{equation}
By definition of the weak-cut convergence (see Definition~\ref{def:defirst}), there exists a sequence of permutations $\{\pi_n\}_{n\ge 1}$ with $\pi_n\in S_n$ such that $$d_\square(W_{n \bA_n^{\pi_n}}, W) \rightarrow 0,\text{ where }\bA_n^{\pi_n}(i,j):=\bA_n(\pi_n(i),\pi_n(j)).$$
Since our proof does not depend on $\pi_n$, we assume $\pi_n(i)=i$, implying $d_\square(W_{n \bA_n}, W) \rightarrow 0$. We will focus on the case $\tau=\tau_0,\btheta=\btheta_0$.

\noindent
\textbf{Upper bound:} We prove the upper bound i.e., given $\bar{\mathbf{T}} \sim \mathrm{Unif}(\{ \pm 1\}^n)$, $\bar{\mathbf{X}} \sim \bP^{\otimes n}_X$,
$$\limsup\limits_{n \rightarrow \infty}\frac{1}{n} \mathbb{E}_{\bar{\bT}, \bar{\bX}}[\log \tilde Z_n (\bar{\bT}, \bar{\bX}) ]\le \sup_{F\in \mathcal{F}} G_{W,\tau_0,\btheta_0,\gamma}(F).$$
Fix $\bv \in [-1,1]^n$. Let $U \sim U(0,1)$ independent of $\bar T_i, \bar{\bX}_i$s. If $U \in (\frac{i-1}{n},\frac{i}{n}]$, set $V=v_i$, $\bar T= \bar T_i$, $\bar{\bX}=\bar{\bX}_i$, $i =1,\ldots, n$. Call $\mathcal{L}_n$ the joint probability distribution of $(U,V,\bar T,\bar\bX)$. Let $(U_j,V_j,\bar T_j,\bar \bX_j)$, $j=1,2$ be two i.i.d. samples from $\mathcal{L}_n$. Then, we have, by definition,
\begin{align}
    \frac{1}{n} \TT(\bv)&= \bE_{\mathcal{L}_n}(W_{n\bA_n}(U_1,U_2)V_1V_2)+ \bE_{\mathcal{L}_n}(V_1 (\tau_0 \bar T_1+ \btheta^\top_0 \bar \bX_1+\gamma)) -\bE_{\mathcal{L}_n}(I(V_1)), \nonumber
\end{align}
where $W_{n\bA_n}$ is defined as in \eqref{def:defirst}. Further, using \eqref{eq:cut_con}, we have
\begin{align}
    &\bE_{\mathcal{L}_n}(W_{n\bA_n}(U_1,U_2)V_1V_2)+ \bE_{\mathcal{L}_n}(V_1 (\tau_0 \bar T_1+ \btheta^\top_0 \bar \bX_1+\gamma)) -\bE_{\mathcal{L}_n}(I(V_1)) \nonumber \\
    &=\bE_{\mathcal{L}_n}(W(U_1,U_2)V_1V_2)+ \bE_{\mathcal{L}_n}(V_1 (\tau_0 \bar T_1+ \btheta^\top_0 \bar \bX_1+ \gamma)) -\bE_{\mathcal{L}_n}(I(V_1))+ \mathcal{R}_n(\mathbf{v}) \nonumber \\
    & =: H(\mathcal{L}_n)+ \mathcal{R}_n(\bv), \label{eq:define_h_ln}
\end{align}
where $\mathcal{R}_n(\bv)$ is a deterministic sequence such that $\sup_{\bv \in [-1,1]^n}|\mathcal{R}_n(\bv)| \rightarrow 0$ as $n \rightarrow \infty$.

Now, let $\mathcal{M}$ be the space of all probability distributions $(U,V, \bar T, \bar \bX)$ such that $U\sim U(0,1)$, $\bar T\sim \mathrm{Unif}(\pm 1)$, $\bar \bX\sim \bP_X$ and they are all independent. Note that, any subsequential limit of $\mathcal{L}_n$ belongs to the set $\mathcal{M}$. Since $I$ is lower semicontinuous, 
\begin{equation}
    \sup_{\bv \in [-1,1]^n} \frac{1}{n}\TT(\bv) \le \sup_{\mathcal{L} \in \mathcal{M}} H(\mathcal{L}).
\end{equation}
Using \eqref{eq:cha_dembo} and  Lemma~\ref{lem:efron_stein}, we obtain
\begin{equation} 
    \limsup\limits_{n \rightarrow \infty}\frac{1}{n} \mathbb{E}_{\bar \bT, \bar\bX}\log \tilde Z_n(\bar\bT, \bar\bX) \le \sup_{\mathcal{L} \in \mathcal{M}} H(\mathcal{L}).
\end{equation}
Finally, for any $\mathcal{L} \in \mathcal{M}$, denote the conditional expectation of $V|U,\bar{T}, \bar \bX$ as $F(U,\bar T, \bar \bX)$. This implies that $F\in \mathcal{F}$. Consequently, using the convexity of $I$ and Jensen's inequality, we have $H(\mathcal{L}) \le G_{W, \tau_0,\btheta_0,\gamma}(F)$. This implies $\sup_{\mathcal{L} \in \mathcal{M}} H(\mathcal{L}) \le \sup_{F \in \mathcal{F}} G_{W, \tau_0,\btheta_0,\gamma}(F)$, completing the proof of the upper bound.

\noindent 
\textbf{Lower bound:}
For any $\varepsilon>0$, there exists $F_\varepsilon \in \mathcal{F}$ such that $$\sup_{F\in \mathcal{F}} G_{W,\tau_0,\btheta_0,\gamma}(F) \le G_{W,\tau_0,\btheta_0,\gamma}(F_\varepsilon)+\varepsilon.$$ 
Define $n$ independent random variables $U_i \sim U(\frac{i-1}{n},\frac{i}{n}]$, $i =1, \ldots, n$. Define $v_i= F_\varepsilon(U_i,\bar T_i,\bar\bX_i)$. If $ \tilde \bv=(v_1,\ldots,v_n)$, then
\begin{align*}
    \frac{1}{n} \TT (\tilde \bv)= \frac{1}{n} \tilde \bv^\top \bA_n \tilde \bv+ \frac{1}{n} \sum_{i=1}^{n} v_i(\tau_0 \bar T_i +\btheta^\top_0 \bar \bX_i+\gamma)- \frac{1}{n} \sum_{i=1}^{n} I(v_i).
\end{align*}
The second and third summand above converges, in probability, to $\bE {F_\varepsilon(U,\bar T,\bar\bX)(\tau_0 \bar T+\btheta^\top_0 \bar \bX+\gamma)}$ and $\bE (I(F_\varepsilon(U, \bar T,\bar \bX)))$ respectively, where $U \sim U(0,1)$, $\bar T\sim \mathrm{Unif}(\pm 1)$, $\bar \bX\sim \bP_X$ and they are all independent. Here, we have used the fact that $I$ is bounded and continuous function. Also, since $U_i$'s are independent, so are $v_i$s. Since $n|\bA_{n}(i,j)| \le 1$ by Assumption~\ref{assn:interaction}, by a direct variance calculation, $\frac{1}{n} \tilde \bv^\top \bA_n \tilde \bv -\frac{1}{n} \bE (\tilde \bv)^\top \bA_n \bE(\tilde \bv) \xrightarrow{\bP} 0$. This implies
\begin{equation}\label{eq:u_to_f}
    \frac{1}{n} \TT (\tilde \bv) \xrightarrow{\bP} G_{W,\tau_0,\btheta_0,\gamma}(F_\varepsilon).
\end{equation}
Hence, with high probability,
\begin{align*}
    \sup_{F\in \mathcal{F}} G_{W,\tau_0,\btheta_0,\gamma}(F) &\le G_{W,\tau_0,\btheta_0,\gamma}(F_\varepsilon) +\varepsilon\\
    &\le \frac{1}{n} \TT ( \tilde \bv)+2 \varepsilon \le \sup_{\bv \in [-1,1]^n} \frac{1}{n} \TT (\bv)+2 \varepsilon\\
    & \le \liminf\limits_{n \rightarrow \infty}\frac{1}{n} \log \tilde Z_n(\bar \bT,\bar \bX) + 2 \varepsilon,
\end{align*}
where the final inequality is due to~\eqref{eq:cha_dembo}. This completes the proof of \eqref{eq:zn_limit}, since Lemma~\ref{lem:efron_stein} implies that
\begin{equation}\label{eq:zn_conc}
    \frac{1}{n}\log \tilde Z_n(\bar \bT,\bar \bX) - \frac{1}{n} \ee_{\bar \bT, \bar \bX} \log \tilde Z_n(\bar \bT,\bar \bX) \xrightarrow{\bP} 0.
\end{equation}
Further, setting $\gamma=0$, we have,
\begin{align*}
    \mathrm{DE}_{\infty}= \lim\limits_{n \rightarrow \infty} \mathrm{DE} &= \lim\limits_{n \rightarrow \infty} \frac{2}{n}\frac{\partial}{\partial \tau} \mathbb{E}_{\bar{\bT},\bar{\bX}} \log \tilde Z_n(\bar \bT,\bar \bX) \vert _{\tau=\tau_0} \\
    &\rightarrow 2\frac{\partial}{\partial \tau}  \sup_{F\in \mathcal{F}} G_{W,\tau_0,\btheta_0,0}(F)\vert _{\tau=\tau_0},
\end{align*}
as long as the supremum above is differentiable w.r.t. $\tau$ at $\tau= \tau_0$, since $ \log \tilde Z_n$ is a convex function by elementary properties of exponential families. This concludes the proof for direct effects.

Turning to the proof for indirect effects, we can follow the same argument above to show $\frac{1}{n}\ee_{\bar \bT, \bar \bX} \log\tilde Z_n(-\mathbf{1},\bX) \rightarrow \sup_{F\in \mathcal{F}} \tilde G_{W,\tau_0,\btheta_0,\gamma}(F)$. By Lemma~\ref{lem:de_define}, we have
\begin{align*}
    \mathrm{IE}_\infty = \lim\limits_{n \rightarrow \infty} \mathrm{IE} & = \lim\limits_{n \rightarrow \infty} \frac{1}{n} \bE_{  \bT,  \bX} \Big[\sum_{i=1}^{n} \langle \bY_i \rangle \Big]- \lim\limits_{n \rightarrow \infty} \frac{1}{n} \bE_{-\mathbf{1},  \bX} \Big[\sum_{i=1}^{n} \langle \bY_i \rangle \Big]- \frac{1}{2}\mathrm{DE}_\infty \\
    & = \lim\limits_{n \rightarrow \infty} \frac{1}{n}\frac{\partial}{\partial \gamma} \log \tilde Z_n(  \bT,   \bX)\Big\vert_{\gamma=0} + \lim\limits_{n \rightarrow \infty} 
 \frac{1}{n}\frac{\partial}{\partial \gamma} \log \tilde Z_n(\mathbf{-1},   \bX)\Big\vert_{\gamma=0}- \frac{1}{2}\mathrm{DE}_\infty\\
    &= \frac{\partial}{\partial \gamma} \sup_{F\in \mathcal{F}} G_{W,\tau_0, \btheta_0,\gamma}(F)\Big\vert_{\gamma=0}+ \frac{\partial}{\partial \gamma} \sup_{F\in \mathcal{F}} \tilde G_{W,\tau_0, \btheta_0,\gamma}(F)\Big\vert_{\gamma=0} - \frac{1}{2}\text{DE}_{\infty},
\end{align*}
as long as the supremum above is differentiable w.r.t. $\gamma$ at $\gamma=0$. This concludes the proof.

\end{proof}

\subsection{Proof of Theorem~\ref{thm:de_consistency}}
\label{sec:de_consistency_proof}

Our first result establishes that the causal effects $\mathrm{DE}$ and $\mathrm{IE}$ are stable under perturbations of the interaction matrix $\bA_n$. To track the dependence of the causal effects on $\bA_n$  explicitly, we denote them as $\mathrm{DE}^{\bA_n}$ and $\mathrm{IE}^{\bA_n}$ respectively.

\begin{lemma}\label{lem:de_ie_close}
    For any $\varepsilon>0$, there exists $\delta:= \delta(\varepsilon)>0$ such that if $\|\bA_n - \mathbf{B}_n\| < \delta$ then 
        \begin{align}
            |\mathrm{DE}^{\bA_n} - \mathrm{DE}^{\bB_n}| < \varepsilon, \,\,\,\,
            |\mathrm{IE}^{\bA_n} - \mathrm{IE}^{\bB_n}| < \varepsilon. \label{eq:A_stability}
        \end{align}
\end{lemma}
We defer the proof of this lemma to the Appendix. 
\begin{proof}[Proof of Theorem~\ref{thm:de_consistency}]
   We first prove that the estimator $\widehat{\mathrm{DE}}_{(r,\delta)}$ is close to $\mathrm{DE}^{\bA_n}$. 
%
Fix $\delta>0$, to be specified later. Using Lemma~\ref{lem:block_approx}, we obtain $(r,\delta)$-block-approximation of $\bA_n$, denoted by $\widetilde{\bA}_n$, in $O(\delta^{-O(1)} n^2 + n r)$ time. By the defintion of block-approximation, we have $\|\bA_n - \widetilde{\bA}_n\| < \delta$. 
Therefore, by Lemma~\ref{lem:graph_approx_zn} 
\begin{equation}\label{eq:de_closeness}
    |\mathrm{DE}^{\bA_n}- \mathrm{DE}^{\widetilde{\bA}_n}|<\varepsilon.
\end{equation}
Since $\widetilde{\bA}_n$ is a block matrix, by~\eqref{eq:de_hat_define} and~\eqref{eq:define_part_sum_1}, we have
\begin{align*}
     \ee_{\overline{\bT}, \overline{\bX}}[\widehat{\mathrm{DE}}_{(r,\delta)}] &=\frac{2}{n} \sum_{a=1}^{m} \sum_{k=0}^{2^r} (\ee_{\overline{\bT}, \overline{\bX}}[\langle V_{a,k,+} \rangle_{(r,\delta)}]- \ee_{\overline{\bT}, \overline{\bX}}[\langle V_{a,k,-}\rangle_{(r,\delta)}] )\\
     &= \frac{2}{n}\ee_{\overline{\bT}, \overline{\bX}}\Big[ \sum_{a=1}^{m} \sum_{k=0}^{2^r} \langle \sum_{\ell \in  \mathcal{A}_{a,k,+}} y_{\ell} \rangle_{(r,\delta)}-  \sum_{a=1}^{m} \sum_{k=0}^{2^r}\langle \sum_{\ell \in  \mathcal{A}_{a,k,-}} y_{\ell}\rangle_{(r,\delta)}  \Big]\\
     &= \frac{2}{n}\ee_{\overline{\bT}, \overline{\bX}}\Big[ \sum_{a=1}^{m} \sum_{k=0}^{2^r} \langle \sum_{\ell \in  \mathcal{A}_{a,k,+}} \bar{T}_{\ell} y_{\ell} \rangle_{(r,\delta)}+ \sum_{a=1}^{m} \sum_{k=0}^{2^r}\langle \sum_{\ell \in  \mathcal{A}_{a,k,-}} \bar{T}_\ell y_{\ell}\rangle_{(r,\delta)} \Big] \nonumber \\
     &= \frac{2}{n}\ee_{\overline{\bT}, \overline{\bX}}\Big[ \sum_{a=1}^{m} \sum_{k=0}^{2^r} \sum_{\ell \in  \mathcal{A}_{a,k,+}} \langle  \bar{T}_{\ell} y_{\ell} \rangle_{(r,\delta)}+ \sum_{a=1}^{m} \sum_{k=0}^{2^r} \sum_{\ell \in  \mathcal{A}_{a,k,-}} \langle \bar{T}_\ell y_{\ell}\rangle_{(r,\delta)} \Big] \nonumber \\
     &= \frac{2}{n} \mathbb{E}_{\bar{\bT}, \bar{\bX}} \Big[ \sum_{\ell=1}^{n} \langle \bar{T}_{\ell} y_{\ell} \rangle_{(r,\delta)}  \Big]  = \frac{2}{n} \mathbb{E}_{\bar{\bT}, \bar{\bX}} \Big[ \sum_{\ell=1}^{n} \langle \bar{T}_{\ell} \bY_{\ell} \rangle_{(r,\delta)}  \Big] = \mathrm{DE}^{\tilde{\bA}_n}. \nonumber  
\end{align*}
The desired conclusion follows upon combining the previous display with~\eqref{eq:de_closeness}. 
The conclusion $\widehat{\mathrm{IE}}_{(r,\delta)}$ follows directly from Lemma~\ref{lem:de_ie_close} since $\ee_{\overline{\bT}, \overline{\bX}} [\widehat{\mathrm{IE}}_{(r,\delta)}] = \mathrm{IE}^{\widetilde{\bA}_n}$ by a similar argument.  
\end{proof}

\subsection{Proof of Theorem~\ref{cor:de_answer}}
\label{proof:sk_limit}

\begin{proof}[Proof of Theorem~\ref{cor:de_answer}]
We start with the Direct effect $\mathrm{DE}$. Using Lemma~\ref{lem:de_define}, we have that 
\begin{align}
    \mathrm{DE} = \frac{2}{n} \frac{\partial}{\partial \tau}\mathbb{E}_{\bar{\bT}, \bar{\bX}} [ \log \tilde{Z}_n(\bar{\bT}, \bar{\bX})]\Big|_{\tau=\tau_0, \gamma=0}. \nonumber 
\end{align}
Using direct computation, we have that at $\gamma=0$, $\mathbb{E}_{\bar{\bT}, \bar{\bX}}[\log \tilde{Z}_n(\bar{\bT}, \bar{\bX})]$ is a convex function of $\tau$. Combining this with \eqref{eq:parisi_answer} and \eqref{eq:define_upsilon}, we have the desired conclusion if the function $\tau \mapsto \upsilon(\tau,0)$ is differentiable in $\tau$ at $\tau=\tau_0$. Using~\cite[Lemma 16]{jagannath2016dynamic}, we have that for any $\mu \in \mathcal{P}([0,1])$, the solution to the Parisi PDE $\Phi_{\mu}(t,x)$ is differentiable in $x$ and $\|\partial_x \Phi_{\mu}\|
_{\infty} \leq 1$. Using Dominated Convergence Theorem, we have that for any $\mu \in \mathcal{P}([0,1])$, 
\begin{align}
   \frac{\partial}{\partial \tau}P_{\tau,\btheta_0,0}(\mu) \Big|_{\tau=\tau_0}= \mathbb{E}[T \partial_x\Phi_{\mu}(0,\tau_0 T + H)]. \nonumber 
\end{align}
The desired conclusion now follows by an application of Danskin's envelope theorem~\cite{bernhard1995theorem}. 

For the indirect effect, using Lemma~\ref{lem:de_define} we have 
\begin{align}
    \mathrm{IE} = \frac{1}{n}\frac{\partial}{\partial \gamma} \mathbb{E}_{\bar{\bT}, \bar{\bX}}[\log \tilde{Z}_n(\bar{\bT}, \bar{\bX})]\Big|_{\gamma=0} - \frac{1}{n}  \frac{\partial}{\partial \gamma} \mathbb{E}_{\bar{\bX}}[\log \tilde{Z}_n(-\mathbf{1}, \bar{\bX})]\Big|_{\gamma=0} - \frac{1}{2} \mathrm{DE}. \nonumber 
\end{align}
By direct computation, it follows that $\mathbb{E}_{\bar{\bT},\bar{\bX}}[\log \tilde{Z}_n(\bar{\bT}, \bar{\bX})]$ and $\mathbb{E}_{\bar{\bX}}[\log \tilde{Z}_n(-\mathbf{1}, \bar{\bX})]$
are convex functions in $\gamma$. The desired conclusion thus follows if $\upsilon(\tau_0, \gamma)$ and $\widehat{\upsilon}(\tau_0, \gamma)$ are differentiable in $\gamma$ at $\gamma=0$. The rest of the argument is the same as that outlined earlier for the differentiability of $\upsilon(\tau,0)$ in $\tau$---we use the $\ell^\infty$-boundedness of $\partial_x\Phi_\mu$, Dominated Convergence and Danskin's Envelope Theorem to conclude the proof.    
\end{proof}

\subsection{Proof of Theorem~\ref{thm:amp}}
\label{proof:amp}
We begin by noting the state evolution equations for the AMP iterate~\eqref{eq:amp_iterate}. Recall the definition of the function $g$ from~\eqref{eq:g_define}. Suppose $G_1,G_2,G_3 \sim N(0,1)$ independent and $\bar T \sim \mathrm{Unif}(\pm 1)$, $H = \btheta^\top_0 \bX_1$, $\bX_1 \sim \mathbb{P}_X$. Define 
\begin{equation}
    \phi(t)= \beta^2 \bE \left[g\Big(\tau_0 \bar T+ H+ G_1 \sqrt{t}+ G_2 \sqrt{\beta^2 q -t}\Big) g\Big(\tau_0 \bar T+ H+ G_1 \sqrt{t}+ G_3 \sqrt{\beta^2 q -t}\Big)\right]
\end{equation}
for $t \leq \beta^2 q$. Define a sequence of real numbers $(a_k)^{\infty}_{k=0}$ as $a_0=0$, $a_{k+1}= \phi(a_k)$. Using~\cite[Lemma 3.4]{sellke2024optimizing}, we obtain that the sequence $a_k$ is increasing with 
\[\lim\limits_{k\rightarrow \infty} a_k= \beta^2 q.\] The sequence $a_k$ identifies the state evolution limits of our Algorithm~\ref{alg:amp}. More precisely, given any $k\in \mathbb{N}$, suppose the limits of $(\bh_1,\bh_2, \bw^{-1}, \bw^0, \bx^0, \bm^0, \ldots, \bw^k, \bx^k, \bm^k)$ are given by $(H_1,H_2, W^{-1}, W^0, X^0, M^0, \ldots, W^k, X^k, M^k)$ as $n \rightarrow \infty$. Using~\cite[Lemma 3.3]{sellke2024optimizing}, we obtain tha each $W^j$ is a Gaussian random variable with $\bE[W^j]=0$, $X^j = W^j+ H_1+H_2$, $M^{j+1}= g(X^j)$. Further, the following holds for $j < k$,
\begin{align}\label{eq:state_evol_var}
    &\mathrm{Var}[W^j]= \beta^2 q, \quad \bE[W^jW^k]= a_j, \nonumber\\
    &\bE[(M^j)^2]= q, \quad \bE[M^j M^k]= \frac{1}{\beta^2} \phi(a_j).
\end{align}
Note that $(W^k)^{\infty}_{k=0}$ is independent of $(H_1,H_2)$, where $H_1 \sim  \tau_0 \bar T$, $H_2 \sim \bar \bX^\top_1 \btheta_0$. It immediately follows that
\begin{align*}
   \lim\limits_{M \rightarrow \infty} \lim\limits_{n \rightarrow \infty} \frac{1}{n} \bar \bT^\top \bm^M  = \lim\limits_{M \rightarrow \infty} \lim\limits_{n \rightarrow \infty}\frac{1}{n \tau_0} \bh^\top_1 \bm^M = \bE [\bar T \partial_x \Phi_{\mu^\star}(q, H_1+H_2+ Z \beta \sqrt{q})],
\end{align*}
in probability. Let $\{X_t: t \in [0,1]\}$ solve the SDE 
\begin{align}
    dX_t = \beta^2 \mu^{\star}_{\tau_0,0}(t) dt + \beta dW_t , \nonumber 
\end{align}
where $\{W_t : t \in [0,1]\}$ represents Brownian motion. Recalling that $q = \inf(\mathrm{supp}(\mu^{\star}_{\tau_0,0}))$, we have
\begin{align}
    X_q = X_0 + \beta W_q. \nonumber 
\end{align}
In turn, this implies 
\begin{align}
    \mathbb{E}[\bar T \partial_x \Phi_{\mu^\star}(q, H_1+H_2 + Z \beta \sqrt{q})] = \mathbb{E}[\mathbf{T}_1 \partial_x \Phi_{\mu^\star}(q, X_q)], \nonumber 
\end{align}
where $X_0 = \mathbf{h}_1 + \mathbf{h}_2$. 
Finally, we recall that the process $ \partial_x \Phi_{\mu^\star}(t, X_t)$ is a martingale~\cite{auffinger2015parisi},\cite{jagannath2016dynamic}, implying that
\begin{align*}
    \bE [\bar T \partial_x \Phi_{\mu^\star}(q, H_1+H_2+ Z \beta \sqrt{q})]= \bE [\bar T \partial_x \Phi_{\mu^\star}(0, H_1+H_2)]
\end{align*}
Therefore using Theorem~\ref{cor:de_answer}, we obtain that
\begin{equation*}
    \Big| \ee_{\bar{\bT},\bar{\bX}} [\widehat{\mathrm{DE}}_M] - \mathrm{DE} \Big| \stackrel{\mathscr{P}_{n,M}}{\longrightarrow} 0
\end{equation*}
To show the consistency of $\widehat{\mathrm{IE}}_M$, note that by the argument above, we also have
\begin{align*}
   \lim\limits_{M \rightarrow \infty} \lim\limits_{n \rightarrow \infty} \frac{1}{n} \mathbf{1}^\top \bm^M = \bE [ \partial_x \Phi_{\mu^\star_{\tau_0,0}}(q, H_1+H_2+ Z \beta \sqrt{q})].
\end{align*}
Similarly, by replacing $\mu^\star_{\tau_0,0}$ with $\widehat{\mu}^\star_{\tau_0,0}$, we obtain that 
\begin{align*}
   \lim\limits_{M \rightarrow \infty} \lim\limits_{n \rightarrow \infty} \frac{1}{n} \mathbf{1}^\top \overline{\bm}^M = \bE [ \partial_x \Phi_{\widehat{\mu}^\star_{\tau_0,0}}(\widehat q, H_1+H_2+ Z \beta \sqrt{q})].
\end{align*}
Using Theorem~\ref{cor:de_answer} and the martingale property, we have the desired conclusion.

\subsection{Proof of Theorem~\ref{lem:mpl_plug_in}}
\label{sec:estimation_proof}
We prove Theorem~\ref{lem:mpl_plug_in} in this section.  
\begin{proof}[Proof of Theorem~\ref{lem:mpl_plug_in}]
    Using~\cite[Theorem 2.3]{bhattacharya2024causal} we have
\begin{equation}\label{eq:root_n}
    \|(\tau_0,\btheta_0) -(\hat \tau_{\text{MPL}}, \hat {\boldsymbol{\theta}}_{\text{MPL}})\| = O_{\mathbb{P}}(n^{-1/2}).
\end{equation}

Note that we use the same block-approximation for both the estimators $\widehat{\mathrm{DE}}_{(r,\delta)}(\tau_0, \btheta_0)$ and $\widehat{\mathrm{DE}}_{(r,\delta)}(\hat \tau_{\mathrm{MPL}}, \hat{\btheta}_{\mathrm{MPL}})$ in Algorithm~\ref{algo:de_method}. Further, $(\tau,\btheta)\mapsto f_{(r,\varepsilon)}$ defined by~\eqref{eq:simplified_gibbs} is continuous. Therefore, the conditional expectation $(\langle V_{a,k,+} \rangle_{(r,\varepsilon)}, \langle V_{a,k,+} \rangle_{(r,\varepsilon)})$, $a \in [m], k \in 0 \cup [2^r]$ is a continuous function of $(\tau,\btheta)$. Hence, the conclusion for $\widehat{\mathrm{DE}}_{(r,\delta)}$ follows from~\ref{eq:root_n}. A similar argument yields the conclusion for $\widehat{\mathrm{IE}}_{(r,\delta)}$.
 
Lemma~\ref{lem:amp_stable} in the Appendix shows that our AMP method (Algorithm~\ref{alg:amp}) is stable w.r.t. magnetization. Hence the conclusion for~ $\widehat{\mathrm{DE}}_M$ and $\widehat{\mathrm{IE}}_M$ given by~\eqref{eq:de_amp} and~\eqref{eq:ie_amp} follows immediately by~\eqref{eq:root_n}.

\end{proof}

\bibliography{ref}

@article{choi2024new,
  title={New estimands for experiments with strong interference},
  author={Choi, David},
  journal={Journal of the American Statistical Association},
  volume={119},
  number={548},
  pages={2670--2679},
  year={2024},
  publisher={Taylor \& Francis}
}

@article{viviano2025policy,
  title={Policy targeting under network interference},
  author={Viviano, Davide},
  journal={Review of Economic Studies},
  volume={92},
  number={2},
  pages={1257--1292},
  year={2025},
  publisher={Oxford University Press UK}
}

@article{carmona2006universality,
  title={Universality in Sherrington--Kirkpatrick's spin glass model},
  author={Carmona, Philippe and Hu, Yueyun},
  journal={Annales de l'Institut Henri Poincare (B) Probability and Statistics},
  volume={42},
  number={2},
  pages={215--222},
  year={2006},
  organization={Elsevier}
}

@article{chatterjee2005simple,
  title={A simple invariance theorem},
  author={Chatterjee, Sourav},
  journal={arXiv preprint math/0508213},
  year={2005}
}

@article{bernhard1995theorem,
  title={On a theorem of Danskin with an application to a theorem of Von Neumann-Sion},
  author={Bernhard, Pierre and Rapaport, Alain},
  journal={Nonlinear Analysis: Theory, Methods \& Applications},
  volume={24},
  number={8},
  pages={1163--1181},
  year={1995},
  publisher={Pergamon}
}

@article{mukherjee2024logistic,
  title={Logistic Regression Under Network Dependence},
  author={Mukherjee, Somabha and Niu, Ziang and Halder, Sagnik and Bhattacharya, Bhaswar B and Michailidis, George},
  journal={Journal of Machine Learning Research},
  volume={25},
  number={411},
  pages={1--62},
  year={2024}
}

@article{fox2019fast,
  title={A fast new algorithm for weak graph regularity},
  author={Fox, Jacob and Lov{\'a}sz, L{\'a}szl{\'o} Mikl{\'o}s and Zhao, Yufei},
  journal={Combinatorics, Probability and Computing},
  volume={28},
  number={5},
  pages={777--790},
  year={2019},
  publisher={Cambridge University Press}
}

@article{berthet2019exact,
  title={Exact recovery in the Ising blockmodel},
  author={Berthet, Quentin and Rigollet, Philippe and Srivastava, Piyush},
  journal={The Annals of Statistics},
  volume={47},
  number={4},
  pages={1805--1834},
  year={2019},
  publisher={JSTOR}
}

@book{boucheronbook,
    author = {Boucheron, Stéphane and Lugosi, Gábor and Massart, Pascal},
    title = "{Concentration Inequalities: A Nonasymptotic Theory of Independence}",
    publisher = {Oxford University Press},
    year = {2013},
    month = {02},
    isbn = {9780199535255}
}

@article{auffinger2015parisi,
  title={The Parisi formula has a unique minimizer},
  author={Auffinger, Antonio and Chen, Wei-Kuo},
  journal={Communications in Mathematical Physics},
  volume={335},
  pages={1429--1444},
  year={2015},
  publisher={Springer}
}

@article{sellke2024optimizing,
  title={Optimizing mean field spin glasses with external field},
  author={Sellke, Mark},
  journal={Electronic Journal of Probability},
  volume={29},
  pages={1--47},
  year={2024}
}

@article{el2021optimization,
  title={Optimization of mean-field spin glasses},
  author={El Alaoui, Ahmed and Montanari, Andrea and Sellke, Mark},
  journal={The Annals of Probability},
  volume={49},
  number={6},
  pages={2922--2960},
  year={2021}
}

@article{parisi1979infinite,
  title={Infinite number of order parameters for spin-glasses},
  author={Parisi, Giorgio},
  journal={Physical Review Letters},
  volume={43},
  number={23},
  pages={1754},
  year={1979},
  publisher={APS}
}

@article{panchenko2013parisi,
  title={The Parisi ultrametricity conjecture},
  author={Panchenko, Dmitry},
  journal={Annals of Mathematics},
  volume={177},
  pages={383--393},
  year={2013}
}

@article{jagannath2016dynamic,
  title={A dynamic programming approach to the Parisi functional},
  author={Jagannath, Aukosh and Tobasco, Ian},
  journal={Proceedings of the American Mathematical Society},
  volume={144},
  number={7},
  pages={3135--3150},
  year={2016}
}

@article{talagrand2006parisi,
  title={The parisi formula},
  author={Talagrand, Michel},
  journal={Annals of mathematics},
  pages={221--263},
  year={2006},
  publisher={JSTOR}
}

@inproceedings{bandeira2020computational,
  title={Computational Hardness of Certifying Bounds on Constrained PCA Problems},
  author={Bandeira, Afonso S and Kunisky, Dmitriy and Wein, Alexander S},
  booktitle={11th Innovations in Theoretical Computer Science Conference (ITCS 2020)},
  volume={151},
  year={2020}
}

@article{dudeja2023universality,
  title={Universality of approximate message passing with semirandom matrices},
  author={Dudeja, Rishabh and M. Lu, Yue and Sen, Subhabrata},
  journal={The Annals of Probability},
  volume={51},
  number={5},
  pages={1616--1683},
  year={2023},
  publisher={Institute of Mathematical Statistics}
}

@article{chen2021universality,
  title={Universality of approximate message passing algorithms},
  author={Chen, Wei Kuo and Lam, Wai-Kit},
  journal={Electronic Journal of Probability},
  volume={26},
  pages={36},
  year={2021},
  publisher={Institute of Mathematical Statistics}
}

@article{bayati2015universality,
  title={Universality in polytope phase transitions and message passing algorithms.},
  author={Bayati, Mohsen and Lelarge, Marc and Montanari, Andrea},
  journal={Annals of applied probability},
  volume={25},
  number={2},
  pages={753--822},
  year={2015},
  publisher={Institute of Mathematical Statistics}
}

@inproceedings{kunisky2024optimality,
  title={Optimality of Glauber dynamics for general-purpose Ising model sampling and free energy approximation},
  author={Kunisky, Dmitriy},
  booktitle={Proceedings of the 2024 Annual ACM-SIAM Symposium on Discrete Algorithms (SODA)},
  pages={5013--5028},
  year={2024},
  organization={SIAM}
}

@inproceedings{jain2018mean,
  title={The mean-field approximation: Information inequalities, algorithms, and complexity},
  author={Jain, Vishesh and Koehler, Frederic and Mossel, Elchanan},
  booktitle={Conference On Learning Theory},
  pages={1326--1347},
  year={2018},
  organization={PMLR}
}

@article{cortez2023exploiting,
  title={Exploiting neighborhood interference with low-order interactions under unit randomized design},
  author={Cortez-Rodriguez, Mayleen and Eichhorn, Matthew and Yu, Christina Lee},
  journal={Journal of Causal Inference},
  volume={11},
  number={1},
  pages={20220051},
  year={2023},
  publisher={De Gruyter}
}

@article{shpitser2017modeling,
  title={Modeling interference via symmetric treatment decomposition},
  author={Shpitser, Ilya and Tchetgen, Eric Tchetgen and Andrews, Ryan},
  journal={arXiv preprint arXiv:1709.01050},
  year={2017}
}

@article{vanderweele2010direct,
  title={Direct and indirect effects for neighborhood-based clustered and longitudinal data},
  author={VanderWeele, Tyler J},
  journal={Sociological methods \& research},
  volume={38},
  number={4},
  pages={515--544},
  year={2010},
  publisher={Sage Publications Sage CA: Los Angeles, CA}
}

@article{yu2022estimating,
  title={Estimating the total treatment effect in randomized experiments with unknown network structure},
  author={Yu, Christina Lee and Airoldi, Edoardo M and Borgs, Christian and Chayes, Jennifer T},
  journal={Proceedings of the National Academy of Sciences},
  volume={119},
  number={44},
  pages={e2208975119},
  year={2022},
  publisher={National Acad Sciences}
}

@article{li2022random,
  title={Random graph asymptotics for treatment effect estimation under network interference},
  author={Li, Shuangning and Wager, Stefan},
  journal={The Annals of Statistics},
  volume={50},
  number={4},
  pages={2334--2358},
  year={2022},
  publisher={Institute of Mathematical Statistics}
}

@article{sherman2018identification,
  title={Identification and estimation of causal effects from dependent data},
  author={Sherman, Eli and Shpitser, Ilya},
  journal={Advances in neural information processing systems},
  volume={31},
  year={2018}
}

@article{lauritzen2002chain,
  title={Chain graph models and their causal interpretations},
  author={Lauritzen, Steffen L and Richardson, Thomas S},
  journal={Journal of the Royal Statistical Society Series B: Statistical Methodology},
  volume={64},
  number={3},
  pages={321--348},
  year={2002},
  publisher={Oxford University Press}
}

@article{jiang2022new,
  title={A new central limit theorem for the augmented IPW estimator: Variance inflation, cross-fit covariance and beyond},
  author={Jiang, Kuanhao and Mukherjee, Rajarshi and Sen, Subhabrata and Sur, Pragya},
  journal={The Annals of Statistics},
  volume={53},
  number={2},
  pages={647--675},
  year={2025},
  publisher={Institute of Mathematical Statistics}
}

@article{shirani2023causal,
  title={Causal message-passing for experiments with unknown and general network interference},
  author={Shirani, Sadegh and Bayati, Mohsen},
  journal={Proceedings of the National Academy of Sciences},
  volume={121},
  number={40},
  pages={e2322232121},
  year={2024},
  publisher={National Academy of Sciences}
}

@article{montanari2025optimization,
  title={Optimization of the Sherrington--Kirkpatrick hamiltonian},
  author={Montanari, Andrea},
  journal={SIAM Journal on Computing},
  volume={54},
  number={4},
  pages={FOCS19--1},
  year={2025},
  publisher={SIAM}
}

@article{bayati2024higher,
  title={Higher-order causal message passing for experimentation with complex interference},
  author={Bayati, Mohsen and Luo, Yuwei and Overman, William and Shirani Faradonbeh, Mohamad Sadegh and Xiong, Ruoxuan},
  journal={Advances in Neural Information Processing Systems},
  volume={37},
  pages={81836--81856},
  year={2024}
}

@article{park2023assumption,
  title={Assumption-lean analysis of cluster randomized trials in infectious diseases for intent-to-treat effects and network effects},
  author={Park, Chan and Kang, Hyunseung},
  journal={Journal of the American Statistical Association},
  volume={118},
  number={542},
  pages={1195--1206},
  year={2023},
  publisher={Taylor \& Francis}
}

@book{hayes2017cluster,
  title={Cluster randomised trials},
  author={Hayes, Richard J and Moulton, Lawrence H},
  year={2017},
  publisher={Chapman and Hall/CRC}
}

@inproceedings{ugander2013graph,
  title={Graph cluster randomization: Network exposure to multiple universes},
  author={Ugander, Johan and Karrer, Brian and Backstrom, Lars and Kleinberg, Jon},
  booktitle={Proceedings of the 19th ACM SIGKDD international conference on Knowledge discovery and data mining},
  pages={329--337},
  year={2013}
}

@article{jagadeesan2020designs,
  title={Designs for estimating the treatment effect in networks with interference},
  author={Jagadeesan, Ravi and Pillai, Natesh S and Volfovsky, Alexander},
    journal = {The Annals of Statistics},
  year={2020}
}

@inproceedings{toulis2013estimation,
  title={Estimation of causal peer influence effects},
  author={Toulis, Panos and Kao, Edward},
  booktitle={International conference on machine learning},
  pages={1489--1497},
  year={2013},
  organization={PMLR}
}

@article{bhattacharya2024causal,
  title={Causal effect estimation under network interference with mean-field methods},
  author={Bhattacharya, Sohom and Sen, Subhabrata},
  journal={arXiv preprint arXiv:2407.19613},
  year={2024}
}

@article{forastiere2021identification,
  title={Identification and estimation of treatment and interference effects in observational studies on networks},
  author={Forastiere, Laura and Airoldi, Edoardo M and Mealli, Fabrizia},
  journal={Journal of the American Statistical Association},
  volume={116},
  number={534},
  pages={901--918},
  year={2021},
  publisher={Taylor \& Francis}
}

@article{basse2019randomization,
  title={Randomization tests of causal effects under interference},
  author={Basse, Guillaume W and Feller, Avi and Toulis, Panos},
  journal={Biometrika},
  volume={106},
  number={2},
  pages={487--494},
  year={2019},
  publisher={Oxford University Press}
}

@article{aronow2017estimating,
title={Estimating average causal effects under general interference, with application to a social network experiment},
 author = {Peter M. Aronow and Cyrus Samii},
 journal = {The Annals of Applied Statistics},
 number = {4},
 publisher = {Institute of Mathematical Statistics},
 volume = {11},
 year = {2017}
}

@article{ferracci2014evidence,
  title={Evidence of treatment spillovers within markets},
  author={Ferracci, Marc and Jolivet, Gr{\'e}gory and van den Berg, Gerard J},
  journal={Review of Economics and Statistics},
  volume={96},
  number={5},
  pages={812--823},
  year={2014},
  publisher={The MIT Press}
}

@article{lundin2014estimation,
  title={Estimation of causal effects in observational studies with interference between units},
  author={Lundin, Mathias and Karlsson, Maria},
  journal={Statistical Methods \& Applications},
  volume={23},
  pages={417--433},
  year={2014},
  publisher={Springer}
}

@article{liu2014large,
  title={Large sample randomization inference of causal effects in the presence of interference},
  author={Liu, Lan and Hudgens, Michael G},
  journal={Journal of the american statistical association},
  volume={109},
  number={505},
  pages={288--301},
  year={2014},
  publisher={Taylor \& Francis}
}

@article{tchetgen2012causal,
  title={On causal inference in the presence of interference},
  author={Tchetgen, Eric J Tchetgen and VanderWeele, Tyler J},
  journal={Statistical methods in medical research},
  volume={21},
  number={1},
  pages={55--75},
  year={2012},
  publisher={SAGE Publications Sage UK: London, England}
}

@article{manski2013identification,
  title={Identification of treatment response with social interactions},
  author={Manski, Charles F},
  journal={The Econometrics Journal},
  volume={16},
  number={1},
  pages={S1--S23},
  year={2013},
  publisher={Oxford University Press Oxford, UK}
}

@article{graham2008identifying,
  title={Identifying social interactions through conditional variance restrictions},
  author={Graham, Bryan S},
  journal={Econometrica},
  volume={76},
  number={3},
  pages={643--660},
  year={2008},
  publisher={Wiley Online Library}
}

@article{hong2006evaluating,
  title={Evaluating kindergarten retention policy: A case study of causal inference for multilevel observational data},
  author={Hong, Guanglei and Raudenbush, Stephen W},
  journal={Journal of the American Statistical Association},
  volume={101},
  number={475},
  pages={901--910},
  year={2006},
  publisher={Taylor \& Francis}
}

@article{hudgens2008toward,
  title={Toward causal inference with interference},
  author={Hudgens, Michael G and Halloran, M Elizabeth},
  journal={Journal of the American Statistical Association},
  volume={103},
  number={482},
  pages={832--842},
  year={2008},
  publisher={Taylor \& Francis}
}

@article{rosenbaum2007interference,
  title={Interference between units in randomized experiments},
  author={Rosenbaum, Paul R},
  journal={Journal of the american statistical association},
  volume={102},
  number={477},
  pages={191--200},
  year={2007},
  publisher={Taylor \& Francis}
}

@article{manski1993identification,
  title={Identification of endogenous social effects: The reflection problem},
  author={Manski, Charles F},
  journal={The review of economic studies},
  volume={60},
  number={3},
  pages={531--542},
  year={1993},
  publisher={Wiley-Blackwell}
}

@article{angrist2014perils,
  title={The perils of peer effects},
  author={Angrist, Joshua D},
  journal={Labour Economics},
  volume={30},
  pages={98--108},
  year={2014},
  publisher={Elsevier}
}

@article{kang2016peer,
  title={Peer encouragement designs in causal inference with partial interference and identification of local average network effects},
  author={Kang, Hyunseung and Imbens, Guido},
  journal={arXiv preprint arXiv:1609.04464},
  year={2016}
}

@article{eichhorn2024low,
  title={Low-order outcomes and clustered designs: combining design and analysis for causal inference under network interference},
  author={Eichhorn, Matthew and Khan, Samir and Ugander, Johan and Yu, Christina Lee},
  journal={arXiv preprint arXiv:2405.07979},
  year={2024}
}

@article{goldsmith2013social,
  title={Social networks and the identification of peer effects},
  author={Goldsmith-Pinkham, Paul and Imbens, Guido W},
  journal={Journal of Business \& Economic Statistics},
  volume={31},
  number={3},
  pages={253--264},
  year={2013},
  publisher={Taylor \& Francis}
}

@article{bramoulle2009identification,
  title={Identification of peer effects through social networks},
  author={Bramoull{\'e}, Yann and Djebbari, Habiba and Fortin, Bernard},
  journal={Journal of econometrics},
  volume={150},
  number={1},
  pages={41--55},
  year={2009},
  publisher={Elsevier}
}

@article{lee2007identification,
  title={Identification and estimation of econometric models with group interactions, contextual factors and fixed effects},
  author={Lee, Lung-Fei},
  journal={Journal of econometrics},
  volume={140},
  number={2},
  pages={333--374},
  year={2007},
  publisher={Elsevier}
}

@article{sobel2006randomized,
  title={What do randomized studies of housing mobility demonstrate? Causal inference in the face of interference},
  author={Sobel, Michael E},
  journal={Journal of the American Statistical Association},
  volume={101},
  number={476},
  pages={1398--1407},
  year={2006},
  publisher={Taylor \& Francis}
}

@article{sasakura2014ising,
  title={Ising model on random networks and the canonical tensor model},
  author={Sasakura, Naoki and Sato, Yuki},
  journal={Progress of Theoretical and Experimental Physics},
  volume={2014},
  number={5},
  pages={053B03},
  year={2014},
  publisher={Oxford University Press}
}

@article{akiyama2019phase,
  title={Phase transition of four-dimensional Ising model with higher-order tensor renormalization group},
  author={Akiyama, Shinichiro and Kuramashi, Yoshinobu and Yamashita, Takumi and Yoshimura, Yusuke},
  journal={Physical review D},
  volume={100},
  number={5},
  pages={054510},
  year={2019},
  publisher={APS}
}

@article{mukherjee2022estimation,
  title={Estimation in tensor Ising models},
  author={Mukherjee, Somabha and Son, Jaesung and Bhattacharya, Bhaswar B},
  journal={Information and Inference: A Journal of the IMA},
  volume={11},
  number={4},
  pages={1457--1500},
  year={2022},
  publisher={Oxford University Press}
}

@article{liu2024tensor,
  title={Tensor recovery in high-dimensional Ising models},
  author={Liu, Tianyu and Mukherjee, Somabha and Biswas, Rahul},
  journal={Journal of Multivariate Analysis},
  pages={105335},
  year={2024},
  publisher={Elsevier}
}

@article{basak2017universality,
  title={Universality of the mean-field for the Potts model},
  author={Basak, Anirban and Mukherjee, Sumit},
  journal={Probability Theory and Related Fields},
  volume={168},
  pages={557--600},
  year={2017},
  publisher={Springer}
}

@article{chatterjee2016nonlinear,
  title={Nonlinear large deviations},
  author={Chatterjee, Sourav and Dembo, Amir},
  journal={Advances in Mathematics},
  volume={299},
  pages={396--450},
  year={2016},
  publisher={Elsevier}
}

@article{eldan2022spectral,
  title={A spectral condition for spectral gap: fast mixing in high-temperature Ising models},
  author={Eldan, Ronen and Koehler, Frederic and Zeitouni, Ofer},
  journal={Probability theory and related fields},
  volume={182},
  number={3-4},
  pages={1035--1051},
  year={2022},
  publisher={Springer}
}

@article{ogburn2014causal,
  title={Causal diagrams for interference},
  author={Ogburn, Elizabeth L and VanderWeele, Tyler J},
  journal={Statistical Science},
  volume={29},
  number={4},
  pages={559--578},
  year={2014},
  publisher={Institute of Mathematical Statistics}
}

@inproceedings{bhattacharya2020causal,
  title={Causal inference under interference and network uncertainty},
  author={Bhattacharya, Rohit and Malinsky, Daniel and Shpitser, Ilya},
  booktitle={Uncertainty in Artificial Intelligence},
  pages={1028--1038},
  year={2020},
  organization={PMLR}
}

@article{athey2018exact,
  title={Exact p-values for network interference},
  author={Athey, Susan and Eckles, Dean and Imbens, Guido W},
  journal={Journal of the American Statistical Association},
  volume={113},
  number={521},
  pages={230--240},
  year={2018},
  publisher={Taylor \& Francis}
}

@inproceedings{koehler2022sampling,
  title={Sampling approximately low-rank Ising models: MCMC meets variational methods},
  author={Koehler, Frederic and Lee, Holden and Risteski, Andrej},
  booktitle={Conference on Learning Theory},
  pages={4945--4988},
  year={2022},
  organization={PMLR}
}

@article{sellke2025exponentially,
  title={Exponentially Slow Mixing of the Low Temperature SK Model},
  author={Sellke, Mark},
  journal={arXiv preprint arXiv:2511.22621},
  year={2025}
}

@article{sompolinsky1981dynamic,
  title={Dynamic theory of the spin-glass phase},
  author={Sompolinsky, Haim and Zippelius, Annette},
  journal={Physical Review Letters},
  volume={47},
  number={5},
  pages={359},
  year={1981},
  publisher={APS}
}

@inproceedings{anari2022entropic,
  title={Entropic independence: optimal mixing of down-up random walks},
  author={Anari, Nima and Jain, Vishesh and Koehler, Frederic and Pham, Huy Tuan and Vuong, Thuy-Duong},
  booktitle={Proceedings of the 54th Annual ACM SIGACT Symposium on Theory of Computing},
  pages={1418--1430},
  year={2022}
}

@article{el2025sampling,
  title={Sampling from mean-field gibbs measures via diffusion processes},
  author={El Alaoui, Ahmed and Montanari, Andrea and Sellke, Mark},
  journal={Probability and Mathematical Physics},
  volume={6},
  number={3},
  pages={961--1022},
  year={2025},
  publisher={Mathematical Sciences Publishers}
}

@article{huang2024sampling,
  title={Sampling from spherical spin glasses in total variation via algorithmic stochastic localization},
  author={Huang, Brice and Montanari, Andrea and Pham, Huy Tuan},
  journal={arXiv preprint arXiv:2404.15651},
  year={2024}
}

@inproceedings{el2022sampling,
  title={Sampling from the Sherrington-Kirkpatrick Gibbs measure via algorithmic stochastic localization},
  author={El Alaoui, Ahmed and Montanari, Andrea and Sellke, Mark},
  booktitle={2022 IEEE 63rd Annual Symposium on Foundations of Computer Science (FOCS)},
  pages={323--334},
  year={2022},
  organization={IEEE}
}

@article{adhikari2024spectral,
  title={Spectral gap estimates for mixed p-spin models at high temperature},
  author={Adhikari, Arka and Brennecke, Christian and Xu, Changji and Yau, Horng-Tzer},
  journal={Probability Theory and Related Fields},
  volume={189},
  number={3},
  pages={879--907},
  year={2024},
  publisher={Springer}
}

@inproceedings{anari2024trickle,
  title={Trickle-down in localization schemes and applications},
  author={Anari, Nima and Koehler, Frederic and Vuong, Thuy-Duong},
  booktitle={Proceedings of the 56th Annual ACM Symposium on Theory of Computing},
  pages={1094--1105},
  year={2024}
}

@inproceedings{galanis2024sampling,
  title={On Sampling from Ising Models with Spectral Constraints},
  author={Galanis, Andreas and Kalavasis, Alkis and Kandiros, Anthimos Vardis},
  booktitle={Approximation, Randomization, and Combinatorial Optimization. Algorithms and Techniques (APPROX/RANDOM 2024)},
  pages={70--1},
  year={2024},
  organization={Schloss Dagstuhl--Leibniz-Zentrum f{\"u}r Informatik}
}

@article{alon1994algorithmic,
  title={The algorithmic aspects of the regularity lemma},
  author={Alon, Noga and Duke, Richard A and Lefmann, Hanno and Rodl, Vojtech and Yuster, Raphael},
  journal={Journal of Algorithms},
  volume={16},
  number={1},
  pages={80--109},
  year={1994},
  publisher={Elsevier}
}

@article{frieze1999quick,
  title={Quick approximation to matrices and applications},
  author={Frieze, Alan and Kannan, Ravi},
  journal={Combinatorica},
  volume={19},
  number={2},
  pages={175--220},
  year={1999},
  publisher={Springer}
}

@book{szemeredi1975regular,
  title={Regular partitions of graphs.},
  author={Szemer{\'e}di, Endre},
  year={1975},
  publisher={Stanford University}
}

@misc{komlos1995szemeredi,
  title={Szemeredi''s Regularity Lemma and its applications in graph theory},
  author={Koml{\'o}s, J{\'a}nos and Simonovits, Mikl{\'o}s},
  year={1995},
  publisher={Center for Discrete Mathematics \& Theoretical Computer Science}
}

@inproceedings{frieze1996regularity,
  title={The regularity lemma and approximation schemes for dense problems},
  author={Frieze, Alan and Kannan, Ravi},
  booktitle={Proceedings of 37th conference on foundations of computer science},
  pages={12--20},
  year={1996},
  organization={IEEE}
}

@article{blanca2025tractability,
  title={On the tractability of sampling from the Potts model at low temperatures via random-cluster dynamics},
  author={Blanca, Antonio and Gheissari, Reza},
  journal={Probability Theory and Related Fields},
  volume={191},
  number={3},
  pages={1121--1168},
  year={2025},
  publisher={Springer}
}

@inproceedings{gheissari2022low,
  title={Low-temperature Ising dynamics with random initializations},
  author={Gheissari, Reza and Sinclair, Alistair},
  booktitle={Proceedings of the 54th Annual ACM SIGACT Symposium on Theory of Computing},
  pages={1445--1458},
  year={2022}
}

@article{eckles2016design,
  title={Design and analysis of experiments in networks: Reducing bias from interference},
  author={Eckles, Dean and Karrer, Brian and Ugander, Johan},
  journal={Journal of Causal Inference},
  volume={5},
  number={1},
  pages={20150021},
  year={2016},
  publisher={De Gruyter}
}

@article{tchetgen2021auto,
  title={Auto-g-computation of causal effects on a network},
  author={Tchetgen Tchetgen, Eric J and Fulcher, Isabel R and Shpitser, Ilya},
  journal={Journal of the American Statistical Association},
  volume={116},
  number={534},
  pages={833--844},
  year={2021},
  publisher={Taylor \& Francis}
}

@article{bc_lpi,
  title={An ${L}^p$ theory of sparse graph convergence {I}: Limits, sparse random graph models, and power law distributions},
  author={Borgs, Christian and Chayes, Jennifer and Cohn, Henry and Zhao, Yufei},
  journal={Transactions of the American Mathematical Society},
  volume={372},
  number={5},
  pages={3019--3062},
  year={2019}
}

@book{hopkins2018statistical,
  title={Statistical inference and the sum of squares method},
  author={Hopkins, Samuel},
  year={2018},
  publisher={Cornell University}
}

@book{levin2017markov,
  title={Markov chains and mixing times},
  author={Levin, David A and Peres, Yuval},
  volume={107},
  year={2017},
  publisher={American Mathematical Soc.}
}

@article{celentano2022fundamental,
  title={Fundamental barriers to high-dimensional regression with convex penalties},
  author={Celentano, Michael and Montanari, Andrea},
  journal={The Annals of Statistics},
  volume={50},
  number={1},
  pages={170--196},
  year={2022},
  publisher={Institute of Mathematical Statistics}
}

@inproceedings{berthet2013complexity,
  title={Complexity theoretic lower bounds for sparse principal component detection},
  author={Berthet, Quentin and Rigollet, Philippe},
  booktitle={Conference on learning theory},
  pages={1046--1066},
  year={2013},
  organization={PMLR}
}

@article{reich2021review,
  title={A review of spatial causal inference methods for environmental and epidemiological applications},
  author={Reich, Brian J and Yang, Shu and Guan, Yawen and Giffin, Andrew B and Miller, Matthew J and Rappold, Ana},
  journal={International Statistical Review},
  volume={89},
  number={3},
  pages={605--634},
  year={2021},
  publisher={Wiley Online Library}
}

@article{matthay2022causal,
  title={Causal inference challenges and new directions for epidemiologic research on the health effects of social policies},
  author={Matthay, Ellicott C and Glymour, M Maria},
  journal={Current Epidemiology Reports},
  volume={9},
  number={1},
  pages={22--37},
  year={2022},
  publisher={Springer}
}

@inproceedings{brennan2018reducibility,
  title={Reducibility and computational lower bounds for problems with planted sparse structure},
  author={Brennan, Matthew and Bresler, Guy and Huleihel, Wasim},
  booktitle={Conference On Learning Theory},
  pages={48--166},
  year={2018},
  organization={PMLR}
}

@article{boix2025average,
  title={The Average-Case Complexity of Counting Cliques in Erdo{\H{o}}s--Re{\'e}nyi Hypergraphs},
  author={Boix-Adseraa, Enric and Brennan, Matthew and Bresler, Guy},
  journal={SIAM Journal on Computing},
  volume={54},
  number={4},
  pages={FOCS19--39},
  year={2025},
  publisher={SIAM}
}

@inproceedings{brennan2019optimal,
  title={Optimal average-case reductions to sparse pca: From weak assumptions to strong hardness},
  author={Brennan, Matthew and Bresler, Guy},
  booktitle={Conference on Learning Theory},
  pages={469--470},
  year={2019},
  organization={PMLR}
}

@book{vazirani2001approximation,
  title={Approximation algorithms},
  author={Vazirani, Vijay V},
  volume={1},
  year={2001},
  publisher={Springer}
}

@article{liu2025auto,
  title={Auto-Doubly Robust Estimation of Causal Effects on a Network},
  author={Liu, Jizhou and Zhang, Dake and Tchetgen, Eric J Tchetgen},
  journal={arXiv preprint arXiv:2506.23332},
  year={2025}
}

@article{panchenko2005free,
  title={Free energy in the generalized Sherrington--Kirkpatrick mean field model},
  author={Panchenko, Dmitry},
  journal={Reviews in Mathematical Physics},
  volume={17},
  number={07},
  pages={793--857},
  year={2005},
  publisher={World Scientific}
}

@article{murphy2013loopy,
  title={Loopy belief propagation for approximate inference: An empirical study},
  author={Murphy, Kevin and Weiss, Yair and Jordan, Michael I},
  journal={arXiv preprint arXiv:1301.6725},
  year={2013}
}

@article{dembo2013factor,
  title={Factor models on locally tree-like graphs},
  author={Dembo, Amir and Montanari, Andrea and Sun, Nike},
  journal={The Annals of Probability},
  pages={4162--4213},
  year={2013},
  publisher={JSTOR}
}

@article{borgs2018p,
  title={An ${L}^p$ theory of sparse graph convergence {II}: LD convergence, quotients and right convergence},
  author={Borgs, Christian and Chayes, Jennifer T and Cohn, Henry and Zhao, Yufei},
  journal={The Annals of Probability},
  volume={46},
  number={1},
  pages={337--396},
  year={2018},
  publisher={Institute of Mathematical Statistics}
}

@article {borgsdense1,
    AUTHOR = {Borgs, C. and Chayes, J. T. and Lov\'{a}sz, L. and S\'{o}s, V. T. and
              Vesztergombi, K.},
     TITLE = {Convergent sequences of dense graphs. {I}. {S}ubgraph
              frequencies, metric properties and testing},
   JOURNAL = {Adv. Math.},
  FJOURNAL = {Advances in Mathematics},
    VOLUME = {219},
      YEAR = {2008},
    NUMBER = {6},
     PAGES = {1801--1851},
      ISSN = {0001-8708},
   MRCLASS = {05C80 (05C12 05C35 68R10)},
  MRNUMBER = {2455626},
MRREVIEWER = {Michael Krivelevich},
       DOI = {10.1016/j.aim.2008.07.008},
       URL = {https://doi.org/10.1016/j.aim.2008.07.008},
}

@article {borgsdense2,
    AUTHOR = {Borgs, C. and Chayes, J. T. and Lov\'{a}sz, L. and S\'{o}s, V. T. and
              Vesztergombi, K.},
     TITLE = {Convergent sequences of dense graphs {II}. {M}ultiway cuts and
              statistical physics},
   JOURNAL = {Ann. of Math. (2)},
  FJOURNAL = {Annals of Mathematics. Second Series},
    VOLUME = {176},
      YEAR = {2012},
    NUMBER = {1},
     PAGES = {151--219},
      ISSN = {0003-486X},
   MRCLASS = {05C80 (05C35 05C60 82B99)},
  MRNUMBER = {2925382},
MRREVIEWER = {Michael Krivelevich},
       DOI = {10.4007/annals.2012.176.1.2},
       URL = {https://doi.org/10.4007/annals.2012.176.1.2},
}

@book {Lovasz2012,
    AUTHOR = {Lov\'{a}sz, L\'{a}szl\'{o}},
     TITLE = {Large networks and graph limits},
    SERIES = {American Mathematical Society Colloquium Publications},
    VOLUME = {60},
 PUBLISHER = {American Mathematical Society, Providence, RI},
      YEAR = {2012},
     PAGES = {xiv+475},
      ISBN = {978-0-8218-9085-1},
   MRCLASS = {05-02 (05C60 05C80 05C82 05D40)},
  MRNUMBER = {3012035},
MRREVIEWER = {Anant P. Godbole},
       DOI = {10.1090/coll/060},
       URL = {https://doi.org/10.1090/coll/060},
}

\section*{Appendix}
This appendix contains the proofs of some results omitted in the main paper. We begin with some stability lemmas that serve as key tools for our main proofs. Next, we establish the validity of our algorithm for general compactly supported covariates. 

\subsection*{Stability Lemmas} 
To state our result, consider a Markov Random Field with interaction matrix $\bA_n$ and external field $\mathbf{h}: = (h_1, h_2, \ldots, h_n) \in \mathbb{R}^n$. For this model, we define, 
\begin{align}
     &Z^{\bA_n}_n (\mathbf{h})= \sum\limits_{\by \in \{-1,1\}^n}\exp\Big(\frac{1}{2}\,\mathbf{y}^\top \mathbf{A}_n \mathbf{y}+ \sum_{i=1}^{n} h_i y_i \Big) 
     \label{eq:define_zn_an} \\
     & F^{\bA_n}_n(\mathbf{h}) = \frac{1}{n} \log Z^{\bA_n}_n(\mathbf{h}), \label{eq:define_fn}
\end{align}
We require the following stability of the log-normalizing constant $F^{\bA_n}_n(\bh)$ w.r.t. interaction matrices and random field:
\begin{lemma}\label{lem:graph_approx_zn}
    Let $\bA_n$, $\mathbf{B}_n$ be two $n \times n$ matrices, and $\bh ,\tilde \bh \in \mathbb R^n$. Suppose $F^{\bA_n}_n(\bh) $ and $F^{\mathbf{B}_n}_n(\tilde \bh)$ are defined as in \eqref{eq:define_fn}. 
    Then we have
    \[
     \left \vert F^{\mathbf{A}_n}_n(\bh) - F^{\mathbf{B}_n}_n(\tilde \bh)  \right \vert \le \Big(\|\bA_n - \mathbf{B}_n\|+ \|\bh -\tilde \bh\|_{\infty}\Big).
    \]
\end{lemma}
\begin{proof}
By triangle inequality, we have
\begin{align}\label{eq:stable_triangle_ineq}
 \left \vert F^{\mathbf{A}_n}_n(\bh) - F^{\mathbf{B}_n}_n(\tilde \bh)  \right \vert \le  \left \vert F^{\mathbf{A}_n}_n(\bh) - F^{\mathbf{A}_n}_n(\tilde \bh)  \right \vert +\left \vert F^{\mathbf{A}_n}_n(\tilde \bh) - F^{\mathbf{B}_n}_n(\tilde \bh)  \right \vert =: \mathcal{T}_1+ \mathcal{T}_2.
\end{align}
We bound $\mathcal{T}_1, \mathcal{T}_2$ separately below. 

\noindent \textbf{Upper bound of $\mathcal{T}_1$:}
The proof is by interpolation. For $\kappa \in [0,1]$, define
    \begin{equation}
        H_n(\kappa) = \sum\limits_{\by \in \{-1,1\}^n}\exp\Big(\frac{1}{2}\,\mathbf{y}^\top \mathbf{A}_n \mathbf{y}+ \sum_{i=1}^{n} y_i (\kappa h_i+ (1-\kappa)\tilde h_i)\Big),
    \end{equation}
    This implies $\frac{1}{n} \log H_n(0)= F^{\mathbf{A}_n}_n(\mathbf{h})$ and  $\frac{1}{n} \log H_n(1)= F^{\mathbf{A}_n}_n(\mathbf{\tilde h})$. Further, we have
    \begin{align*}
        \Big|\frac{\partial}{\partial \kappa} \log  H_n(\kappa)\Big| \le \sqrt{n}\| \mathbf{h} -\mathbf{\tilde h}\|_2 \le  n \| \mathbf{h} -\mathbf{\tilde h}\|_\infty,
    \end{align*}
    since $|y_i| \le 1$. This implies
    \begin{align*}
        \left \vert F^{\mathbf{A}_n}_n(\mathbf{h}) - F^{\mathbf{A}_n}_n(\mathbf{\tilde h})  \right \vert \le \frac{1}{n} \int_{0}^{1}\Big|\frac{\partial}{\partial \kappa} \log  H_n(\kappa)\Big| d\kappa \le \| \mathbf{h} -\mathbf{\tilde h}\|_\infty.
    \end{align*}
    This provides the necessary upper bound of $\mathcal{T}_1$.

\noindent \textbf{Upper bound of $\mathcal{T}_2$:} By definition of $F^{\bA_n}_n$ in~\eqref{eq:define_fn}, we have
   \begin{align*}
       F^{\bA_n}_n (\tilde \bh)&= \frac{1}{n} \log \sum\limits_{\by \in \{-1,1\}^n}\exp\Big(\frac{1}{2}\,\mathbf{y}^\top \mathbf{A}_n \mathbf{y}+ \sum_{i=1}^{n} \tilde h_i y_i \Big) \\
       &= \frac{1}{n}  \log \sum\limits_{\by \in \{-1,1\}^n}\exp\Big(\frac{1}{2}\,\mathbf{y}^\top \mathbf{B}_n \mathbf{y}+ \sum_{i=1}^{n} \tilde h_i y_i + \by^\top (\bA_n -\mathbf{B}_n) \by \Big) \\ 
& \le \frac{1}{n} \log \left [e^{n \|\bA_n-\mathbf{B}_n\|}\sum\limits_{\by \in \{-1,1\}^n}\exp\Big(\frac{1}{2}\,\mathbf{y}^\top \mathbf{B}_n \mathbf{y}+ \sum_{i=1}^{n} h_i y_i  \Big)\right]\\
 & \le F^{\mathbf{B}_n}_n(\tilde \bh) + \|\bA_n-\mathbf{B}_n\| ,
   \end{align*}
where the first inequality uses $\by^\top (\bA_n- \mathbf{B}_n) \by \le \|\bA_n-\mathbf{B}_n\| \|\by\|^2 \le n \|\bA_n-\mathbf{B}_n\|$. Similarly, one can show a lower bound, i.e, $F^{\bA_n}_n(\tilde \bh) \ge F^{\mathbf{B}_n}_n (\tilde \bh)- \|\bA_n-\mathbf{B}_n\|$. This provides the required upper bound of $\mathcal{T}_2$, completing the proof of the Lemma.  
\end{proof}
We now prove Lemma~\ref{lem:de_ie_close} using Lemma~\ref{lem:graph_approx_zn}.
We use the notation $\text{DE}^{\bA_n}$ and $\text{IE}^{\bA_n}$ when direct and indirect effects are computed w.r.t. interaction matrices $\bA_n$. For the remaining results, we will choose $h_i =\tau T_i + \btheta^\top_0 \bX_i +\gamma$ following~\eqref{eq:define_gibbs}. To highlight the dependence on $(\tau, \btheta_0,\gamma)$, we use the notation $F^{\bA_n}_n(\tau,\btheta_0,\gamma)$.

\begin{proof}[Proof of Lemma~\ref{lem:de_ie_close}]
   Fix $\delta>0$ to be chosen later. Since $\|\bA_n-\bB_n\| \le \delta$, Lemma~\ref{lem:graph_approx_zn} implies that 
   \begin{equation}\label{eq:fn_close}
       \max_{\tau, \gamma \in [-1,1]} \left \vert F^{\mathbf{A}_n}_n(\tau,\btheta_0,\gamma) - F^{\mathbf{B}_n}_n(\tau,\btheta_0,\gamma)  \right \vert  \le \delta.
   \end{equation}
We have 
    \begin{equation}
        \text{DE}^{\mathbf{u}}=2 \ee_{\bar\bT,\bar\bX}\frac{\partial}{\partial \tau} F^{\mathbf{u}}_n(\tau,\btheta_0,0)\Big\vert_{\tau=\tau_0},
    \end{equation} 
for $\mathbf{u} \in \{\bA_n,\mathbf{B}_n\}$. Since $F^{\bA_n}_n(\tau, \btheta,0)$ and $F^{\mathbf{B}_n}_n(\tau, \btheta,0)$ are convex functions in first coordinate, we obtain for any fixed $ \eta >0$,

\begin{equation*}
   \frac{\ee_{\bar\bT,\bar\bX}F^{\bA_n}_n(\tau_0,\btheta_0,0) -\ee_{\bar\bT,\bar\bX}F^{\bA_n}_n(\tau_0- \eta,\btheta_0,0)}{\eta}  \le \frac{\text{DE}^{\mathbf{A}_n}}{2} \le \frac{\ee_{\bar\bT,\bar\bX}F^{\bA_n}_n(\tau_0+ \eta,\btheta_0,0) -\ee_{\bar\bT,\bar\bX}F^{\bA_n}_n(\tau_0,\btheta_0,0)}{\eta}.
\end{equation*}
Similar bounds hold for $\mathbf{B}_n$ as well. therefore, we obtain

\begin{align*}
   &\frac{\text{DE}^{\mathbf{A}_n} -\text{DE}^{\mathbf{B}_n}}{2} \\
   &\le \frac{(\ee_{\bar\bT,\bar\bX}F^{\bA_n}_n(\tau_0+ \eta,\btheta_0,0) -\ee_{\bar\bT,\bar\bX}F^{\bA_n}_n(\tau_0,\btheta_0,0))- (\ee_{\bar\bT,\bar\bX}F^{\mathbf{B}_n}_n(\tau_0,\btheta_0,0) -\ee_{\bar\bT,\bar\bX}F^{\mathbf{B}_n}_n(\tau_0 -\eta,\btheta_0,0))}{\eta} \\ 
   &\le   \frac{(\ee_{\bar\bT,\bar\bX}F^{\mathbf{B}_n}_n(\tau_0+ \eta,\btheta_0,0) -\ee_{\bar\bT,\bar\bX}F^{\mathbf{B}_n}_n(\tau_0,\btheta_0,0))- (\ee_{\bar\bT,\bar\bX}F^{\mathbf{B}_n}_n(\tau_0,\btheta_0,0) -\ee_{\bar\bT,\bar\bX}F^{\mathbf{B}_n}_n(\tau_0 -\eta,\btheta_0,0))+2\delta}{\eta}\\
   &=\frac{\eta}{2} \ee_{\bar\bT,\bar\bX}\frac{\partial^2}{\partial^2 \tau} F^{\mathbf{B}_n}_n(\tau,\btheta_0,0)\Big\vert_{\tau=\tilde\tau} + \frac{2\delta}{\eta},
\end{align*}
for some $\tilde{\tau} \in [\tau_0-\eta,\tau_0+\eta]$. Further we have, 
\begin{align*}
    \ee_{\bar\bT,\bar\bX}\frac{\partial^2}{\partial^2 \tau} F^{\mathbf{B}_n}_n(\tau,\btheta_0,0)\Big\vert_{\tau=\tilde\tau}=  \ee_{\bar\bT,\bar\bX} \text{Var}_{\tilde \tau}\left(\frac{1}{n} \sum_{i=1}^{n} T_i Y_i\right) \le 1.
\end{align*}
Therefore, we obtain
\begin{align*}
    \text{DE}^{\mathbf{A}_n} -\text{DE}^{\mathbf{B}_n}  \le 2\eta+ \frac{4\delta}{ \eta}
\end{align*}
Choosing $\eta=\varepsilon/4$ and $\delta= \varepsilon^2/32$ yields $ \text{DE}^{\mathbf{A}_n} -\text{DE}^{\mathbf{B}_n}  \le \varepsilon$. Similarly, a lower bound can be obtained for $\text{DE}^{\mathbf{A}_n}- \text{DE}^{\mathbf{B}_n}$ proving the stability of direct effect. 

Turning to the proof of indirect effect, note that by Lemma~\ref{lem:graph_approx_zn} we have
\[\max_{\gamma \in [-1,1]} \left \vert F^{\mathbf{A}_n}_n(\tau_0,\btheta_0,\gamma) - F^{\mathbf{B}_n}_n(\tau_0,\btheta_0,\gamma)  \right \vert  \le \delta.\]
Using Lemma~\ref{lem:de_define}, we obtain,
\begin{align}
    \mathrm{IE}^{\mathbf{A}_n}= \ee_{\bar\bT,\bar\bX}\frac{\partial}{\partial \gamma} \tilde F^{\bA_n}(\tau_0,\btheta_0, \gamma)\Big\vert_{\gamma=0} +\ee_{-\mathbf{1},\bar\bX} \frac{\partial}{\partial \gamma} \tilde F^{\bA_n}_{n}(\tau_0,\btheta_0,\gamma)\Big\vert_{\gamma=0} -\mathrm{DE}^{\mathbf{A}_n}.
\end{align}
We invoke the same argument used above for direct effect for each of the summands individually. Since the first two summands are convex in $\gamma$, this complete the proof. 
\end{proof}

We next prove a stability lemma regarding the AMP iterates. To state our result, we need a technical lemma whose proof is deferred to the end of this section. We begin by metrizing weak convergence on the space of probability measures on $[0,1]$ with the metric
\begin{equation}
    d(\mu,\nu)= \int_{0}^{1}\Big|\mu[0,s]-\nu[0,s]\Big| ds.
\end{equation}
The following lemma provides continuity bounds on the solution of the Parisi PDE.
\begin{lemma}\label{lem:parisi_pde_cont}
    Consider two probability measures $\mu,\nu$ on $[0,1]$. Let $\Phi_{\mu}, \Phi_{\nu}$ be the solutions of the Parisi PDE corresponding to $\mu$ and $\nu$ respectively. There exists $C= C(\beta)>0$ such that
\begin{equation}\label{eq:pde_continuity}
    \max\left\{\|\Phi_{\mu}-\Phi_{\nu}\|_{\infty},\|\partial_x\Phi_{\mu}-\partial_x\Phi_{\nu}\|_{\infty},\|\partial_{xx}\Phi_{\mu}-\partial_{xx}\Phi_{\nu}\|_{\infty}\right\} \le C d(\mu, \nu).
\end{equation}
\end{lemma}
Our next lemma establishes the stability of the AMP algorithm. 
\begin{lemma}\label{lem:amp_stable}
   Let $\mathbf{h} = (h_i), \tilde{\mathbf{h}}= (\tilde{h}_i)$ be two i.i.d. random vectors such that $\|h_i-\tilde{h}_i\|_\infty \le \varepsilon$ almost surely. Let the corresponding AMP iterates given by Algorithm~\eqref{alg:amp} be denoted as $\mathbf{m}^{k}$ and $\tilde{\mathbf{m}}^{k}$ respectively. Define $\Delta_k:= \frac{1}{n} \|\mathbf{m}^{k}- \tilde{\mathbf{m}}^{k}\|^2$. Then for any $k \in \mathbb N  \cup \{0\}$, almost surely,
    \begin{equation}\label{eq:amp_Stable}
        \lim\limits_{\varepsilon \rightarrow 0^+}\lim\limits_{n \rightarrow \infty} \Delta_k =0.
    \end{equation}
\end{lemma}
\begin{proof}[Proof of Lemma~\ref{lem:amp_stable}]
The conclusion holds trivially for $k=0$. For $k \geq 1$ we proceed by induction. 
Recall the definition of Parisi PDE from~\eqref{eq:parisi_pde}. The Parisi functionals given magnetization $\bh,\tilde\bh$ are denoted by $P$, $\tilde P$ respectively, where
\begin{align*}
    P(\mu) = \mathbb{E}[\Phi_{\mu}(0,h)]- \frac{\beta^2}{2} \int_0^1 t \mu(t) dt,\qquad 
    \tilde P(\mu) = \mathbb{E}[\Phi_{\mu}(0,\tilde{h})]- \frac{\beta^2}{2} \int_0^1 t \mu(t) dt
\end{align*}
The above Parisi functionals have unique minima~\cite{auffinger2015parisi}, \cite{jagannath2016dynamic}, which will be denoted by $\mu^\star$ and $\tilde \mu^\star$ respectively. Define the functions that govern the AMP iterations by
\begin{equation}\label{eq:tilde_g_define}
    g(x) = \partial_x \Phi_{\mu^\star} (q,x), \qquad \tilde g(x) = \partial_x \Phi_{\tilde \mu^\star} (\tilde q,x),
\end{equation}
where $q=\inf (\rm{supp}(\mu^\star))$, $\tilde q= \inf (\rm{supp}(\tilde{\mu}^\star))$. Finally, following~\eqref{eq:amp_iterate}, define
\begin{equation}\label{eq:tilde_dk_define}
    d_k = \frac{1}{n} \sum_{i=1}^{n}\partial_{xx} \Phi_{\mu^\star}(q, x^k_i), \qquad \tilde d_k = \frac{1}{n} \sum_{i=1}^{n}\partial_{xx} \Phi_{\tilde \mu^\star}(\tilde q, \tilde x^k_i).
\end{equation}
Define $\Gamma_k:= \frac{1}{n} \| \bx^k - \tilde{\bx}^k\|^2$. 
Assume that the following statements hold simultaneously for $k \in \mathbb{N}$:
\begin{equation}\label{eq:induction}
    \lim\limits_{\varepsilon \rightarrow 0^+}\lim\limits_{n \rightarrow \infty} \Delta_k=0, \quad \lim\limits_{\varepsilon \rightarrow 0^+}\lim\limits_{n \rightarrow \infty} (d_k -\tilde d_k)^2=0, \quad \lim\limits_{\varepsilon \rightarrow 0^+}\lim\limits_{n \rightarrow \infty}  \Gamma_k=0.
\end{equation}
 Following Algorithm~\ref{alg:amp}, we obtain for $k+1$,
\begin{align}\label{eq:t1+t2}
   \Delta_{k+1}&= \frac{1}{n} \Big\|\mathbf{m}^{k+1}- \tilde{\mathbf{m}}^{k+1}\Big\|^2 \nonumber\\
   &= \frac{1}{n} \Big\| g(\bw^{k+1}+ \bh) - \tilde g( \tilde \bw^{k+1}+ \tilde \bh) \Big\|^2 \nonumber \\
    & \le\frac{2}{n} \|g(\bw^{k+1}+ \bh) -  g( \tilde \bw^{k+1}+ \tilde \bh) \|^2 + \frac{2}{n} \|g( \tilde \bw^{k+1}+ \tilde \bh) -  \tilde g( \tilde \bw^{k+1}+ \tilde \bh) \|^2 \nonumber \\
    & \le\frac{2}{n} \|g(\bw^{k+1}+ \bh) -  g( \tilde \bw^{k+1}+ \tilde \bh) \|^2 + 2 \|g-\tilde g\|^2_{\infty} \nonumber\\
    &=: \mathcal{T}_1 + \mathcal{T}_2.
\end{align}
We will bound the two terms above separately. To bound $\mathcal{T}_2$, note that for any measure $\mu$ on $[0,1]$, we have $|\partial_x \Phi_{\mu}(t,x)| \le 1$~\cite{auffinger2015parisi}. Therefore, we have for any measure $\mu$ on $[0,1]$,
\[
|P(\mu)-\tilde P(\mu)| \le \| H - \tilde{H}\|_{\infty} \le \varepsilon.
\]
As the functions $\mu \mapsto P(\cdot)$ and $\mu \mapsto \tilde{P}(\cdot)$ are continuous and strictly convex, there exists $\delta= \delta(\varepsilon)>0$ such that $d(\mu^\star,\tilde \mu^\star) \le \delta$, where $\delta \rightarrow 0$ as $\varepsilon \rightarrow 0$. By definition, we also have $|q-\tilde q| \le \delta$. Therefore, 
\begin{align*}
    \mathcal{T}_2 & \lesssim \sup_{x} \Big|\partial_x \Phi_{\mu^\star} (q,x)- \partial_x \Phi_{\tilde \mu^\star} (\tilde q,x) \Big| \\
    &\lesssim \sup_{x} \Big|\partial_x \Phi_{\mu^\star} (q,x)- \partial_x \Phi_{ \mu^\star} (\tilde q,x) \Big| + \sup_{x} \Big|\partial_x \Phi_{\mu^\star} (\tilde q,x)- \partial_x \Phi_{\tilde \mu^\star} (\tilde q,x) \Big| \\
    & \lesssim \|\partial_t \partial_x \Phi_{\mu^{\star}}\|_{\infty} |q-\tilde q|+ d(\mu^\star,\tilde{\mu}^{\star}) \lesssim \delta,
\end{align*}
where the third inequality uses~\eqref{eq:pde_continuity} and the final inequality uses the fact $\|\partial_t \partial_x \Phi_{\mu^\star}\|_{\infty}$ is bounded, which follows from~\cite[Theorem 4]{jagannath2016dynamic}. This provides the necessary upper bound on $\mathcal{T}_2$. Turning to the bound on $\mathcal{T}_1$, note that
\begin{equation}\label{eq:del_xx}
    \|\partial_x g(x)\|_{\infty}= \|\partial_{xx} \Phi_{\mu^\star}(q,x)\|_{\infty} \le 1
\end{equation}
by~\cite[Proposition 2]{jagannath2016dynamic}. Therefore,
\begin{align}\label{gamma_recursion}
    \mathcal{T}_1 & \le \frac{2}{n} \|\bx^{k+1}-\tilde\bx^{k+1}\|^2 = 2 \Gamma_{k+1} \nonumber \\
    &\le \frac{2}{n} \|(\bw^{k+1}+ \bh) - (\tilde \bw^{k+1}+ \tilde \bh)\|^2 \nonumber \\ 
    &\le \frac{4}{n} \| \bw^{k+1}-\tilde\bw^{k+1}\|^2+ 4 \|\bh -\tilde \bh\|^2_{\infty} \nonumber \\
    & \le \frac{8\beta^2}{n} \|\mathbf{G}_n(\bm^{k}- \tilde \bm^{k})\|^2+ \frac{8\beta^4}{n} \|\bm^{k} d_k - \tilde \bm^{k} \tilde d_k\|^2 +4 \varepsilon \nonumber \\
    & \lesssim \frac{1}{n}\Big\|\bm^{k} - \tilde \bm^{k}\Big\|^2 + \frac{1}{n}\|\bm^{k}\|^2 (d_k - \tilde d_k)^2+ \frac{1}{n} \Big\|\bm^{k} - \tilde \bm^{k}\Big\|^2 \tilde d^2_k + \varepsilon,
\end{align}
where we have used $\|\mathbf{G}_n\| \leq 3$ almost surely. 
We can bound the above display as follows: by~\cite[Proposition 2]{jagannath2016dynamic}, we have $\|\partial_x \Phi_{\mu^\star}\|_{\infty} \le 1$ and thus $\|\bm^{k}\|^2 \le n$. Using~\eqref{eq:del_xx}, we have $\tilde d^k \le 1$. Therefore, we obtain
\begin{align}\label{eq:t1_Step_1}
    \mathcal{T}_1 &\lesssim \frac{1}{n}\Big\|\bm^{k} - \tilde \bm^{k}\Big\|^2 + (d_k - \tilde d_k)^2+ \varepsilon.
\end{align}
To bound $(d_k - \tilde d_k)^2$,  note that
\begin{align}
    (d_k - \tilde d_k)^2 &\le  \left(\frac{1}{n} \sum_{i=1}^{n}( \partial_{xx} \Phi_{\mu^\star}(q, x^k_i)- \partial_{xx} \Phi_{\tilde \mu^\star}(\tilde q, \tilde x^k_i))\right)^2 \nonumber\\
    & \le \frac{1}{n} \sum_{i=1}^n( \partial_{xx} \Phi_{\mu^\star}(q, x^k_i)- \partial_{xx} \Phi_{\tilde \mu^\star}(\tilde q, \tilde x^k_i))^2 \nonumber\\
    & \le \frac{2}{n} \sum_{i=1}^n( \partial_{xx} \Phi_{\mu^\star}(q, x^k_i)- \partial_{xx} \Phi_{\tilde \mu^\star}(q, x^k_i))^2+ \frac{2}{n} \sum_{i=1}^n( \partial_{xx} \Phi_{\tilde \mu^\star}(q, x^k_i)- \partial_{xx} \Phi_{\tilde{\mu}^\star}(\tilde q, \tilde x^k_i))^2 \nonumber\\
    & \lesssim d(\mu^\star, \tilde\mu^\star) + \frac{2}{n} \sum_{i=1}^n( \partial_{xx} \Phi_{\tilde \mu^\star}(q, x^k_i)- \partial_{xx} \Phi_{\tilde{\mu}^\star}(\tilde q, \tilde x^k_i))^2 \nonumber\\
    &\lesssim d(\mu^\star, \tilde\mu^\star)+ |q-\tilde q| + \frac{1}{n}\|\bx^k- \tilde \bx^k\|^2 \lesssim \delta + \Gamma_k, \label{eq:dk_recursion}
\end{align}
where the fourth inequality uses~\eqref{eq:pde_continuity} and the fifth inequality uses~\cite[Theorem 4]{jagannath2016dynamic}. Since~\eqref{eq:induction} holds for $k$, we get by~\eqref{gamma_recursion}
\begin{equation}\label{eq:gamma_induction}
    \Gamma_{k+1} \lesssim \Delta_{k} + (d_k - \tilde d_k)^2 + \varepsilon,
\end{equation}
which converges to $0$ as $n\rightarrow \infty$ followed by $\varepsilon \rightarrow 0+$. This, along with~\eqref{eq:dk_recursion} implies 
\begin{equation}\label{eq:dk_induction}
    (d_{k+1} - \tilde d_{k+1})^2 \lesssim \delta + \Gamma_{k+1}
\end{equation}
 converges to $0$ as $n\rightarrow \infty$ followed by $\varepsilon \rightarrow 0+$. Finally,~\eqref{eq:t1+t2} implies that
 \begin{equation}\label{eq:delta_induction}
     \Delta_{k+1} \lesssim \delta+(d_k - \tilde d_k)^2 + \Delta_k+  \varepsilon,
 \end{equation}
 which also converges to $0$ as $n\rightarrow \infty$ followed by $\varepsilon \rightarrow 0+$. Therefore \eqref{eq:gamma_induction}, \eqref{eq:dk_induction} and \eqref{eq:delta_induction} prove the induction hypothesis~\eqref{eq:induction} for $(k+1)$. This completes the proof of the Lemma.
\end{proof}

Now, we provide the proof of Lemma~\ref{lem:parisi_pde_cont}.

\begin{proof}[\textbf{Proof of Lemma~\eqref{lem:parisi_pde_cont}}]
    Using \cite[Lemma 14]{jagannath2016dynamic} we immediately obtain that $\max\left\{\|\Phi_{\mu}-\Phi_{\nu}\|_{\infty},\|\partial_x\Phi_{\mu}-\partial_x\Phi_{\nu}\|_{\infty}\right\} \le C d(\mu, \nu)$. Hence we only need to prove $\|\partial_{xx}\Phi_{\mu}-\partial_{xx}\Phi_{\nu}\|_{\infty}\le C d(\mu, \nu)$. Let $u:= \Phi_\mu$, $v:= \Phi_{\nu}$ be weak solutions to the Parisi PDE. Then $w= u-v$ is a weak solution to the following PDE:
\begin{align}\label{eq:w_pde}
    &w_t+ \frac{\beta^2}{2} \left(w_{xx}+ \mu[0,t] (u_x+v_x) w_x+ (\mu[0,t]-\nu[0,t]) v^2_x\right) =0, \quad (t,x) \in (0,1)\times \mathbb{R}, \nonumber \\
    &w(1,x)=0.
\end{align}
We denote partial derivative w.r.t. $t,x$ by subscripts respectively. We can write down the expression for $w, w_x, w_{xx}$ by solving the following SDE:
\begin{equation*}
    dX_t= \beta^2 \mu[0,t] \frac{u_x+v_x}{2}(t,X_t) dt + \beta dW_t,
\end{equation*}
where $W_t$ is  standard Brownian motion. Note that this SDE has a strong solution as $u_x, v_x$ are Lipschitz in $x$ uniformly in $t$; additionally, $u_x, v_x$ are also bounded in $t$. 
Differentiating~\eqref{eq:w_pde}, we obtain by continuity of $w$
\begin{equation*}
    w_{tx}+ \frac{\beta^2}{2} \left(w_{xxx}+  \mu[0,t] (u_x+v_x) w_{xx} +\mu[0,t] (u_{xx}+v_{xx}) w_{x} + 2(\mu[0,t]-\nu[0,t]) v_x v_{xx} \right)=0
\end{equation*}
Using the shorthand notation $\alpha= w_{xx}$, we have by differentiating the above display w.r.t. $x$ to obtain
\begin{align}\label{eq:w_xx}
    \alpha_t&+ \frac{\beta^2}{2} \Bigg( \alpha_{xx}+ \mu[0,t] (u_x+v_x) \alpha_x  +2 \mu[0,t] (u_{xx} +v_{xx}) \alpha \nonumber \\
    &+ \mu[0,t] (u_{xxx}+v_{xxx})w_x+ 2(\mu[0,t]-\nu[0,t]) (v_x v_{xxx}+ v^2_{xx})\Bigg)=0
\end{align}
Therefore, by an application of~\cite[Proposition 22]{jagannath2016dynamic}, we have $\alpha$ has the following representation
\begin{equation}\label{eq:alpha_closed}
    \alpha= \mathbb{E}\left[\frac{\beta^2}{2}\int\limits_{t}^1 I(t,s) \left\{\mu[0,s] (u_{xxx}+v_{xxx})w_x+ 2(\mu[0,s]-\nu[0,s]) (v_x v_{xxx}+ v^2_{xx}) \right\}(s,X_s) ds\bigg| X_t=x \right].
\end{equation}
Here, $I(t,s)$ has the closed-form expression
\begin{equation*}
    I(t,s)= \exp\left(\int_{t}^{s} \beta^2 \mu[0,\tau] (u_{xx}+v_{xx})(\tau,X_{\tau}) d\tau\right)
\end{equation*}
Since $u_{xx}, v_{xx}$ are uniformly bounded, we have $|I(t,s)| \le C$ for some $C>0$. Similarly $u_{xxx}, v_{xxx}$ are uniformly bounded by~\cite[Theorem 4]{jagannath2016dynamic}. Hence we obtain from~\eqref{eq:alpha_closed} that
\begin{equation*}
    \|u_{xx}-v_{xx}\|_{\infty} = \|\alpha\|_{\infty} \lesssim \|w_x\|_{\infty} + d(\mu,\nu).
\end{equation*}
Since we have already established that $\|w_x\|_{\infty} \lesssim d(\mu,\nu)$, we get the desired conclusion.
\end{proof}

\subsection*{Discretizing the covariate support:} Next, we demonstrate how Algorithm~\ref{algo:de_method} and the corresponding Theorem~\ref{thm:de_consistency} change when the variables $\bX_i$ have general compact support instead of finite support. Specifically, we modify Step 1 of Algorithm~\ref{algo:de_method} as follows:
\begin{enumerate}
    \item Fix $m \in \mathbb{N}$. Generate $\bar \bT= (\bar T_1,\ldots, \bar T_n)$ from the uniform probability distribution on $\{\pm 1\}^n$. Generate $\bar \bX  = (\bar \bX_1,\cdots, \bar \bX_n) $ i.i.d. $\mathbb{P}_X$ independent of $\bar \bT$. Define $\mathcal{H}_m= \{-1, -1+\frac{1}{2^m}, \ldots, 1-\frac{1}{2^m},1 \}^d$. Set $\tilde \bX_i$ as the closest point of $\bar \bX_i$ in the set $\mathcal{H}_m$.
\end{enumerate}
We then proceed with the remaining steps of Algorithm~\ref{algo:de_method}, replacing $\bar \bX_i$ with $\tilde \bX_i$. Denote the resulting estimators $\widetilde {\mathrm{DE}}_{(r,\delta)}$ and $\widetilde {\mathrm{IE}}_{(r,\delta)}$ respectively. We have the following consequence.

\begin{lemma}\label{lem:general_h}
Consider the setup of Theorem~\ref{thm:de_consistency}. Fix $\varepsilon >0$. Then there exists $\delta :=\delta(\varepsilon) >0$ and $m:= m(\varepsilon) \in \mathbb{N}$ such that the estimates $\widetilde{\mathrm{DE}}_{(r,\delta)}$ and $\widetilde{\mathrm{IE}}_{(r,\delta)}$ satisfy
    \begin{align}
        \Big| \ee_{\bar\bT,\tilde\bX}[\widetilde{\mathrm{DE}}_{(r,\delta)}] - \mathrm{DE} \Big| <\varepsilon , \qquad 
        \Big| \ee_{\bar\bT,\tilde\bX}[\widetilde{\mathrm{IE}}_{(r,\delta)}] - \mathrm{IE} \Big| <\varepsilon.\,\,\,  \nonumber 
    \end{align}
\end{lemma}
\begin{proof}
Let $\mathrm{DE}(\mathbb{P})$ and $\mathrm{IE}(\mathbb{P})$ denote the direct and indirect causal effect respectively under the covariate distribution $\mathbb{P}$. Define the distribution of $\tilde \bX_i$ by $\tilde{\mathbb{P}}_X$. We state the following result~\cite[Theorem D.1]{bhattacharya2024causal}: There
exists $C:=C(d, \|\boldsymbol{\theta}_0\|)>0$ such that 
\begin{align}
    |\mathrm{DE}(\mathbb{P}_X) - \mathrm{DE}(\tilde{\mathbb{P}}_X)| \leq C \sqrt{d_{W_2}(\mathbb{P}_X, \tilde{\mathbb{P}}_X)}, \,\,\,\, |\mathrm{IE}(\mathbb{P}_X) - \mathrm{IE}(\tilde{\mathbb{P}}_X)| \leq C \sqrt{d_{W_2}(\mathbb{P}_X, \tilde{\mathbb{P}}_X)}, \nonumber 
\end{align}
where $d_{W_2}$ denotes the $2$-Wasserstein distance between two probability distributions. Therefore, we need an upper bound of $d_{W_2}(\mathbb{P}_X, \tilde{\mathbb{P}}_X)$. For a set $K$, define $P_K$ to be the projection onto $K$. Hence $P_{\mathcal{H}_m}(\bar \bX_i)= \tilde \bX_i$. Note that, $\sup_{\bx \in [-1,1]}\| \bx - P_{\mathcal{H}_m}(\bx)\| \le \frac{\sqrt{d}}{2^m}$. Since $(\bX,\tilde \bX)$ is a coupling of $\mathbb{P}_X$ and $\tilde{\mathbb{P}}_X$, we have
\[d_{W_2}(\mathbb{P}_X, \tilde{\mathbb{P}}_X) \le \frac{\sqrt{d}}{2^m}.\]
Hence, there exists $m$ such that 
\[ |\mathrm{DE}(\mathbb{P}_X) - \mathrm{DE}(\tilde{\mathbb{P}}_X)| \le \frac{\varepsilon}{2}, \qquad |\mathrm{IE}(\mathbb{P}_X) - \mathrm{IE}(\tilde{\mathbb{P}}_X)| \le \frac{\varepsilon}{2}.\]
Further, using Theorem~\ref{thm:de_consistency}, we have for $\delta$ small enough,
\begin{equation*}
     \Big| \ee_{\bar\bT,\tilde\bX}[\widetilde{\mathrm{DE}}_{(r,\delta)}] - \mathrm{DE}(\tilde{\mathbb{P}}_X) \Big| <\frac{\varepsilon}{2}, \qquad 
        \Big| \ee_{\bar\bT,\tilde\bX}[\widetilde{\mathrm{IE}}_{(r,\delta)}] - \mathrm{IE}(\tilde{\mathbb{P}}_X) \Big| <\frac{\varepsilon}{2}.
\end{equation*}
Combining the above two displays, we obtain the desired conclusion.
\end{proof}

\subsection*{Proof of Lemma~\ref{lem:block_approx}}

We need the following definition to prove the result.
\begin{definition}\label{def:r-reg-define}
    An $n \times n$ matrix $\mathbf{B}$ is called $r$-regular for some $r \in \mathbb{N}$ if $\max_{i,j}|\mathbf{B}_{i,j}| \le 1$ and there exists sets $Q_1,\ldots, Q_r,P_1,\ldots, P_r \subseteq [n]$ and $c_1,\ldots, c_r \in \mathbb R$ such that $\mathbf{B}= \sum_{k=1}^{r} c_k \mathbf{1}_{Q_k}\mathbf{1}^\top_{P_k}$. 
\end{definition}
Fix $\varepsilon >0$. Since the interaction matrix $\bA_n$ satisfies $\max_{i,j} |n \bA_n(i,j)|\le 1$, we can use \cite[Theorem 2.1]{fox2019fast} to find $r= r(\varepsilon)>0$ and an $r$-regular matrix $\widetilde{\bA}_n$ such that $\|n\bA_n-n \widetilde{\bA}_n\| \le \varepsilon n$. Moreover, the matrix $\widetilde{\bA}_n$ can be computed in $\varepsilon^{-O(1)}n^2$ time.

Following Definition~\ref{def:r-reg-define}, we denote $\widetilde{\bA}_n= \sum_{k=1}^{r} c_k \mathbf{1}_{Q_k}\mathbf{1}^\top_{P_k}$. Next, we obtain a partition of vertices by finding common refinement of all $Q_k,P_l$'s in time $O(nr)$. This is achieved by going through the vertices of $\widetilde{\bA}_n$, and checking, for each vertex, which parts it does and does not belong to. The vertex partition has size at most $2^r$. With an abuse of notation, we call such refinement also as $\widetilde{\bA}_n$ and assume the partition size $2^r$. Call the partition as $\{U_1,\ldots U_{2^r}\}$. Hence, we have,
\begin{equation}\label{eq:final_block}
    \widetilde{\bA}_n= \sum_{k,l=1}^{2^r} c_{kl} \mathbf{1}_{U_k}\mathbf{1}^\top_{U_l}
\end{equation}
By definition, we have $\|\bA_n-\widetilde{\bA}_n\|= O( \varepsilon n)$. This completes the proof of the Lemma.

\subsection*{Proof of Lemma~\ref{lem:gibbs_simplify}}
We begin by noting that using~\eqref{eq:define_part_sum_1}, we have 
\begin{equation}\label{eq:u_sum}
    \sum_{\ell \in U_k} y_\ell= \sum_{a=1}^m \Big(V_{a,k,+}+V_{a,k,-}\Big).
\end{equation}
Recall that $[n]= \cup_{k=0}^{2^r} U_{k}$. Therefore, using \eqref{eq:u_sum}, 
\begin{align}\label{eq:t,x_sum}
    \mathbf{y}^\top (\tau_0 \mathbf{t}+ \mathbf{x} \btheta_0) &=  \sum_{k=0}^{2^r} \sum_{\ell \in U_k} y_\ell (\tau_0 t_\ell+ \bx^\top_\ell \btheta_0)=  \sum_{a=1}^m \sum_{k=0}^{2^r} \sum_{\ell \in S_k \cap U_k} y_\ell (\tau_0 t_\ell+ h_a^\top \btheta_0) \nonumber \\
    &= \sum_{a=1}^m \sum_{k=0}^{2^r} \left( V_{a,k,+} (\tau_0+ h^\top_a \btheta_0)+ V_{a,k,-} ( - \tau_0+ h^\top_a \btheta_0)\right),
\end{align}
since $\bx_i= h_a$ if $i \in S_a$. Combining \eqref{eq:u_sum} and \eqref{eq:t,x_sum}, the Hamiltonian corresponding to the Gibbs measure~\eqref{eq:define_gibbs} simplifies as 
\begin{align}
    &\frac{1}{2}\,\mathbf{y}^\top \widetilde{\bA}_n \mathbf{y}+ \mathbf{y}^\top (\tau_0 \mathbf{t}+ \mathbf{x} \btheta_0) =\sum_{k,l=1}^{2^r} c_{kl} \by^\top\mathbf{1}_{U_k}\mathbf{1}^\top_{U_l} \by+ \mathbf{y}^\top (\tau_0 \mathbf{t}+ \mathbf{x} \btheta_0) \nonumber\\
    &= \sum_{k,l=1}^{2^r} c_{kl} \Big(\sum_{i \in U_k }y_k\Big)\Big(\sum_{j \in U_l }y_j\Big)+ \sum_{k=0}^{2^r} \sum_{\ell \in U_k} y_\ell (\tau_0 t_\ell+ \bx^\top_\ell \btheta_0) \nonumber\\
    &= \sum_{k,l=1}^{2^r} c_{kl} \Big(\sum_{a=1}^m\Big(V_{a,k,+}+V_{a,k,-}\Big)\Big)\Big(\sum_{a=1}^m\Big(V_{a,k,+}+V_{a,k,-}\Big)\Big) \nonumber\\
    &+\sum_{a=1}^m \sum_{k=0}^{2^r} \left( V_{a,k,+} (\tau_0+ h^\top_a \btheta_0)+ V_{a,k,-} ( - \tau_0+ h^\top_a \btheta_0)\right).\label{eq:gibbs_simplify}
\end{align}

This implies, the conditional distribution \eqref{eq:define_gibbs} can be written as :
\begin{align*}
     &f_{(r,\varepsilon)}(V_{a,k,+} = v_{a,k,+}, V_{a,k,-} = v_{a,k,-}, a\in[m], 0\leq k \leq 2^r| \overline{\bT}, \overline{\bX}) \nonumber \\ 
     &\propto \prod_{a=1}^{m} \prod_{k=1}^{2^r} \binom{| \mathcal{A}_{a,k,+}|}{\frac{| \mathcal{A}_{a,k,+}|+ v_{a,k,+}}{2}} \binom{| \mathcal{A}_{a,k,-}|}{\frac{| \mathcal{A}_{a,k,-}|+ v_{a,k,-}}{2}} \times \nonumber\\
     &\exp \Bigg( \sum_{k,l=1}^{2^r} c_{kl} \Big(\sum_{a=1}^m\Big(V_{a,k,+}+V_{a,k,-}\Big)\Big)\Big(\sum_{a=1}^m\Big(V_{a,k,+}+V_{a,k,-}\Big)\Big)  \nonumber\\
    &+\sum_{a=1}^m \sum_{k=0}^{2^r} \left( V_{a,k,+} (\tau_0+ h^\top_a \btheta_0)+ V_{a,k,-} ( - \tau_0+ h^\top_a \btheta_0)\right) \Bigg).
\end{align*}
This simplification above implies the Gibbs measure \eqref{eq:define_gibbs} is a probability measure on the real numbers $V_{a,k,+}, V_{a,k-}$, $a \in [m]$, $k \in 0 \cup [2^r]$. Note that this reduces the number of indices from $n$ to $2m(2^r+1)$ which does not grow with the number of vertices. Therefore it is possible to sample $V_{a,k,+}, V_{a,k-}$'s from \eqref{eq:define_gibbs} exactly as long as $\widetilde{\bA}_n$ is a block matrix. Note that, the normalizing constant of $f$ can be computed in $O(n^{2m(2^r+1)})$ time.

\end{document}